\documentclass[11pt, a4paper]{article}

\usepackage[english]{babel}
\usepackage[T1]{fontenc}
\usepackage[utf8]{inputenc}
\usepackage[
  margin=2.5cm,
  includefoot,
  footskip=30pt,
]{geometry}
\usepackage{lmodern}
\usepackage{amsmath}
\usepackage{amssymb}
\usepackage{amsthm}
\usepackage{booktabs}
\usepackage{array}
\usepackage{graphicx}
\usepackage{cite}
\usepackage[dvipsnames]{xcolor}
\usepackage{tikz-cd}
\usetikzlibrary{decorations.markings,angles}
\usepackage{stmaryrd} \SetSymbolFont{stmry}{bold}{U}{stmry}{m}{n}
\usepackage{mathrsfs}
\usepackage{extarrows}
\usepackage{mathtools}
\usepackage{tensor}
\usepackage{enumerate}
\usepackage{adjustbox}
\usepackage[mathcal]{eucal}
\usepackage{hyperref}
\hypersetup{
	colorlinks,
	urlcolor={red!55!black},
	citecolor={green!60!black},
	linkcolor={red!55!black}
}
\usepackage{bm}
\usepackage{tabularx}
\usepackage{cleveref}
\usepackage[strict]{changepage}
\usepackage{tensor}

\mathchardef\mhyphen="2D
\def\on{\operatorname}

\makeatletter
\providecommand{\leftsquigarrow}{%
  \mathrel{\mathpalette\reflect@squig\relax}%
}
\newcommand{\reflect@squig}[2]{%
  \reflectbox{$\m@th#1\rightsquigarrow$}%
}
\makeatother

\newcommand\noloc{%
  \nobreak
  \mspace{6mu plus 1mu}
  {:}
  \nonscript\mkern-\thinmuskip
  \mathpunct{}
  \mspace{2mu}
}

\newcommand{\trivalent}{{\bf G}}
\newcommand{\idtr}{trivalent spanning graph}

\newcommand{\I}{I}
\newcommand{\A}{\mathcal{A}}

\newcommand{\C}{\mathcal{C}}
\newcommand{\D}{\mathcal{D}}

\definecolor{ao}{rgb}{0.0, 0.5, 0.0}

\newtheorem{theorem}{Theorem}[section]
\newtheorem{lemma}[theorem]{Lemma}

\newtheorem{proposition}[theorem]{Proposition}
\newtheorem{corollary}[theorem]{Corollary}

\theoremstyle{definition}
\newtheorem{construction}[theorem]{Construction}
\newtheorem{definition}[theorem]{Definition}

\newtheorem{remark}[theorem]{Remark}
\newtheorem{example}[theorem]{Example}

\title{Cluster theory of topological Fukaya categories}
\author{Merlin Christ}
\date{\today}

\begin{document}
\maketitle

\abstract{We establish a novel relation between the cluster categories associated with marked surfaces and the topological Fukaya categories of the surfaces. We consider a generalization of the triangulated cluster category of the surface by a $2$-Calabi--Yau extriangulated/exact $\infty$-category, which arises via Amiot's construction from the relative Ginzburg algebra of the triangulated surface. This category is shown to be equivalent to the $1$-periodic version of the topological Fukaya category of the marked surface, as well as to Wu's Higgs category. We classify the cluster tilting objects in this extriangulated cluster category and describe a cluster character to the upper cluster algebra of the marked surface with coefficients in the boundary arcs.

We furthermore give a general construction of $2$-Calabi--Yau Frobenius extriangulated structures/exact $\infty$-structures on stable $\infty$-categories equipped with a relative right $2$-Calabi--Yau structure in the sense of Brav--Dyckerhoff, that may be of independent interest.}

\tableofcontents

\section{Introduction}

Cluster algebras are a class of commutative algebras equipped with generators exhibiting a remarkable combinatorial structure introduced by Fomin--Zelevinski \cite{FZ02}. These generators are referred to as clusters and can be subjected to a process called mutation, which produces other clusters. It was soon observed that cluster algebras admit a rich theory of categorification and this has lead to many fruitful interactions with the representation theory of algebras. A particular class of (additive) categorifications of cluster algebras are formed by so called cluster categories, which are triangulated $2$-Calabi-Yau categories. Cluster categories of acyclic quivers were introduced in \cite{BMRRT06} as certain orbit categories. A generalization to arbitrary quivers with potential was achieved by Amiot \cite{Ami09}. To apply her construction, one can start with a quiver with potential $(Q,W)$ and consider the perfect derived category of the $3$-Calabi-Yau Ginzburg dg-algebra $\mathscr{G}(Q,W)$. Amiot's generalized cluster category $\mathcal{C}_{Q,W}$ assigned to $\mathscr{G}(Q,W)$ is then defined as the Verdier quotient 
\[ \mathcal{C}_{Q,W}\coloneqq \mathcal{D}^{\on{perf}}(\mathscr{G}(Q,W))/\mathcal{D}^{\on{fin}}(\mathscr{G}(Q,W))\]
of the perfect derived category by the finite derived category. 

Recently, generalizations of Ginzburg algebras have emerged \cite{Wu21,Chr22}, called relative Ginzburg algebras, which can also be seen as examples of deformed relative Calabi-Yau completions \cite{Yeu16}. An interesting class of relative Ginzburg algebras, studied by the author in \cite{Chr22,Chr21b}, arises from marked surfaces equipped with an ideal triangulation. In this paper, we suppose that the marked surfaces has no marked points in the interior, i.e.~no punctures. We denote by $\mathscr{G}_\mathcal{T}$ the relative Ginzburg algebra associated with a marked surface ${\bf S}$ equipped with an ideal triangulation $\mathcal{T}$. 

In this article, we study the category 
\begin{equation}\label{eq:cqout}
 \mathcal{C}_{\bf S}\coloneqq \mathcal{D}^{\on{perf}}(\mathscr{G}_\mathcal{T})/\mathcal{D}^{\on{fin}}(\mathscr{G}_\mathcal{T})
 \end{equation}
arising from applying Amiot's quotient formula for the generalized cluster category to the 
relative Ginzburg algebra $\mathscr{G}_\mathcal{T}$. Our results on the category $\C_{\bf S}$ belong to two subjects. Firstly, we show that $\C_{\bf S}$ yields an additive categorification of the cluster algebra of the marked surface with coefficients in the boundary arcs, and that as an exact $\infty$-category, $\C_{\bf S}$ is equivalent to the Higgs category of \cite{Wu21}. This implies that the standard triangulated cluster category associated with the marked surface is the stable category, i.e.~a certain localization, of the homotopy category of $\C_{\bf S}$. The class of cluster algebras categorified describe certain coordinates on the decorated Teichm\"uller space of the surface called lambda lengths. These cluster algebras were studied in detail in Fomin--Thurston's beautiful paper \cite{FT12} and correspond to the $\on{SL}_2$-case in the work of Fock--Goncharov \cite{FG06}. The coefficients of the cluster algebras are given by the so-called boundary arcs.

The second perspective from which we study $\C_{\bf S}$ is that of topological Fukaya categories, or more generally from the perspective of perverse schobers, meaning categorified perverse sheaves in the sense of Kapranov-Schechtman \cite{KS14}. The new result is that the category $\C_{\bf S}$ arises as the global sections of such a perverse schober on ${\bf S}$, and furthermore that $\C_{\bf S}$ is equivalent to the topological Fukaya category of the marked surface ${\bf S}$ valued in the derived category of $1$-periodic chain complexes, or the $1$-periodic topological Fukaya category for short. In total, our results thus establish a new link between the two subjects of perverse schobers or topological Fukaya categories as well as cluster algebra categorification. 

To describe perverse schobers on marked surfaces, we use the framework of perverse schobers parametrized by a ribbon graph of \cite{Chr22}. The derived category $\mathcal{D}(\mathscr{G}_\mathcal{T})$ of the relative Ginzburg algebra was shown in \cite{Chr22} to arise as the global sections of such a perverse schober $\mathcal{F}_\mathcal{T}$. We show that the passage to the quotient category $\C_{\bf S}$ can be realized on the level of perverse schobers in terms of the passage to a quotient $\mathcal{F}_\mathcal{T}^{\on{clst}}$ of $\mathcal{F}_\mathcal{T}$. We note that this is a non-trivial property of the perverse schober $\mathcal{F}_\mathcal{T}$, which fails for variants of $\mathcal{F}_\mathcal{T}$, such as the perverse schober whose global sections describe the derived category of the non-relative Ginzburg algebra associated with the triangulated surface. A local computation shows that $\mathcal{F}_\mathcal{T}^{\on{clst}}$ is a perverse schober of a particularly simple type, referred to as locally constant, which is the type of (co)sheaf considered in \cite{DK18,DK15,HKK17} for the construction of the topological Fukaya category.

The perverse schober description allows us to study $\mathcal{C}_{\bf S}$ using powerful and flexible local-to-global arguments. Results include a classification of all indecomposable objects in $\mathcal{C}_{\bf S}$ in terms of suitable curves in ${\bf S}$ and a description of the Homs in terms of intersections of the curves. This classification can be considered as a $1$-periodic version of the classification of indecomposable objects in the derived category of a gentle algebra given for instance in \cite{HKK17,OPS18}.

Before describing the results of this article in more detail, we finally comment on how $\C_{\bf S}$ categorifies the surface cluster algebra with coefficients. We equip the stable $\infty$-category $\C_{\bf S}$ with the structure of a Frobenius exact $\infty$-category. This structure gives rise to a Frobenius $2$-Calabi--Yau extriangulated structure on the homotopy $1$-category $\on{ho}\mathcal{C}_{\bf S}$. The exact $\infty$-structure is obtained from pulling back a split-exact structure along a functor carrying a right $2$-Calabi--Yau structure in the sense of Brav--Dyckerhoff \cite{BD19}. The exact extensions are described by a certain subfunctor $\on{Ext}^{1,\on{CY}}_{\mathcal{C}_{\bf S}}(\mhyphen,\mhyphen)\subset \on{Ext}^1_{\mathcal{C}_{\bf S}}(\mhyphen,\mhyphen)$ equipped with a $2$-Calabi--Yau duality 
\[ \on{Ext}^{1,\on{CY}}_{\mathcal{C}_{\bf S}}(X,Y)\simeq  \on{Ext}^{1,\on{CY}}_{\mathcal{C}_{\bf S}}(Y,X)^*\,,\]
which becomes the $2$-Calabi--Yau duality in the extriangulated homotopy category. We give a bijective correspondence between the clusters of the cluster algebra and the cluster tilting objects in the extriangulated category $\on{ho}\mathcal{C}_{\bf S}$. We further describe a cluster character to the commutative Skein algebra of the surface, which embeds into the upper cluster algebra of the surface by work of Muller \cite{Mul16}.

The remainder of this section is structured as follows. In the preparatory \Cref{sec1.1}, we review the additive categorification of cluster algebras (with and without coefficients) in terms of cluster categories. In \Cref{sec1.2}, we introduce the $1$-periodic topological Fukaya category. In \Cref{sec1.3}, we explain how the passage to the generalized cluster category $\mathcal{C}_{\bf S}$ can be understood in the framework of perverse schobers. We proceed in \Cref{sec1.4} by discussing our results about ${\C}_{\bf S}$ concerning the categorification of a cluster algebra with coefficients arising from ${\bf S}$. In the final \Cref{introsubsec:otherclustercats}, we discuss the relation between the generalized cluster category $\C_{\bf S}$, the standard triangulated cluster categories associated with marked surfaces, as well as the Higgs category.

\subsection{Background on cluster categories}\label{sec1.1}

To (additively) categorify a cluster algebra $A$ without coefficients, one asks for a triangulated category $\C$ with the following two properties: firstly, $C$ is required to be triangulated $2$-Calabi--Yau, meaning that there exists a bifunctorial isomorphism
\[ \on{Ext}^i_{\C}(X,Y)\simeq \on{Ext}^{2-i}_{\C}(Y,X)^\ast\,.\]
Secondly, one asks that $C$ is equipped with cluster tilting objects, meaning objects $T\in C$ satisfying that
\begin{itemize}
\item $T$ is rigid, i.e.~$\on{Ext}^1_\C(T,T)\simeq 0$, and
\item every object $X\in \C$ is part of a distinguished triangle $T_1\to T_0\to X$, with $T_0,T_1$ given by finite direct sums of direct summands of $T$.
\end{itemize} 
Cluster tilting objects in triangulated categories admit a well behaved theory of mutations, see \cite{IY08}, that categorifies the mutations of clusters in the cluster algebra. For an additive categorification of the cluster algebra $A$, one typically asks that the cluster tilting objects in $\C$ are in bijection with the clusters of the cluster algebra. 

Given any $2$-Calabi--Yau triangulated category $\C$ with a cluster tilting object, there is an associated cluster character, see \cite{Pal08}. A cluster character on $\C$ with values in a commutative ring $A$ is a function $\chi\colon\on{ob}(\C)\rightarrow A$ on the set of isomorphism classes of objects in $\C$, satisfying the two properties we now describe. Note that a typical choice of $A$ would be a cluster algebra. Firstly, the cluster character is a kind of exponential map, in the sense that $\chi(X\oplus Y)= X\cdot Y$ for any two $X,Y\in \C$. We note that this is the main difference between an additive categorification and another approach called monoidal categorification of cluster algebras; in the latter the multiplication of the cluster algebra is realized in terms of a symmetric monoidal product (instead of the direct sum). To connect to the cluster algebra, the cluster character has to satisfy a further relation, called the cluster multiplication formula. It states that if $\on{dim}_k\on{Ext}^1_\C(Y,X)= \on{dim}_k\on{Ext}^1_\C(X,Y)=1$, and $X\rightarrow B\rightarrow Y$ and $Y\rightarrow B'\rightarrow X$ are the corresponding non-split extensions, then 
\[ \chi(X)\cdot \chi(Y)= \chi(B)+\chi(B')\,.\]
This allows to recover the cluster exchange relations in the cluster algebra in terms of the mutation of cluster tilting objects in $\C$ and the cluster character. 

For the categorification of cluster algebras \textit{with coefficients}, meaning that a subset of the cluster variables is frozen, one can consider Frobenius exact categories, or more generally Frobenius extriangulated categories \cite{NP19}. The definition of cluster tilting object and cluster character translate to this setting by replacing all extensions groups by the groups of exact extensions in the extriangulated category. The indecomposable injective-projective objects in such an extriangulated category play the role of the frozen cluster variables. They appear as direct summands of every cluster tilting object and cannot be mutated at.   

Triangulated and exact cluster categories can be constructed in various ways. As proved in \cite{KL23}, any $2$-Calabi--Yau triangulated category with a cluster tilting object satisfying typical properties arises via Amiot's construction \cite{Ami09}, this is known as Amiot's conjecture. We briefly describe Amiot's construction in the next paragraph. A generalization of this construction to the setting of cluster algebras with coefficients is the $2$-Calabi--Yau extriangulated Higgs category \cite{Wu21}, which we briefly describe in \Cref{introsubsec:otherclustercats} and consider in more detail in \Cref{subsec:Higgs}.

As input for Amiot's construction serves the Ginzburg algebra $\mathscr{G}(Q,W)$ associated with a quiver with potential $(Q,W)$, satisfying that the Jacobian algebra $\on{H}_0(\mathscr{G}(Q,W))$ is finite dimensional. The fact that $\D(\mathscr{G}(Q,W))$ is smooth and admits a left $3$-Calabi--Yau structure imply that the derived category of finite dimensional modules $\mathcal{D}^{\on{fin}}(\mathscr{G}(Q,W))$ is contained in the perfect derived category $\mathcal{D}^{\on{perf}}(\mathscr{G}(Q,W))$. The generalized cluster category of $\mathscr{G}(Q,W)$ can thus be defined as the Verdier quotient 
\begin{equation}\label{clcateq}
\mathcal{C}(Q,W)=\mathcal{D}^{\on{perf}}(\mathscr{G}(Q,W))/\mathcal{D}^{\on{fin}}(\mathscr{G}(Q,W))\,.
\end{equation}
The triangulated category $\C(Q,W)$ is indeed triangulated $2$-Calabi--Yau and the image of $\mathscr{G}(Q,W)$ in $\C(Q,W)$ defines a cluster tilting object, see \cite{Ami09}.

\subsection{1-periodic topological Fukaya categories}\label{sec1.2}

We fix an oriented compact surface ${\bf S}$ with non-empty boundary and a set $M\subset \partial{\bf S}$ of marked points on the boundary, such that each boundary component contains at least one marked point. We call ${\bf S}$ a marked surface. For example, the simple marked surfaces are $n$-gons and marked annuli. We further choose an ideal triangulation of ${\bf S}$. As will be explained in \Cref{sec1.3}, the generalized cluster category associated to the relative Ginzburg algebra arising from the ideal triangulation turns out to be equivalent to the $1$-periodic topological Fukaya category of ${\bf S}$. In this section, we describe the construction and properties of this topological Fukaya category. 

The $2$-periodic topological Fukaya category of the oriented surface ${\bf S}$ can be seen as describing the $2$-periodic partially wrapped Fukaya category of the symplectic manifold ${\bf S}$ with stops at the marked points. It is however arises via a purely topological and combinatorial construction as the global sections of a (co)sheaf of stable $\infty$-categories on a ribbon graph \cite{DK18,DK15,HKK17}, hence its name. As carefully explained in \cite{DK18}, such a cosheaf can be obtained from any $2$-periodic stable $\infty$-category $\D$, by evaluating a cyclic $2$-Segal object obtained from the Waldhausen $S_\bullet$-construction of $\D$, and the resulting $\D$-valued topological Fukaya category is an invariant of the oriented marked surface ${\bf S}$. More generally, if one is given an arbitrary stable $\infty$-category, one must further equip ${\bf S}$ with a framing (or more generally a line field) to define the $\D$-valued topological Fukaya category. 

In this paper, we mainly consider the choice of $2$-periodic stable $\infty$-category $\mathcal{D}=\mathcal{D}(k[t_1^\pm])$ given by the derived $\infty$-category of the graded Laurent algebra with generator in degree $|t_1|=1$ over an algebraically closed field $k$. This $\infty$-category can also be described as the derived $\infty$-category of $1$-periodic chain complexes, meaning chain complexes which are isomorphic to their shift. We call the topological Fukaya categories arsing from $\mathcal{D}(k[t_1^\pm])$ the $1$-periodic topological Fukaya categories. We warn the reader that contrary to the name, neither the $1$-periodic topological Fukaya categories, nor $\mathcal{D}(k[t_1^\pm])$, are $1$-periodic $\infty$-categories in the sense of admitting a trivialization of the suspension functor $[1]\simeq \on{id}$, unless $\on{char}(k)=2$. In \Cref{subsec:nclustercats}, we also consider $n$-periodic topological Fukaya categories, which describe the generalized $(n+1)$-cluster categories of marked surfaces. 

The $\mathcal{D}(k)$-valued topological Fukaya category is by \cite{HKK17} equivalent to the derived category of a graded gentle algebra and its representation type is tame. Its indecomposable compact object have been classified in \cite{HKK17} in terms of suitable curves in the underlying surface. Note that these curves are Lagrangians in the surface. The dimensions of the Homs between two objects corresponding to two curves can be described in terms of counts of intersections of the curves. In representation theory, these and further results relating the category to the underlying surface are commonly referred to as a geometric model for the topological Fukaya category, see \cite{OPS18}. In \Cref{sec5}, we give a similar geometric model for the $1$-periodic topological Fukaya category. The construction of objects from suitable curves, called matching curves, and the description of the Homs was given in the prequel paper \cite{Chr21b} for the derived $\infty$-category of a relative Ginzburg algebra $\mathcal{D}(\mathscr{G}_\mathcal{T})$, into which the $1$-periodic topological Fukaya category embeds fully faithfully (the embedding does not preserve compact objects). We prove the following geometrization Theorem.

\begin{theorem}[\Cref{thm:geom}]\label{thm:introgeom}
Let ${\bf S}$ be a marked surface and $\mathcal{C}_{\bf S}$ the associated $1$-periodic topological Fukaya category. Every compact object $X\in \mathcal{C}_{\bf S}$ splits uniquely into the direct sum of indecomposable objects associated to matching curves.
\end{theorem}  

Note that the restriction to compact objects in \Cref{thm:introgeom} arises from the fact that $\mathcal{C}_{\bf S}$ is an $\on{Ind}$-complete $\infty$-category, the 'perfect' version is the full subcategory of compact objects.

A similar classification result for the objects in the generalized cluster categories of non-relative Ginzburg algebras arising from triangulated surfaces without punctures appears in \cite{BZ11}, the case with punctures was treated in \cite{AP21,QZ17}. We purse a novel approach to the proof of \Cref{thm:geom} in that we employ general $\infty$-categorical techniques on the description of objects in limits of diagrams of $\infty$-categories in terms of sections of the Grothendieck construction and use the gluing properties of the objects arising from matching curves. This formalizes the idea that global sections arise by gluing together compatible families of local sections.

Brav and Dyckerhoff \cite{BD19} define a notion of a Calabi--Yau structure associated with a functor between dg-categories, also referred to as a relative Calabi--Yau structure. They further show that these relative Calabi--Yau structures have nice gluing properties and use this to construct relative $1$-Calabi--Yau structures on $\mathcal{D}(k)$-valued topological Fukaya categories. In \cite{Chr23}, the author extends this construction to topological Fukaya categories valued in an arbitrary Calabi--Yau $\infty$-category which is linear over a commutative base dg-algebra or more generally base $\mathbb{E}_\infty$-ring spectrum. The derived $\infty$-category $\mathcal{D}(k[t_1^\pm])$ of $1$-periodic chain complexes is $2$-periodic and thus linear over the commutative dg-algebra $k[t_2^\pm]$ of graded Laurent polynomials with generator in degree $|t_2|=2$. As a $k[t_2^\pm]$-linear $\infty$-category, $\mathcal{D}(k[t_1^\pm])$ is smooth and proper and admits a right $1$-Calabi--Yau structure. Applying the gluing results yields a relative right $2$-Calabi--Yau structure on the $\mathcal{D}(k[t_1^\pm])$-valued topological Fukaya category, considered as a $k[t_2^\pm]$-linear smooth and proper $\infty$-category. More precisely, the right $2$-Calabi--Yau structure is defined on the functor 
\begin{equation}\label{eq:introg} G=\prod_{e\in \mathcal{T}_1^\partial}\on{ev}_e\colon\mathcal{C}_{\bf S}\longrightarrow \prod_{e\in \mathcal{T}_1^\partial} \mathcal{D}(k[t_1^\pm])\,,\end{equation}
which takes a global section and evaluates it at all external edges of the ribbon graph $\mathcal{T}$ (the set of which is denoted by $\mathcal{T}_1^\partial$).

Finally, we note that the results of \cite{DK18} show that the $1$-periodic topological Fukaya category $\mathcal{C}_{\bf S}$ admits an action of the mapping class group of the surface, see \Cref{thm:mcgact}. This mapping class group action localizes to an action on the triangulated cluster category of the marked surface, see \Cref{prop:MCGonclustercat}. This action has not been previously constructed.

\subsection{The generalized cluster categories of marked surfaces and perverse schobers}\label{sec1.3}

We again fix a marked surface ${\bf S}$ with marked points $M$ and choose an ideal triangulation with dual trivalent ribbon graph $\mathcal{T}$. This triangulation gives rise to a quiver with potential \cite{Lab09}. We consider its ice quiver version, where the frozen vertices corresponding to the boundary components of ${\bf S}\backslash M$. Associated with this is a smooth $3$-Calabi--Yau dg-algebra, denoted $\mathscr{G}_{\mathcal{T}}$ and known as the relative Ginzburg algebra  \cite{Chr22,Wu21}. As shown in \cite{Chr22}, its derived $\infty$-category $\mathcal{D}(\mathscr{G}_\mathcal{T})$ arises as the global sections of another constructible sheaf of stable $\infty$-categories on the ribbon graph $\mathcal{T}$, denoted $\mathcal{F}_{\mathcal{T}}$. This sheaf is an example of a perverse schober, meaning a categorified perverse sheaf in the sense of \cite{KS14}. At each vertex of the ribbon graph, a perverse schober is described by a spherical adjunction $\mathcal{V}\leftrightarrow \mathcal{N}$, meaning an adjunction where the left adjoint (or equivalently the right adjoint) is a spherical functor in the sense of \cite{AL17}. The constructible sheaves giving rise to $\mathcal{D}$-valued topological Fukaya categories are also examples of perverse schobers, where the spherical adjunctions at all vertices are trivial, i.e.~given by $0\leftrightarrow \mathcal{D}$. We refer to \cite{Chr22} for more background on perverse schobers on surfaces and how to describe them as constructible sheaves on ribbon graphs. The spherical adjunction describing the perverse schober $\mathcal{F}_\mathcal{T}$ at any vertex of $\mathcal{T}$ is given by 
\[ \phi^*\colon\mathcal{D}(k)\leftrightarrow \mathcal{D}(k[t_1])\noloc \phi_*\,,\]
where $k[t_1]$ denotes the graded polynomial algebra with generator in degree $1$ and $\phi$ is the morphism of dg-algebras $k[t_1]\xrightarrow{t_1\mapsto 0} k$.

We may formally apply Amiot's quotient construction to the relative Ginzburg algebra $\mathscr{G}_\mathcal{T}$ and define the $\on{Ind}$-complete version of the generalized cluster category as 
\[ \mathcal{C}_{\bf S}\coloneqq \mathcal{D}(\mathscr{G}_{\mathcal{T}})/\on{Ind}\mathcal{D}^{\on{fin}}(\mathscr{G}_{\mathcal{T}})\simeq \on{Ind}\left(\mathcal{D}^{\on{perf}}(\mathscr{G}_\mathcal{T})/\mathcal{D}^{\on{fin}}(\mathscr{G}_\mathcal{T})\right)\,.\]
We consider the $\on{Ind}$-complete version of the generalized cluster category because of the superior formal properties of presentable $\infty$-categories.

\begin{theorem}[\Cref{thm:clstcat}, \Cref{prop:indfin}]\label{introthm:schobers}
The stable $\infty$-categories $\mathcal{C}_{\bf S}$ and $\on{Ind}\mathcal{D}^{\on{fin}}(\mathscr{G}_{\mathcal{T}})$ arise as the global sections of two perverse subschobers $\mathcal{F}_{\mathcal{T}}^{\on{clst}}$ and $\mathcal{F}_{\mathcal{T}}^{\on{mnd}}$ of $\mathcal{F}_\mathcal{T}$, respectively. These fit into a fiber and cofiber sequence of perverse schobers:
\begin{equation}\label{eq:scofs}
\begin{tikzcd}
\mathcal{F}_{\mathcal{T}}^{\on{Ind-fin}}\arrow[r] \arrow[rd, phantom, "\square"]\arrow[d]& \mathcal{F}_{\mathcal{T}} \arrow[d] \\ 0\arrow[r] & \mathcal{F}_{\mathcal{T}}^{\on{clst}}
\end{tikzcd}
\end{equation}
\end{theorem}

At the vertices of $\mathcal{T}$, the perverse schober $\mathcal{F}_{\mathcal{T}}^{\on{Ind-fin}}$ is described by the spherical adjunction 
\begin{equation}\label{eq:intro1} \mathcal{D}(k)\longleftrightarrow \on{Ind}\mathcal{D}^{\on{fin}}(k[t_1])\end{equation}
arising from restricting $\phi^*\dashv \phi_*$ and $\mathcal{F}_{\mathcal{T}}^{\on{clst}}$ is described by the 'quotient' spherical adjunction 
\[ 0\longleftrightarrow \mathcal{D}(k[t_1])/\on{Ind}\mathcal{D}^{\on{fin}}(k[t_1])\simeq \mathcal{D}(k[t_1^\pm])\,,\]
which is the trivial spherical adjunction from above with $\mathcal{D}=\mathcal{D}(k[t_1^\pm])$. The derived $\infty$-category of $1$-periodic chain complexes $\mathcal{D}(k[t_1^\pm])\simeq \mathcal{D}(k[t_1])/\on{Ind}\mathcal{D}^{\on{fin}}(k[t_1])$ arises here naturally as the quotient of the derived category of the graded polynomial algebra $k[t_1]$ by it's $\on{Ind}$-complete finite derived category and may also be seen as the $1$-Calabi--Yau cluster category of type $A_1$. In this way, we arrive at the observation that the generalized cluster category $\mathcal{C}_{\bf S}$ associated to the relative Ginzburg algebra $\mathscr{G}_{\mathcal{T}}$ is the $1$-periodic topological Fukaya category. 

Replacing the $1$-periodic derived category by the $1$-Calabi--Yau cluster category of type $A_n$ and forming the corresponding topological Fukaya category, we expect to obtain a categorification of the higher rank cluster algebras associated with the marked surfaces in the sense of Fock--Goncharov \cite{FG06}. We will return to this question in future work. 

We further refine the statement of \Cref{introthm:schobers} by introducing a notion of semiorthogonal decomposition of perverse schobers, see \Cref{def:SODschobers}, and showing that the perverse schober $\mathcal{F}_{\mathcal{T}}$ admits a semiorthogonal decomposition into $\mathcal{F}_{\mathcal{T}}^{\on{clst}}$ and $\mathcal{F}_{\mathcal{T}}^{\on{Ind-fin}}$. We will see that the spherical adjunction \eqref{eq:intro1} is the monadic adjunction associated to the adjunction monad $\phi_*\phi^*$ of the adjunction $\phi^*\dashv \phi^*$. We expect the semiorthogonal decomposition of $\mathcal{F}_{\mathcal{T}}$ to be a special case of a general principle, by which, under mild assumptions, a perverse schober decomposes into subschobers with monadic or trivial spherical adjunctions. The passage from $\mathcal{D}(\mathscr{G}_\mathcal{T})$ to the generalized cluster category $\mathcal{C}_{\bf S}$ can hence be motivated intrinsically in the framework of perverse schobers.

Work of Ivan Smith \cite{Smi15} gives a symplectic interpretation of the finite derived category of the (non-relative) Ginzburg algebra associated with a triangulated marked surface: its homotopy category embeds fully faithfully into the derived Fukaya category of a Calabi--Yau threefold $Y$ with a fibration $\pi$ to ${\bf S}$. The wrapped Fukaya category of the regular fiber of $\pi$ is equivalent to $\D^{\on{perf}}(k[t_1])$. From this perspective, the $1$-periodic derived category $\D(k[t_1^\pm])$ describes the Rabinowitz wrapped Fukaya category of this fiber, which is by \cite{GGV22} equivalent to Efimov's (algebraizable) formal punctured neighborhood of infinity \cite{Efi17} of the wrapped Fukaya category of the fiber:

\begin{proposition}\label{prop:formalneighborhood}
The category of perfect complexes on the formal punctured neighborhood of infinity of $\D^{\on{perf}}(k[t_1])$ is equivalent to the perfect derived category of $1$-periodic chain complexes $\D^{\on{perf}}(k[t_1^\pm])$. 
\end{proposition}

\subsection{The exact \texorpdfstring{$\infty$}{infinity}-structure and the categorification of the cluster algebra}\label{sec1.4}

Let again ${\bf S}$ be a marked surface and choose an auxiliary ideal triangulation with dual trivalent ribbon graph $\mathcal{T}$. As explained in \Cref{sec1.3}, the generalized cluster category $\mathcal{C}_{\bf S}$ arising from the relative Ginzburg algebra $\mathscr{G}_{\mathcal{T}}$ is equivalent to the $1$-periodic topological Fukaya category of the marked surface ${\bf S}$. In this section, we explain how we equip the stable $\infty$-category $\mathcal{C}_{\bf S}$ with the structure of an exact $\infty$-category and how this exact $\infty$-category gives an additive categorification of a cluster algebra with coefficients. 
 
As mentioned in \Cref{sec1.2}, the functor $G\colon\mathcal{C}_{\bf S}\rightarrow \prod_{e\in \mathcal{T}_1^\partial}\mathcal{D}(k[t_1^\pm])$ from \eqref{eq:introg} that restricts global sections to the external edges has a right $2$-Calabi--Yau structure. We use this relative Calabi--Yau structure to identify a subset of morphisms in $\mathcal{C}_{\bf S}$ which satisfy the $2$-Calabi--Yau duality.

\begin{definition}[\Cref{def:cyfun} and \Cref{lem:cyfun}]
Let
\[ \on{Mor}_{\mathcal{C}_{\bf S}}^{\on{CY}}(\mhyphen,\mhyphen)\subset \on{Mor}_{\mathcal{C}_{\bf S}}(\mhyphen,\mhyphen)\] 
be the maximal subfunctor of the derived Hom-functor satisfying that 
\[ \on{Mor}_{\mathcal{C}_{\bf S}}^{\on{CY}}(G(\mhyphen),G(\mhyphen))\simeq 0\,.\] 
For $X,Y\in \mathcal{C}_{\bf S}$, we call the $k$-vector space $\on{Ext}^{i,\on{CY}}(X,Y)\coloneqq\on{H}_0\on{Mor}_{\mathcal{C}_{\bf S}}^{\on{CY}}(X,Y[i])$ the Calabi--Yau extensions. Calabi--Yau extensions form a subfunctor of the $\on{Ext}$-functor.
\end{definition}

\begin{proposition}[\Cref{prop:cydual}]\label{prop:introcy}
The Calabi--Yau extensions satisfy the $2$-Calabi--Yau duality 
\[ \on{Ext}^{i,\on{CY}}_{\mathcal{C}_{\bf S}}(X,Y)\simeq \on{Ext}^{2-i,\on{CY}}_{\mathcal{C}_{\bf S}}(Y,X)^*\,,\]
bifunctorially in $X$ and $Y$.
\end{proposition}

In terms of the geometric model for $\mathcal{C}_{\bf S}$ based on curves in ${\bf S}$, Calabi--Yau extensions have a very simple interpretation. Given two objects arising from curves in ${\bf S}$, a basis of the extensions arises from crossings and directed boundary intersections. The extensions arising from crossings are Calabi--Yau extensions, whereas the extensions arising from directed boundary intersections are not Calabi--Yau. 

The Calabi--Yau extensions can be seen as the exact extensions in an exact $\infty$-structure on $\mathcal{C}_{\bf S}$, which arises from the Calabi--Yau functor $G\colon\mathcal{C}_{\bf S}\rightarrow \prod_{e\in \mathcal{T}_1^\partial}\mathcal{D}(k[t_1^\pm])$ as follows. We consider $\prod_{e\in \mathcal{T}_1^\partial}\mathcal{D}(k[t_1^\pm])$ as an exact $\infty$-category with the split-exact structure, meaning that the exact sequences are the split fiber and cofiber sequences. We can now pull back the exact structure along $G$, to an exact structure on $\mathcal{C}_{\bf S}$, see \Cref{ex:splitpb}. The exact sequences in $\mathcal{C}_{\bf S}$ are by definition those which are mapped under $G$ to split-exact sequences in $\prod_{e\in \mathcal{T}_1^\partial} \mathcal{D}(k[t_1^\pm])$. The isomorphism classes of such exact sequences $X\rightarrow Z\rightarrow Y$ are described by $\on{Ext}^{1,\on{CY}}_{\mathcal{C}_{\bf S}}(X,Y)$. The exact structure further induces an extriangulated structure in the sense of Nakaoka-Palu on the homotopy $1$-category $\on{ho}\C_{\bf S}^{\on{c}}$ of the subcategory of compact objects. \Cref{prop:introcy} states that this extriangulated category is $2$-Calabi--Yau. We note that pulling back exact structures is a general source of exact structures. In the setting of extriangulated categories, this is known as 'relative theory', see \cite[Section 3.2]{HLN21}. 

If the functor along which the split-exact structure is pulled back is a spherical functor, we show that the exact $\infty$-category is furthermore Frobenius, meaning that it has enough injective and enough projective objects and that the two classes of these objects coincide, see \Cref{prop:frobex}. In the setting of triangulated categories, this was also observed in \cite{BS21}. The functor $G$ is indeed spherical, see \Cref{cor:sphbdry}. 

The following Theorem is the first part of the categorification of the cluster algebra in terms of $\mathcal{C}_{\bf S}$.

\begin{theorem}[\Cref{cor:bmrcluster}]\label{thm:introct}
There are canonical bijections between the sets of the following objects.
\begin{itemize}
\item Clusters of the cluster algebra with coefficients in the boundary arcs associated with ${\bf S}$.
\item Ideal triangulations of ${\bf S}$.
\item Cluster tilting objects in the $2$-Calabi--Yau Frobenius extriangulated category $\on{ho}\mathcal{C}_{\bf S}^{\on{c}}$.
\end{itemize}
\end{theorem}

In \Cref{subsec:clustertilting}, we also describe the endomorphism algebras of the cluster tilting objects, they are given by finite-dimensional non-smooth gentle algebras. 

The second part of the additive categorification of a cluster algebra in terms of $\mathcal{C}_{\bf S}$ is the description of a cluster character on the extriangulated category $\on{ho}\mathcal{C}_{\bf S}^{\on{c}}$ with values in the commutative Kauffman Skein algebra $\on{Sk}^1_{\bf S}$ of the surface ${\bf S}$ (with parameter $q=1$). The elements of this algebra are links in ${\bf S}$, meaning superpositions of curves, modulo relations such as the Kauffman Skein relation, see \Cref{def:skein}. This algebra embeds into the upper cluster algebra of ${\bf S}$ with coefficients in the boundary arcs, see \cite{Mul16}. 

\begin{theorem}[\Cref{thm:char}]
There is a cluster character $\chi\colon\on{obj}(\mathcal{C}_{\bf S}^{\on{c}})\rightarrow \on{Sk}^1_{\bf S}$ with values in the commutative Skein algebra of links in ${\bf S}$. The character maps an object corresponding to a matching curve in ${\bf S}$ to the matching curve considered as a link with a single component. Furthermore, composing $\chi$ with the inclusion of $\on{Sk}^1_{\bf S}$ into the upper cluster algebra yields a cluster character to the upper cluster algebra.
\end{theorem}

\subsection{Relation with other types of cluster categories of surfaces}\label{introsubsec:otherclustercats}

We begin by comparing the approach of this paper with Yilin Wu's recent approach to the categorification of cluster algebras with coefficients called relative cluster categories and Higgs categories, see \cite{Wu21}. In loc.~cit., a different generalization of Amiot's construction is considered than in this paper, however the outcome is also a Frobenius $2$-Calabi--Yau extriangulated category $\on{ho}\mathcal{H}$, called the Higgs category. The extriangulated structure of the Higgs category arises from an ambient triangulated category $\on{ho}\C^{\on{rel}}$, called the relative cluster category, in which the Higgs category lies as an extension closed subcategory. The relative cluster category associated with the relative Ginzburg algebra $\mathscr{G}_\mathcal{T}$ of the marked surface ${\bf S}$ with ideal triangulation $\mathcal{T}$ is given by the Verdier quotient $\C^{\on{rel}}=\D^{\on{perf}}(\mathscr{G}_\mathcal{T})/\D^{\on{fin}}_F(\mathscr{G}_\mathcal{T})$, where $\D^{\on{fin}}_F(\mathscr{G}_\mathcal{T})\subset \D^{\on{fin}}(\mathscr{G}_\mathcal{T})$ is a certain subcategory. There is thus a canonical quotient functor $\tau\colon \C^{\on{rel}}\to \C_{\bf S}$. We show in \Cref{thm:equivHiggs} that $\tau$ induces an equivalence of extriangulated categories $\on{ho}\mathcal{H}\simeq \on{ho}\C_{\bf S}$, and also an equivalence on the enhancements given by exact $\infty$-categories. 

We next discuss the relation with the standard triangulated $2$-Calabi--Yau cluster category $C_{\bf S}$ associated with the marked surface ${\bf S}$. As shown in \cite{JKPW22}, the stable category $\bar{\mathcal{C}^{\on{c}}_{\bf S}}$ of the Frobenius exact $\infty$-category $\mathcal{C}_{\bf S}^{\on{c}}$, obtained by a localization annihilating the injective projective objects, is a stable $\infty$-category, see also \Cref{prop:jkpw}. Its homotopy category is equivalent to the triangulated stable category underlying the extriangulated category $\on{ho}\mathcal{H}\simeq \on{ho}\C_{\bf S}$. As shown in \cite{Wu21}, the homotopy category $\on{ho}\mathcal{H}$ agrees with $C_{\bf S}$. We thus obtain:

\begin{theorem}
The stable category $\bar{\mathcal{C}^{\on{c}}_{\bf S}}$ of the Frobenius exact $\infty$-category $\C^{\on{c}}_{\bf S}$ is an enhancement of the standard triangulated cluster category $C_{\bf S}$ associated with ${\bf S}$. Stated differently, there exists an equivalence of triangulated categories $\on{ho}\bar{\mathcal{C}^{\on{c}}_{\bf S}}\simeq C_{\bf S}$.
\end{theorem}

Our results on $\C_{\bf S}$ recover many results for $C_{\bf S}$. For instance the classification of the indecomposables objects in $C_{\bf S}$ in terms of matching curves which are not boundary arcs is also given in \cite{BZ11}. The description of the dimensions of the $\on{Ext}^1$'s in $C_{\bf S}$ in terms of numbers of crossings of the matching curves was also shown in \cite{ZZZ13,CS17}. \Cref{rem:shift} shows that $S^{-1}[2]\simeq \on{id}$ on objects in $\bar{\mathcal{C}^{\on{c}}_{\bf S}}$, where $S$ is the Serre functor of $\mathcal{C}_{\bf S}$, which agrees with the description of the shift functor in $C_{\bf S}$ given in \cite[Thm.~5.2]{QZ17}.

Topological Fukaya categories of surfaces (with values in $\D(k)$) are equivalent to the derived categories of graded gentle algebras by \cite{HKK17}. The generalized cluster category $\C_{\bf S}$ can thus be seen as the derived category of $1$-periodic representations of a gentle algebra, see also \Cref{prop:derivedgtl}. We remark that this gentle algebra is not derived equivalent to the Jacobian gentle algebra of the quiver with potential associated with the ideal triangulation of ${\bf S}$ of \cite{Lab09,ABCP10}. The surface associated with the latter gentle algebra in the sense of \cite{OPS18} is the surface obtained from ${\bf S}$ by removing the vertices of the trivalent spanning graph dual to the ideal triangulation. Typical approaches in the literature study the triangulated cluster category $C_{\bf S}$ by relating its representation theory with the representation theory of this Jacobian gentle algebra. The relationship explored in this paper between the cluster category and the derived category of the former gentle algebra, i.e.~the topological Fukaya category, is much closer. 

\subsection{Acknowledgements}

I wish to thank my Ph.D.~advisor Tobias Dyckerhoff for his support. I thank Gustavo Jasso and Mikhail Gorsky for crucial suggestions on the role of exact $\infty$-structures and Higgs categories. I also thank Martin Herschend, Bernhard Keller, Yann Palu, Yu Qiu, Ivan Smith and Yilin Wu for helpful discussions. The author acknowledges support by the Deutsche Forschungsgemeinschaft under Germany’s Excellence Strategy – EXC 2121 “Quantum Universe” – 390833306. This project has received funding from the European Union’s Horizon 2020 research and innovation programme under the Marie Skłodowska-Curie grant agreement No 101034255.

\subsection{Notations and conventions}

In this paper, we freely employ the langue of $\infty$-categories, as developed in \cite{HTT,HA}. We generally follow the notation and conventions of \cite{HTT,HA}. In particular, we use the homological grading. Given an $\infty$-category $\mathcal{C}$ and two objects $X,Y\in \mathcal{C}$, we denote by $\on{Map}_{\mathcal{C}}(X,Y)$ the mapping space. We further use the following conventions.
\begin{itemize}
\item We denote the homotopy $1$-category of an $\infty$-category $\mathcal{C}$ by $\on{ho}\mathcal{C}$.
\item We denote the right adjoint of a functor $F$ by $\on{radj}(F)$ and the left adjoint by $\on{ladj}(F)$ (if existent).
\item Given a functor $F\colon\mathcal{A}\rightarrow \mathcal{B}$, we denote by $F(\mathcal{A})$ the essential image of $F$, meaning the smallest full subcategory of $\mathcal{B}$ which is closed under equivalences in $\mathcal{B}$ and contains the objects in the image of $F$. 
\item Matrices can be used to describe morphisms between direct sums of objects in additive $\infty$-categories, such as stable $\infty$-categories. For this we use the convention of multiplication with a matrix from the left. 
\end{itemize}

\section{Preliminaries}

\subsection{\texorpdfstring{$\infty$}{infinity}-categories of \texorpdfstring{$\infty$}{infinity}-categories}
We use the following notation for different $\infty$-categories of $\infty$-categories. 
We denote by
\begin{itemize}
\item $\mathcal{S}$ the $\infty$-category of spaces.
\item $\on{Cat}_\infty$ the $\infty$-category of $\infty$-categories.
\item $\on{St}$ the $\infty$-category of stable $\infty$-categories and exact functors.
\item $\mathcal{P}r^L$ the $\infty$-category of presentable $\infty$-categories and left adjoint functors. 
\item $\mathcal{P}r^R$ the $\infty$-category of presentable $\infty$-categories and right adjoint functors. 
\item by $\mathcal{P}r^L_{\on{St}}\subset \mathcal{P}r^L$ and $\mathcal{P}r^R_{\on{St}}\subset \mathcal{P}r^R$ the full subcategories consisting of stable $\infty$-categories.
\end{itemize}

We remark that passing to adjoint functors defines an equivalence of $\infty$-categories $\mathcal{P}r^L\simeq (\mathcal{P}r^R)^{\on{op}}$, see \cite[5.5.3.4]{HTT}.

Let $\mathcal{C}$ be an $\infty$-category. An object $x$ in an $\infty$-category $\mathcal{C}$ is called compact if the functor $\on{Map}_{\mathcal{C}}(x,\mhyphen)\colon\mathcal{C}\rightarrow \mathcal{S}$ preserves filtered colimits. We denote by $\mathcal{C}^{\on{c}}\subset \mathcal{C}$ the full subcategory of compact objects. If $\C$ is small, we denote by $\on{Ind}(\mathcal{C})$ the $\on{Ind}$-completion of $\mathcal{C}$. If $\mathcal{C}$ is small, stable and idempotent complete, then $\mathcal{C}\simeq \on{Ind}(\mathcal{C})^{\on{c}}$.

The $\infty$-category $\mathcal{P}r^L$ admits a symmetric monoidal structure, such that a commutative algebra object in $\mathcal{P}r^L$ amounts to a symmetric monoidal presentable $\infty$-category $\mathcal{C}$, satisfying that its tensor product $\mhyphen \otimes \mhyphen \colon\mathcal{C}\times \mathcal{C}\rightarrow \mathcal{C}$ preserves colimits in both entries, see \cite[Section 4.7]{HA}. An example of such a commutative algebra object in $\mathcal{P}r^L$ is the derived $\infty$-category $\mathcal{D}(R)$ of a commutative dg-algebra $R$. 

\begin{definition}
Let $R$ be a commutative dg-algebra. We call the $\infty$-category $\on{Mod}_{\mathcal{D}(R)}(\mathcal{P}r^L)$ of modules, see \cite[Def.~4.5.1.1]{HA}, the $\infty$-category of $R$-linear presentable $\infty$-categories and denote it by $\on{LinCat}_R$.
\end{definition}

As noted in \cite[Section D.1.5]{SAG}, $R$-linear $\infty$-categories in the above sense are automatically stable.
 
\begin{definition}[$\!\!${\cite[4.2.1.28]{HA}}]
Let $R$ be a commutative dg-algebra. Let $\mathcal{A}$ be an $R$-linear $\infty$-category and let $X,Y\in \mathcal{A}$. A morphism object for $X,Y$ is an $R$-module $\on{Mor}_{\mathcal{C}}(X,Y)\in \mathcal{D}(R)$ equipped with a map $\alpha\colon\on{Mor}_{\mathcal{C}}(X,Y)\otimes_R X\rightarrow Y$ in $\mathcal{C}$ such that for every object $C\in \mathcal{D}(R)$ composition with $\alpha$ induces an equivalence of spaces
\[ \on{Map}_{\mathcal{D}(R)}(C,\on{Mor}_{\mathcal{C}}(X,Y))\rightarrow \on{Map}_{\mathcal{C}}(C\otimes_R X,Y)\,.\]
\end{definition}

We thus have $\on{H}_i\on{Mor}_{\mathcal{C}}(X,Y)\simeq \pi_0\on{Map}_{\mathcal{C}}(X[i],Y)$. 

\begin{remark}
Morphism objects always exist and the formation of morphism objects forms an exact functor 
\[ \on{Mor}_\mathcal{C}(\mhyphen,\mhyphen)\colon\mathcal{C}^{\on{op}}\times \mathcal{C}\longrightarrow \mathcal{D}(R)\,.\]
\end{remark}

Given an $R$-linear stable $\infty$-category $\mathcal{C}$ and two objects $A,B\in \mathcal{C}$, the $n$-th $\on{Ext}$-group is defined as 
\[ \on{Ext}^n_{\mathcal{C}}(A,B)\coloneqq \on{H}_{-n}\on{Mor}_{\mathcal{C}}(A,B)\simeq \on{H}_0\on{Mor}_{\mathcal{C}}(A,B[n])\,.\]

\subsection{Semiorthogonal decompositions and Verdier quotients}

Given a stable $\infty$-category $\mathcal{C}$ and a full subcategory $\mathcal{A}\subset \mathcal{C}$, we call $\mathcal{A}$ a stable subcategory of $\mathcal{C}$ if $\mathcal{A}$ is stable, the inclusion is an exact functor and $\mathcal{A}$ is closed under equivalences in $\mathcal{C}$. 

Given a fully faithful functor $i\colon \A \to \C$, we will later on sometimes abuse notation by identifying $\A$ with the stable subcategory of $\C$ given by the essential image of $i$. 

\begin{definition}\label{def:SOD}
Let $\mathcal{C}$ be a stable $\infty$-category and let $\mathcal{A},\mathcal{B}\subset \mathcal{B}$ be stable subcategories. We say that the pair $(\mathcal{A},\mathcal{B})$ forms a semiorthogonal decomposition of $\mathcal{C}$ if 
\begin{itemize}
\item for all $a\in \mathcal{A}$ and $b\in \mathcal{B}$, the mapping space $\on{Map}_{\mathcal{C}}(b,a)$ is contractible and
\item for all $x\in \mathcal{C}$, there exists a fiber and cofiber sequence $b\rightarrow x\rightarrow a$ in $\mathcal{C}$ with $a\in \mathcal{A}$ and $b\in \mathcal{B}$.
\end{itemize}
\end{definition}

\begin{definition}\label{exseqdef}
An exact sequence of stable, presentable $\infty$-categories consists of a cofiber sequence in $\mathcal{P}r^L_{\on{St}}$
\[ 
\begin{tikzcd}
\mathcal{A} \arrow[r, "\iota"] \arrow[d] \arrow[rd, "\ulcorner", phantom] & \mathcal{C} \arrow[d] \\
0 \arrow[r]                                                               & \mathcal{B}          
\end{tikzcd}
\] 
such that $i$ is a fully faithful functor.  
\end{definition}

\begin{remark}
In the setting of \Cref{exseqdef}, the triangulated homotopy category of $\mathcal{B}$  is equivalent to the triangulated Verdier quotient of $\mathcal{C}$ by $\mathcal{A}$, see \cite[Prop.~5.9]{BGT13}. We thus call $\mathcal{B}$ the Verdier quotient of $\mathcal{C}$ by $\mathcal{A}$. 
\end{remark}

The following Lemma shows that the datum of a semiorthogonal decomposition of a stable, presentable $\infty$-category is equivalent to the datum of an exact sequence in $\mathcal{P}r^L_{\on{St}}$.

\begin{lemma}\label{sodquotlem}
Consider a diagram in $\mathcal{P}r^L_{\on{St}}$
\begin{equation}\label{eq:acbseq}
\mathcal{A}\xrightarrow{i}\mathcal{C} \xrightarrow{\pi}\mathcal{B}\,.
\end{equation}
Suppose that $i$ and $\on{radj}(\pi)$ are fully faithful. Then the diagram can be extended to an exact sequence if and only if $\left(\on{radj}(\pi)(\mathcal{B}),i(\mathcal{A})\right)$ forms a semiorthogonal decomposition of $\mathcal{C}$.
\end{lemma} 

\begin{proof}
Suppose that \eqref{eq:acbseq} is part of an exact sequence. We find that for all $a\in \mathcal{A}$ and $b\in \mathcal{B}$, the mapping space $\on{Map}_{\mathcal{C}}(\iota(a),\on{radj}(\pi)(b))\simeq \on{Map}_{\mathcal{B}}(\pi\iota(a),b)\simeq \on{Map}_{\mathcal{B}}(0,b)$ is contractible. Let $c\in \mathcal{C}$ and $\on{cu}_c\colon c\rightarrow \on{radj}(\pi) (\pi(c))$ be the counit. The map $\pi(\on{cu}_c)$ is an equivalence by the 2/3-property and the fact that $\on{radj}(\pi)$ is fully faithful. We hence have $\pi(\on{fib}(\on{cu}_c))\simeq 0$. It follows that $\on{fib}(\on{cu}_c)\in \on{Im}(i)$. There thus exists an object $a\in \mathcal{A}$ and an exact sequence $a\rightarrow c\rightarrow \on{radj}(\pi)\circ \pi(c)$ in $\mathcal{C}$. We have shown that $\left(i(\mathcal{A}),\on{radj}(\pi)(\mathcal{B})\right)$ forms a semiorthogonal decomposition of $\mathcal{C}$. 

For the converse implication, assume that $(\on{radj}(\pi)(\mathcal{B}),i(\mathcal{A}))$ forms a semiorthogonal decomposition of $\mathcal{C}$. The assertion that \eqref{eq:acbseq} can be extended to an exact sequence is via the equivalence $\on{radj}\colon\mathcal{P}r^L_{\on{St}}\simeq (\mathcal{P}r^R_{\on{St}})^{\on{op}}$ equivalent to the assertion that there exists a pullback diagram in $\on{Cat}_{\infty}$ as follows:
\[
\begin{tikzcd}
\mathcal{B} \arrow[r, "\on{radj}(\pi)"] \arrow[d] \arrow[rd, "\lrcorner", phantom] & \mathcal{C} \arrow[d, "\on{radj}(i)"] \\
0 \arrow[r]                                                                        & \mathcal{A}                          
\end{tikzcd}
\]
This is in turn equivalent to the assertion that $\on{radj}(\pi)$ defines an equivalence between $\mathcal{B}$ and the full subcategory of $\mathcal{C}$ spanned by objects $c\in \mathcal{C}$, satisfying that $\on{Map}_{\mathcal{C}}(i(a),c)$ is contractible for all $a\in \mathcal{A}$. Proposition 2.2.4 from \cite{DKSS21} shows that this is a property of any semiorthogonal decomposition, concluding the proof of the converse implication.
\end{proof}

\subsection{Monadicity and derived categories of graded Laurent algebras}\label{sec2.3}

Let $\mathcal{C}$ be an $\infty$-category. A monad $M$ on $\mathcal{C}$ is an associative algebra object in the monoidal $\infty$-category $\on{End}(\mathcal{C})=\on{Fun}(\mathcal{C},\mathcal{C})$ of endofunctors. In other words, a monad is a functor $M\colon\mathcal{C}\rightarrow \mathcal{C}$, together with a unit $u\colon\on{id}_{\mathcal{C}}\rightarrow M$ and a multiplication map $M\circ M\rightarrow M$ equipped with data exhibiting coherent associativity and unitality. A main source of monads are adjunctions: given an adjunction $F\colon \mathcal{C}\leftrightarrow \mathcal{D}\noloc G$ of $\infty$-categories, there is an associated monad $M=GF$, see \cite[Section 4.7.3]{HA}, which we call the adjunction monad. Associated to a monad $M$ on $\mathcal{C}$ is its $\infty$-category $\on{LMod}_M(\mathcal{C})$ of left modules in $\mathcal{C}$ and the free-forget adjunction 
\[ \on{Free}\colon\mathcal{C}\longleftrightarrow \on{LMod}_M(\mathcal{C})\noloc \on{Forget}\,,\]
whose adjunction monad is equivalent to $M$, see \cite{HA}. The $\infty$-category $\on{LMod}_M(\mathcal{C})$ is also called the Eilenberg-Moore $\infty$-category of $M$. A typical example is given by $\mathcal{C}=\mathcal{D}(k)$ the derived $\infty$-category of a field and $M=A\otimes_k\mhyphen$ the monad arising from tensoring with a dg-algebra $A$. In this case $\on{LMod}_M(\mathcal{D}(k))\simeq \mathcal{D}(A)$ and associated adjunction is equivalent to the usual free-forget adjunction $\mathcal{D}(k)\leftrightarrow \mathcal{D}(A)$.

Given an adjunction $F\colon\mathcal{C}\leftrightarrow \mathcal{D}\noloc G$ with monad $M=GF$, there is an associated functor $\mathcal{D}\rightarrow \on{LMod}_M(\mathcal{C})$, making the following diagram commute:
\[
\begin{tikzcd}
\mathcal{D} \arrow[rd, "G"'] \arrow[rr] &             & \on{LMod}_M(\mathcal{C}) \arrow[ld, "\on{Forget}"] \\
                                        & \mathcal{C} &                                                   
\end{tikzcd}
\]

\begin{definition}
Let $G\colon\mathcal{D}\rightarrow \mathcal{C}$ be a functor between $\infty$-categories.
\begin{enumerate}
\item The functor $G$ is called monadic if it admits a left adjoint $F$ with adjunction monad $M=GF$ and the associated functor $\mathcal{D}\rightarrow \on{LMod}_M(\mathcal{C})$ is an equivalence of $\infty$-categories.
\item The functor $G$ is called comonadic if the opposite functor $G^{\on{op}}\colon \mathcal{D}^{\on{op}}\rightarrow \mathcal{C}^{\on{op}}$ is monadic.
\end{enumerate}
\end{definition}

If $\mathcal{C},\mathcal{D}$ are presentable $\infty$-categories and $G$ preserves (sufficient) colimits, then $\mathcal{D}\rightarrow \on{LMod}_M(\mathcal{C})$ admits a fully faithful left adjoint, see \cite[Lemma 4.7.3.13]{HA}. In this case, $\on{LMod}_M(\mathcal{C})$ is presentable and if $\mathcal{C}$ and $\mathcal{D}$ are furthermore stable, then $\on{LMod}_M(\mathcal{C})$ is also stable, see \cite[Prop.~4.2.3.4]{HA}.

To determine whether a functor is monadic, one can use the Lurie--Barr--Beck monadicity theorem, see \cite[Thm.~4.7.3.5]{HA}. Assuming that all involved $\infty$-categories are presentable and that the functor preserves colimits, the theorem reduces to the statement that a right adjoint $G\colon \mathcal{D}\rightarrow \mathcal{C}$ is monadic if and only if it is conservative, i.e.~reflects isomorphisms. Similarly, in this setting a left adjoint $F$ which preserves sufficient limits is comonadic if and only if it is conservative.

We now turn to a specific example of a monadic adjunction. Let $k$ be a field and $n\geq 2$. We are interested in the monadic adjunction arising from the spherical adjunction 
\[ f^*\colon \mathcal{D}(k)\longleftrightarrow \on{Fun}(S^n,\mathcal{D}(k))\noloc f_*\,,\]
where $\on{Fun}(S^n,\mathcal{D}(k))$ is the $\infty$-category of $\mathcal{D}(k)$-valued local systems on the $n$-sphere $S^n$ and $f^*$ is the pullback along $f\colon S^n\rightarrow \ast$, see \cite{Chr20} for a detailed study of this adjunction. In the case $n=2$, this adjunction describes the singularities of the perverse schober $\mathcal{F}_\mathcal{T}$ considered in \Cref{sec4.1}. Let $M=f_*f^*$ be the adjunction monad and denote by $\on{Fun}(S^n,\mathcal{D}(k))^{\on{mnd}}\coloneqq \on{LMod}_M(\mathcal{D}(k))$ the stable, presentable $\infty$-category of left modules over $M$. We denote the associated fully faithful functor $\on{Fun}(S^n,\mathcal{D}(k))^{\on{mnd}}\rightarrow \on{Fun}(S^n,\mathcal{D}(k))$ by $i_{\on{mnd}}$ and we will identify $\on{Fun}(S^n,\mathcal{D}(k))^{\on{mnd}}$ with its essential image under $i_{\on{mnd}}$.

\begin{lemma}\label{lem:cmpgen}
The object $f^*(k)$ is a compact generator of the $\infty$-category $\on{Fun}(S^n,\mathcal{D}(k))^{\on{mnd}}$.  
\end{lemma}

\begin{proof}
The restriction of $f_*$ to $\on{Fun}(S^n,\mathcal{D}(k))^{\on{mnd}}$ is by definition monadic. Its left adjoint $(f^*)^{\on{mnd}}$ is obtained from $f^*$ by restricting the target. The $k$-linear functor $(f^*)^{\on{mnd}}$ is fully determined by the image of $k$, which is $f^*(k)$. It is thus equivalent to the functor $(\mhyphen)\otimes f^*(k)$ (defined using the $k$-linear structure). The adjunction
 \[ (\mhyphen)\otimes f^*(k)\dashv \on{Mor}_{\on{Fun}(S^n,\mathcal{D}(k))^{\on{mnd}}}(f^*(k),\mhyphen)\,,\] 
and the facts that the right adjoint $\on{Mor}_{\on{Fun}(S^n,\mathcal{D}(k))^{\on{mnd}}}(f^*(k),\mhyphen)\simeq (f_*)^{\on{mnd}}$ is conservative, because monadic, and that $k\in \mathcal{D}(k)$ is compact generator, now imply that $f^*(k)$ compactly generates $\on{Fun}(S^n,\mathcal{D}(k))^{\on{mnd}}$.
\end{proof}

As shown in \cite[Prop.~5.5]{Chr22}, there exists an equivalence of $k$-linear $\infty$-categories
\[\on{Fun}(S^n,\mathcal{D}(k))\simeq \mathcal{D}(k[t_{n-1}])\,,\] 
where $k[t_{n-1}]$ denotes the graded polynomial algebra with generator in degree $|t_{n-1}|=n-1$. 

\begin{lemma}\label{fin=mndlem}
There exist equivalences of $k$-linear $\infty$-categories 
\[ \on{Fun}(S^n,\mathcal{D}(k))^{\on{mnd}}\simeq \on{Ind}\on{Fun}(S^n,\mathcal{D}^{\on{perf}}(k))\simeq \on{Ind}\mathcal{D}^{\on{fin}}(k[t_{n-1}])\,,\]
compatible with their inclusions into $\on{Fun}(S^n,\mathcal{D}(k))$, where $\mathcal{D}^{\on{fin}}(k[t_{n-1}])$ denotes the derived $\infty$-category of $k[t_{n-1}]$-modules with finite dimensional homology.
\end{lemma}

\begin{proof}
Under the equivalence $\on{Fun}(S^n,\mathcal{D}(k))\simeq \mathcal{D}(k[t_{n-1}])$, the forgetful functor $\mathcal{D}(k[t_{n-1}])\rightarrow \mathcal{D}(k)$ corresponds to the evaluation functor $\on{Fun}(S^n,\mathcal{D}(k))\rightarrow \mathcal{D}(k)$ at any point $x\in S^n$. This implies that $\on{Fun}(S^n,\mathcal{D}^{\on{perf}}(k))\simeq \mathcal{D}^{\on{fin}}(k[t_{n-1}])$, which yields  $\on{Ind}\on{Fun}(S^n,\mathcal{D}^{\on{perf}}(k))\simeq \on{Ind}\mathcal{D}^{\on{fin}}(k[t_{n-1}])$ by passing to $\on{Ind}$-completions.

We proceed by showing the equivalence $\on{Fun}(S^{n},\mathcal{D}(k))^{\on{mnd}}\simeq \on{Ind}\mathcal{D}^{\on{fin}}(k[t_{n-1}])$. Consider the trivial $k[t_{n-1}]$-module with homology $k$ (corresponding to $f^*(k)$ under the equivalence $\mathcal{D}^{\on{fin}}(k[t_{n-1}])\simeq \on{Fun}(S^n,\mathcal{D}(k)^{\on{perf}})$). Let $X\in \mathcal{D}^{\on{fin}}(k[t_{n-1}])$. We show by induction on the dimension of the homology of $X$, that $X$ lies in the stable closure $\langle k\rangle \subset \mathcal{D}(k[t_{n-1}])$ of $k$, i.e.~the smallest stable subcategory containing $k$. The induction beginning is the case $X\simeq 0$ and thus clear. For the induction step, suppose that $m\in \mathbb{Z}$ is the maximal degree in which the homology $\on{H}_*\on{Mor}_{\mathcal{D}(k[t_{n-1}])}(k[t_{n-1}],X)$ of $X$ is nontrivial. We find a corresponding non-zero morphism $\alpha\colon k[t_{n-1}][m]\rightarrow X$, such that the composite with $k[t_{n-1}][m+n-1]\rightarrow k[t_{n-1}][m]$ is zero. The morphism $\alpha$ thus induces a non-zero morphism $\alpha'\colon k\rightarrow X$, which is injective on homology. Taking the fiber of $\alpha'$ we obtain a module $\on{fib}(\alpha')$, whose homology has one dimension  less. We have by the induction assumption, that $\on{fib}(\alpha')\in \langle k\rangle$. From $X\simeq \on{cof}(\on{fib}(\alpha')\rightarrow k)$, it follows that $X$ also lies in $\langle k \rangle$. Since any object in $\langle k \rangle$ also lies in $\mathcal{D}^{\on{fin}}(k[t_{n-1}])$, we obtain $\langle k\rangle = \mathcal{D}^{\on{fin}}(k[t_{n-1}])$. This $\infty$-category is stable and also idempotent-complete, as having finite dimensional homology is a condition preserved under retracts. Consider the $k$-linear endomorphism dg-algebras $\on{End}_{\mathcal{D}(k[t_{n-1}])}(k)\simeq\on{End}_{\on{Fun}(S^n,\mathcal{D}(k))}(f^*(k))\simeq k\oplus k[-n]$. Since $f^*(k)$ is by \Cref{lem:cmpgen} a compact generator of $\on{Fun}(S^n,\mathcal{D}(k))^{\on{mnd}}$, we have $\on{Fun}(S^n,\mathcal{D}(k))^{\on{mnd}}\simeq \mathcal{D}(\on{End}_{\on{Fun}(S^{n},\mathcal{D}(k))}(f^*(k)))\simeq \mathcal{D}(\on{End}_{\mathcal{D}(k[t_{n-1}])}(k))$, see \cite[7.1.2.1]{HA}. The perfect derived $\infty$-category $\mathcal{D}^{\on{perf}}(\on{End}_{\mathcal{D}(k[t_{n-1}])}(k))$ (consisting of compact objects) is by \cite[7.2.4.1, 7.2.4.4]{HA} the smallest stable subcategory of $\on{Fun}(S^n,\mathcal{D}(k))^{\on{mnd}}\simeq \mathcal{D}(\on{End}_{\mathcal{D}(k[t_{n-1}])}(k))$ containing $k$ which is closed under retracts and thus equivalent to $\langle k\rangle = \mathcal{D}^{\on{fin}}(k[t_{n-1}])$. It follows that 
\[ \on{Fun}(S^n,\mathcal{D}(k))^{\on{mnd}}\simeq \mathcal{D}(\on{End}_{\mathcal{D}(k[t_{n-1}])}(k))\simeq \on{Ind}\mathcal{D}^{\on{perf}}(\on{End}_{\mathcal{D}(k[t_{n-1}])}(k))\simeq \on{Ind}\mathcal{D}^{\on{fin}}(k[t_{n-1}])\,,\]
concluding the proof. 
\end{proof}

We denote by $k[t_{n-1}^\pm]$ the dg-algebra of graded Laurent polynomials with generator in degree $|t_{n-1}|=n-1$. Note that if $i$ is even, then $k[t_{n-1}]$ and $k[t_{n-1}^\pm]$ are graded commutative dg-algebras, whereas if $i$ is odd, then $k[t_{n-1}]$ and $k[t_{n-1}^\pm]$ are not graded commutative, because $t_{n-1}^2\neq 0$. By \Cref{sodquotlem}, the following Lemma shows that $\mathcal{D}(k[t_{n-1}^\pm])$ arises as the Verdier quotient of $\mathcal{D}(k[t_{n-1}])$ by $\on{Ind}\mathcal{D}^{\on{fin}}(k[t_{n-1}])$, or by \Cref{fin=mndlem} equivalently as the Verdier quotient of $\on{Fun}(S^{n},\mathcal{D}(k))$ by $\on{Fun}(S^n,\mathcal{D}(k))^{\on{mnd}}$.

\begin{lemma}\label{lem:laurent}
Let $n\geq 1$. The $\infty$-category $\mathcal{D}(k[t_{n-1}])$ admits a semiorthogonal decomposition $(\mathcal{D}(k[t_{n-1}^\pm]),\on{Ind}\mathcal{D}^{\on{fin}}(k[t_{n-1}]))$.
\end{lemma}

\begin{proof}
For the proof, we show that $\mathcal{D}(k[t_{n-1}^\pm])=\on{Ind}\mathcal{D}^{\on{fin}}(k[t_{n-1}])^\perp$ is the stable subcategory of $\mathcal{D}(k[t_{n-1}])$ consisting of objects $b$, such that $\on{Mor}_{\mathcal{D}(k[t_{n-1}])}(a,b)\simeq 0$ for all $a\in \on{Ind}\mathcal{D}^{\on{fin}}(k[t_{n-1}])$. Using that the inclusion $\on{Ind}\mathcal{D}^{\on{fin}}(k[t_{n-1}])\subset \mathcal{D}(k[t_{n-1}])$ admits a right adjoint, the Lemma then follows from \cite[Prop.~2.2.4,~Prop.~2.3.2]{DKSS21}.

The objects of $\mathcal{D}(k[t_{n-1}])$ can be identified with dg-modules over the dg-algebra $k[t_{n-1}]$. Using this identification, we observe that $\mathcal{D}(k[t_{n-1}^\pm])\subset \mathcal{D}(k[t_{n-1}])$ is the full subcategory consisting of $k[t_{n-1}]$ dg-modules $M_\bullet$, such that $t_{n-1}\colon M_{i}\rightarrow M_{i+n-1}$ is an isomorphism of $k$-vector spaces for all $i\in \mathbb{Z}$. We have  
\[\on{Mor}_{\mathcal{D}(k[t_{n-1}])}(k[t_{n-1}][i],M_\bullet) \simeq M_{\bullet +i}\,.\]
Using that $k\simeq \on{cof}(k[t_{n-1}][n-1]\rightarrow k[t_{n-1}])$, we thus have 
\[ \on{Mor}_{\mathcal{D}(k[t_{n-1}])}(k[i],M_\bullet) \simeq \on{cof}(t_{n-1}\colon M_{\bullet+i}\rightarrow M_{\bullet+i+n-1})\in \mathcal{D}(k)\,.\]
We thus find that $t_{n-1}\colon M_{i}\rightarrow M_{i+n-1}$ is an isomorphism for all $i\in \mathbb{Z}$ if and only if $M_\bullet\in\on{Ind}\mathcal{D}^{\on{fin}}(k[t_{n-1}])^\perp$. This shows the desired equality
\[\mathcal{D}(k[t_{n-1}^\pm])=\on{Ind}\mathcal{D}^{\on{fin}}(k[t_{n-1}])^\perp\,,\]
concluding the proof.
\end{proof}

\begin{lemma}\label{lem:vectk}
There exists an equivalence of $1$-categories $\on{Vect}_k\simeq \on{ho}\mathcal{D}(k[t_1^\pm])$, where $\on{Vect}_k$ denotes the $1$-category of $k$-vector spaces and $\on{ho}\mathcal{D}(k[t_1^\pm])$ denotes the homotopy category of $\mathcal{D}(k[t_1^\pm])$.
\end{lemma}

\begin{proof}
Since any complex of $k$-vector spaces is quasi-isomorphic to its homology, any $1$-periodic complex with values in $k$ is equivalent to an object in the image of $N(\on{Vect}_k)\hookrightarrow \mathcal{D}(k)\xrightarrow{\mhyphen \otimes k[t_1^\pm]} \mathcal{D}(k[t_1^\pm])$. This functor is thus essentially surjective, and using that 
\[ \pi_0\on{Map}_{\mathcal{D}(k[t_1^\pm])}(k[t_1^\pm],k[t_1^\pm])\simeq \on{H}_0(k[t_1^\pm])\simeq k\,,\] 
one sees that this functor is also fully faithful on the level of homotopy categories. 
\end{proof}

In the triangulated homotopy category $\on{ho}\mathcal{D}(k[t_1^\pm])$, distinguished triangles take the form
\[ \on{ker}(\alpha)\oplus \on{coker}(\alpha)\xlongrightarrow{(\iota,0)} k[t_1^\pm]^{\oplus I}\xlongrightarrow{\alpha}k[t_1^\pm]^{\oplus J}\xrightarrow{(0,\pi)}\on{ker}(\alpha)\oplus \on{coker}(\alpha)\]
with $\iota$ being the kernel map and $\pi$ being the cokernel map and $I,J$ two sets.

Finally, we record the proof of \Cref{prop:formalneighborhood}.

\begin{proof}[Proof of \Cref{prop:formalneighborhood}.]
Consider the Verdier localization sequence 
\[ \D^{\on{perf}}(k[\epsilon_{-2}]/\epsilon_{-2}^2)\longrightarrow \D^{\on{perf}}(kQ) \longrightarrow \D^{\on{perf}}(k[t_1])\,,\]
where $k[\epsilon_{-2}]/\epsilon_{-2}^2$ denotes the graded dual numbers with $|\epsilon_{-2}|=-2$ and $kQ$ denotes the graded Kronecker quiver, with two vertices $1,2$ and arrows $a,b\colon 1\to 2$ in degrees $|a|=0,|b|=1$. The $k$-linear $\infty$-category $\D^{\on{perf}}(kQ)$ is smooth and proper. The $\infty$-category of perfect complexes on the formal punctured neighborhood of infinity of $\D^{\on{perf}}(k[t_1])$ is thus given by the (automatically idempotent complete) singularity category $\D^{\on{fin}}(k[\epsilon_{-2}]/\epsilon_{-2}^2)/\D^{\on{perf}}(k[\epsilon_{-2}]/\epsilon_{-2}^2)$. As shown for instance in the proof of Theorem 4.10 in \cite{Ami09}, this singularity category is equivalent to the the generalized cluster category of $\D^{\on{perf}}(k[t_1])$, which is by \Cref{lem:laurent} given by 
\[ \D^{\on{perf}}(k[t_1])/\D^{\on{fin}}(k[t_1])\simeq \D^{\on{perf}}(k[t_1^\pm])\,,\]
as desired. 
\end{proof}

\subsection{Relative Calabi--Yau structures}\label{sec2.4}

Let $k$ be a field and $\kappa=k$ or $\kappa=k[t_{2m}^\pm]$ the dg-algebra of graded Laurent polynomials, see \Cref{sec2.3}, for some $m\in \mathbb{N}$. In this section, we sketch the definition of a relative right $\kappa$-linear Calabi--Yau structure. The notion of relative Calabi--Yau structure was described in case $\kappa=k$ the setting of dg-categories in \cite{BD19} and generalized to arbitrary $R$-linear relative Calabi--Yau structures, with $R$ an $\mathbb{E}_\infty$-ring spectrum, on stable $\infty$-categories in \cite{Chr23}.

\begin{definition}\label{def:serre}
Let $\mathcal{A}$ be a compactly generated $\kappa$-linear $\infty$-category. A Serre functor on $\mathcal{A}$ is a $\kappa$-linear endofunctor $U\colon \mathcal{A}\rightarrow \mathcal{A}$, satisfying that there exists a natural equivalence of functors 
\[ \on{Mor}_{\mathcal{A}}(\mhyphen_1,\mhyphen_2)\simeq \on{Mor}_{\mathcal{A}}(\mhyphen_2,U(\mhyphen_1))^\ast\colon \mathcal{A}^{\on{c},\on{op}}\times \mathcal{A}^{\on{c}}\rightarrow \mathcal{D}(\kappa)\,.
\]
Here $\mathcal{A}^{\on{c}}$ denotes the full subcategory of $\mathcal{A}$ of compact objects and $(\mhyphen)^*=\on{Mor}_{\mathcal{D}(\kappa)}(\mhyphen,\kappa)\colon \mathcal{D}(\kappa)^{\on{op}}\rightarrow \mathcal{D}(\kappa)$.
\end{definition}

Every compactly generated, proper $\kappa$-linear $\infty$-category $\mathcal{A}$ admits a Serre functor $U$. This functor is obtained by identifying the bimodule right dual (i.e.~the right adjoint) of the evaluation functor $\on{Mor}_\D(\mhyphen,\mhyphen)\colon \A^\vee\otimes \A\to \D(\kappa)$ with an endofunctor $\on{id}_{\mathcal{A}}^\ast=U$ of $\A$. If $\A$ is firthermore smooth, then the evaluation functor also admits a left dual, which we can identify with a functor $\on{id}_\A^!$. In this case, the functor $\on{id}_\A^!$ is inverse to $\on{id}_\A^*$. We refer to \cite{Chr23} for background. A weak right $n$-Calabi--Yau structure on such an $\infty$-category $\mathcal{A}$ consists of a natural equivalence $U\simeq \on{id}_{\mathcal{A}}[n]$. Such an equivalence arises from a $\kappa$-linear dual Hochschild homology class 
\[ \sigma\colon \kappa[n] \to \on{HH}(\mathcal{A})^* \simeq \on{Mor}_{\on{Lin}_{\kappa}(\mathcal{A},\mathcal{A})}(\on{id}_{\mathcal{A}},\on{id}_{\mathcal{A}}^\ast)\,,\]
where $\on{Lin}_{\kappa}(\mathcal{A},\mathcal{A})$ denotes the $\kappa$-linear $\infty$-category of $\kappa$-linear endofunctors of $\A$. A right $n$-Calabi--Yau structure on $\A$ consists of a lift of a weak right Calabi--Yau structure $\sigma$ to a dual cyclic homology class $\eta \colon \kappa[n] \to \on{HH}(\A)_{S^1}^*$.

Consider a $\kappa$-linear functor $G\colon \mathcal{B}\rightarrow \mathcal{A}$ between compactly generated, proper $\kappa$-linear $\infty$-categories that admits a $\kappa$-linear right adjoint $H$. The functor $G$ induces a morphism $\on{HH}(G)^*\colon \on{HH}(\mathcal{A})^*\rightarrow \on{HH}(B)^*$ on the dual Hochschild homology and we define $\on{HH}(\mathcal{B},\mathcal{A})^*=\on{cof}(\on{HH}(G)^*)$. The relative dual cyclic homology $\on{HH}(\mathcal{B},\mathcal{A})^*_{S^1}$ is defined similarly. A relative Hochschild class $\sigma \colon \kappa[n]\to \on{HH}(\mathcal{B},\mathcal{A})^*$ gives rise to a diagram
\[
\begin{tikzcd}
{\on{id}_{\mathcal{B}}} \arrow[r, "{\on{u}}"] & {HG} \arrow[d]               &                       \\
                                         & {H\on{id}_{\mathcal{A}}^*G[1-n]} \arrow[r, "\tilde{\on{cu}}"] & {\on{id}_{\mathcal{B}}^*[1-n]}
\end{tikzcd}
\]
where $\on{u}$ is the unit of $G\dashv H$ and $\tilde{\on{cu}}$ a kind of counit, see \cite{BD21,Chr23}, together with a choice of null-homotopy of the composite morphism.

\begin{definition}\label{def:cystr}
A right $n$-Calabi--Yau structure on the functor $G\colon\mathcal{B}\rightarrow \mathcal{A}$ consists of a class $\eta \colon \kappa[n]\to \on{HH}(\mathcal{B},\mathcal{A})^*_{S^1}$, whose composite with $\on{HH}(\mathcal{B},\mathcal{A})^*_{S^1}\to \on{HH}(\mathcal{B},\mathcal{A})^*$ induces a diagram with horizontal fiber and cofiber sequences
\[\begin{tikzcd}
{\on{id}_{\mathcal{B}}} \arrow[r] \arrow[d] & {HG} \arrow[d] \arrow[r]       & \on{cof} \arrow[d]    \\
\on{fib} \arrow[r]                               & {H\on{id}_{\mathcal{A}}^*G[1-n]} \arrow[r] & {\on{id}_{\mathcal{B}}^*[1-n]}
\end{tikzcd}
\]
whose vertical morphisms are equivalences. 
\end{definition}

Informally, relative right Calabi--Yau structures can be understood as follows. There is a morphism $U[-n]\simeq \on{id}_{\mathcal{B}}^*[-n]\rightarrow \on{id}_{\mathcal{B}}$ of endofunctors, with $U$ the Serre functor, which is not an equivalence but whose fiber is equivalent to $HG[-1]$. Thus, a relative Calabi--Yau structure is a description of how much the Serre functor differs from a suspension of the identity, in terms of the adjunction $G\dashv H$.

\section{Perverse schobers on surfaces}

\subsection{Marked surfaces and ribbon graphs}

\begin{definition}\label{def:surf}
A surface is a smooth compact oriented surface ${\bf S}$ with non-empty boundary $\partial {\bf S}$ and interior ${\bf S}^\circ$. Note that this implies that the boundary of ${\bf S}$ is a disjoint union of circles.

A marked surface consists of a surface ${\bf S}$ and a finite set $M\subset \partial {\bf S}$ of marked points on the boundary of ${\bf S}$ such that each connected component of $\partial {\bf S}$ contains at least one marked point. We always assume that ${\bf S}$ is not the monogon or digon.
\end{definition}

\begin{definition}
\begin{itemize}
\item A graph ${\bf G}$ consists of two finite sets ${\bf G}_0$ of vertices and $\on{H}$ of halfedges together with an involution $\tau\colon\on{H}\rightarrow \on{H}$ and a map $\sigma\colon\on{H}\rightarrow {\bf G}_0$.
\item Let ${\bf G}$ be a graph. We denote the set of orbits of $\tau$ by ${\bf G}_1$. The elements of ${\bf G}_1$ are called the edges of ${\bf G}$. An edge is called internal if the orbit contains two elements and called external if the orbit contains a single element. An internal edge is called a loop at $v\in {\bf G}_0$ if it consists of two halfedges both being mapped under $\sigma$ to $v$. The set of internal edges of ${\bf G}$ is denoted by ${\bf G}_1^\circ$ and the set of external edges is denoted by ${\bf G}_1^\partial$.
\item A ribbon graph consists of a graph ${\bf G}$ together with a choice of a cyclic order on the set $\on{H}(v)$ of halfedges incident to $v$ for each $v\in{\bf G}_0$.
\end{itemize}
\end{definition}

\begin{definition}
Let ${\bf G}$ be a graph. We denote by $\on{Exit}({\bf G})$ the simplicial set given by the nerve of the $1$-category with 
\begin{itemize}
\item objects given by the set ${\bf G}_0 \amalg {\bf G}_1$ and 
\item non-identity morphisms of the form $v\rightarrow e$ with $v\in {\bf G}_0$ a vertex and $e\in {\bf G}_1$ an edge incident to $v$. If $e$ is a loop at $v$, then there are two morphisms $v\rightarrow e$.
\end{itemize}
We call $\on{Exit}({\bf G})$ the exit path category of ${\bf G}$.

The geometric realization $|{\bf G}|$ of the graph ${\bf G}$ is defined as the geometric realization $|\on{Exit}({\bf G})|$ of the simplicial set $\on{Exit}({\bf G})$. 
\end{definition}

\begin{remark}
Let ${\bf G}$ be a graph and ${\bf S}$ an oriented surface. An embedding of $|{\bf G}|$ into ${\bf S}$ determines a canonical ribbon graph structure on ${\bf G}$ by requiring the halfedges at any vertex to be ordered in the counterclockwise direction.
\end{remark}

\begin{definition}\label{def:sg}
Let ${\bf S}$ be a marked surface. We call a graph ${\bf G}$ together with an embedding $i\colon |{\bf G}|\subset {\bf S}\backslash M$ a spanning graph for ${\bf S}$, if $i$ is a homotopy equivalence and the restriction $i^{-1}(\partial {\bf S}\backslash M) \rightarrow \partial {\bf S}\backslash M$ of $i$ is also a homotopy equivalence. We consider spanning graphs as endowed with the canonical ribbon graph structure arising from the embedding into the oriented surface ${\bf S}$.  
\end{definition}

\begin{remark}
Any \idtr{} of a marked surface ${\bf S}$ arises as the dual graph of an ideal triangulation of ${\bf S}$ in the sense of \Cref{def:idealtr} and this defines a bijection between sets of suitable homotopy classes of \idtr{}s and ideal triangulations.
\end{remark}

\subsection{Perverse schobers on marked surfaces}\label{sec3.2}

In this section, we recall the definition of perverse schober parametrized by a ribbon graph. For more background on perverse schobers on marked surfaces and their interpretation as categorified perverse sheaves, we refer to \cite[Sections 3 and 4]{Chr22}.

We first briefly recall the definition of spherical functor or a spherical adjunction. Let $F\colon\mathcal{A}\leftrightarrow \mathcal{B}\noloc G$ be an adjunction between stable $\infty$-categories with unit $\on{u}\colon \on{id}_{\mathcal{A}}\rightarrow GF$ and counit $\on{cu}\colon FG\rightarrow \on{id}_{\mathcal{B}}$. Taking the cofiber of the unit $\on{u}$ in the stable $\infty$-category $\on{Fun}(\mathcal{A},\mathcal{A})$ of endofunctors, we obtain a functor $T_{\mathcal{A}}\colon \mathcal{A}\rightarrow \mathcal{A}$ called the twist functor. Similarly, the cotwist functor $T_{\mathcal{B}}\colon \mathcal{B}\rightarrow \mathcal{B}$ is defines as the fiber of the counit. The adjunction $F\dashv G$ is called spherical if both $T_{\mathcal{A}}$ and $T_{\mathcal{B}}$ are autoequivalences. In this case, we also call the functors $F$ and $G$ spherical. For treatments of spherical adjunctions in the setting of stable $\infty$-categories, we refer to \cite{DKSS21,Chr20}.

We turn to describing the local model of a perverse schober arising from a spherical functor. Let ${\bf G}$ be a ribbon graph and $v$ a vertex of ${\bf G}$ of valency $m$. The objects of the under category $\on{Exit}({\bf G})_{v/}$ are the vertex $v$ and its incident halfedges. We choose a labeling of these halfedges by $1,\dots,m$, compatibly with their given cyclic order. Given a spherical functor $F_v\colon \mathcal{V}_v\rightarrow \mathcal{N}$, we define a functor $\mathcal{G}_v(F_v)\colon\on{Exit}({\bf G})_{/v}\rightarrow \on{St}$ as follows. Consider the diagram $D\colon [m-1]\rightarrow \on{Set}_\Delta$ given by
\begin{equation}\label{eq:VNdiag} \mathcal{V}_v\xrightarrow{F_v} \underbrace{\mathcal{N}\xrightarrow{\on{id}} \dots \xrightarrow{\on{id}} \mathcal{N}}_{(m-1)\text{-many}}\,.\end{equation}
Let $p\colon \Gamma(D)\rightarrow \Delta^{m-1}\simeq N([m-1])$ be the (covariant) Grothendieck construction of $D$, see \cite[3.2.5.2]{HA}, where it is called the relative nerve. The Grothendieck construction is an explicit model for the coCartesian fibration classified by the functor $D$.

\begin{itemize}
\item We set $\mathcal{G}_v(F_v)(v)=\on{Fun}_{\Delta^{m-1}}(\Delta^{m-1},\Gamma(D))$ to be the stable $\infty$-category of sections of $p$. We also denote $\mathcal{V}_{F_v}^m=\mathcal{G}_v(F_v)(v)$. We note that if $\mathcal{V}_v\simeq 0$, then $\mathcal{V}^m_{F_v}\simeq \on{Fun}(\Delta^{m-2},\mathcal{N})$, this should be considered as the $\infty$-category of representations of the $A_{n-1}$-quiver in $\mathcal{N}$. In general, $\mathcal{V}^m_{F_v}$ is a model for the lax limit of the diagram \eqref{eq:VNdiag} in the $(\infty,2)$-category of stable $\infty$-categories.
\item We set $\mathcal{G}_v(F_v)(i)=\mathcal{N}$ for all $1,\dots,m$.
\item We denote $\varrho_i\coloneqq \mathcal{G}_v(F_v)(v\rightarrow i)\colon \mathcal{V}^m_{F_v}\rightarrow \mathcal{N}$. We set \[\varrho_1=\on{ev}_{m-1}\] to be the evaluation functor of the sections of $p$ at $m-1\in \Delta^{m-1}$. For $1\leq i<m$, we set $\varrho_{i+1}=\on{ladj}(\on{ladj}(\varrho_i))$. For a detailed description (and thus also a proof of existence) of these functors, we refer to \cite[Section 3]{Chr22}.
\end{itemize} 

In the case $m=1$, we have $\on{Exit}({\bf G})_{/v}\simeq \Delta^1$ and the functor $\mathscr{G}_v(F_v)$ is simply the functor $F_v$. One can further show that making a different choice of total order on the halfedges of $v$ changes the functor $\mathcal{G}_v(F_v)$ only by composition with an autoequivalence, see \cite[Prop.~3.11]{Chr22}.

\begin{definition}\label{def:ps}
Let ${\bf G}$ be a ribbon graph. A ${\bf G}$-parametrized perverse schober is a functor $\mathcal{F}\colon \on{Exit}({\bf G})\rightarrow \mathcal{P}r^L_{\on{St}}$ such that for each vertex $v$ of $\mathcal{T}$ there exists a spherical functor $F_v\colon \mathcal{V}_v\rightarrow \mathcal{N}$, satisfying that the composite of $\on{Exit}({\bf G})_{v/}\rightarrow \on{Exit}({\bf G})$ with $\mathcal{F}$ is equivalent to $\mathcal{G}_{v}(F_v)$ in $\on{Fun}(\on{Exit}({\bf G})_{v/}\,\on{St})$. 

Given a vertex $v$ of ${\bf G}$, the $\infty$-category $\mathcal{V}_v$ is called the $\infty$-category of vanishing cycles at $v$. If $\mathcal{V}_v\not\simeq 0$, we call $v$ a singularity of $\mathcal{F}$. We call $F_v$ the spherical functor describing or underlying $\mathcal{F}$ at $v$ (this functor is only determined up to equivalence and composition with autoequivalences).
\end{definition}

Assuming that the graph ${\bf G}$ is connected, a ${\bf G}$-parametrized perverse schober assigns to each edge of ${\bf G}$ an equivalent $\infty$-category $\mathcal{N}$, called the generic stalk of the perverse schober, or also the $\infty$-category of nearby cycles.

\begin{remark}
We restrict in \Cref{def:ps} to perverse schobers taking values in stable and presentable $\infty$-categories and colimit preserving functors, but one could also consider perverse schobers valued in stable $\infty$-categories. Any perverse schober taking values in presentable $\infty$-categories also defines a functor to $\mathcal{P}r^R_{\on{St}}$, because the functors $\varrho_i$ defined above admits repeated left and right adjoints, see \cite[Rem.~4.11]{Chr21b}.
\end{remark}

\begin{definition}\label{def:pscat}
The $\infty$-categories of left/right ${\bf G}$-parametrized perverse schobers $\mathfrak{P}^{L}({\bf G})$ and $\mathfrak{P}^R({\bf G})$ are defined as the full subcategories of $\on{Fun}(\on{Exit}({\bf G}),\mathcal{P}r^L_{\on{St}})$ and $\on{Fun}(\on{Exit}({\bf G}),\mathcal{P}r^R_{\on{St}})$, respectively, spanned by ${\bf G}$-parametrized perverse schobers.
\end{definition}

\begin{definition}\label{def:glsec}
Let ${\bf G}$ be a ribbon graph and $\mathcal{F}$ a ${\bf G}$-parametrized perverse schober. 
\begin{itemize}
\item The $\infty$-category of global sections $\mathcal{H}({\bf G},\mathcal{F})\in \mathcal{P}r^L_{\on{St}}$ of $\mathcal{F}$ is defined as the limit of $\mathcal{F}$. 
\item The $\infty$-category of local sections $\mathcal{L}({\bf G},\mathcal{F})=\on{Fun}_{\on{Exit}({\bf G})}(\on{Exit}({\bf G}),\Gamma(\mathcal{F}))\in \mathcal{P}r^L_{\on{St}}$ of $\mathcal{F}$ is defined as the $\infty$-category of sections of the Grothendieck construction $p\colon \Gamma(\mathcal{F})\rightarrow \on{Exit}({\bf G})$. 
\end{itemize}
\end{definition}

\begin{remark}
As a model for the limit $\mathcal{H}({\bf G},\mathcal{F})$ of a ${\bf G}$-parametrized perverse schober $\mathcal{F}$, we usually consider the full subcategory of the $\infty$-category of sections $\mathcal{L}({\bf G},\mathcal{F})$ of the  Grothendieck construction $p\colon \Gamma(\mathcal{F})\rightarrow \on{Exit}({\bf G})$ of $\mathcal{F}$ spanned by coCartesian sections. This description of the limit is proven in \cite[7.4.1.10]{Ker}. A section $s\colon \on{Exit}({\bf G})\rightarrow \Gamma(\mathcal{F})$ of $p$ is called coCartesian if $s(\alpha)$ is a coCartesian edge in the dual sense of \cite[2.4.1.1]{HTT} for each $1$-simplex $\alpha$ in $\on{Exit}({\bf G})$.
\end{remark}

Let $\mathcal{F}$ be a ${\bf G}$-parametrized perverse schober. Given an edge $e$ of ${\bf G}$, we denote by $\on{ev}_e\colon \mathcal{H}({\bf G},\mathcal{F})\rightarrow \mathcal{F}(e)$ the evaluation functor, which maps a coCartesian section $s$ of $p$ to its value $s(e)$ at $e$. Consider the product of the evaluation functors at the external edges 
\[ \prod_{e\in {\bf G}_1^\partial} \on{ev}_e\colon \mathcal{H}({\bf G},\mathcal{F})\rightarrow  \prod_{e\in {\bf G}_1^\partial} \mathcal{F}(e)\,.\]
The functor $\prod_{e\in {\bf G}_1^\partial} \on{ev}_e$ preserves all limits and colimits by \cite[5.1.2.3, 4.3.1.10, 4.3.1.16]{HTT}. By the $\infty$-categorical adjoint functor theorem, $\prod_{e\in {\bf G}_1^\partial}\on{ev}_e$ thus admits a left adjoint.

\begin{definition}\label{def:cupfunctor}
The functor 
\[\partial \mathcal{F}\colon \prod_{e\in {\bf G}_1^\partial} \mathcal{F}(e)\longrightarrow \mathcal{H}({\bf G},\mathcal{F})\]  
is defined as the left adjoint of $\prod_{e\in {\bf G}_1^\partial}\on{ev}_e$.
\end{definition}

\begin{remark}
In the study of partially wrapped Fukaya categories, particularly Fukaya-Seidel categories, the functor $\partial \mathcal{F}$ is often called the Orlov functor or cup functor and $\prod_{e\in {\bf G}_1^\partial}\on{ev}_e$ is called the cap functor, see for instance \cite{Syl19} for background. 
\end{remark}

\subsection{Semiorthogonal decompositions of perverse schobers}\label{sec3.3} 

Let ${\bf G}$ be a ribbon graph. We note that a morphism $\alpha\colon \mathcal{F}\rightarrow \mathcal{G}$ of ${\bf G}$-parametrized perverse schobers in one of the $\infty$-categories $\mathfrak{P}^{L}({\bf G})$ or $\mathfrak{P}^{R}({\bf G})$, see \Cref{def:pscat}, is simply a natural transformation $\Delta^1\times \on{Exit}({\bf G})\rightarrow \mathcal{P}r^{L}_{\on{St}}$ or $\Delta^1\times \on{Exit}({\bf G})\rightarrow \mathcal{P}r^{R}_{\on{St}}$ between the diagrams defining $\mathcal{F}$ and $\mathcal{G}$.

\begin{definition}
Let ${\bf G}$ be a ribbon graph and let $\alpha\colon \mathcal{F}\rightarrow \mathcal{G}$ be a morphism of ${\bf G}$-parametrized perverse schobers in $\mathfrak{P}^L({\bf G})$ or $\mathfrak{P}^R({\bf G})$.
\begin{enumerate}
\item We call $\alpha$ an inclusion of perverse schobers in $\mathfrak{P}^{L}({\bf G})$ or $\mathfrak{P}^R({\bf G})$ if $\alpha(x)\colon \mathcal{F}(x)\rightarrow \mathcal{G}(x)$ is the inclusion of a stable subcategory (in particular fully faithful) for all $x\in \on{Exit}({\bf G})$.
\item Suppose that $\alpha$ is an inclusion in $\mathfrak{P}^L({\bf G})$. We call $\alpha$ right admissible if there exists a morphism $\beta\colon \mathcal{G}\rightarrow \mathcal{F}$ in $\mathfrak{P}^{R}({\bf G})$ such that $\beta(x)$ is right adjoint to $\alpha(x)$ for all $x\in \on{Exit}({\bf G})$.
\item Suppose that $\alpha$ is an inclusion in $\mathfrak{P}^R({\bf G})$. We call $\alpha$ left admissible if there exists a morphism $\beta\colon \mathcal{G}\rightarrow \mathcal{F}$ in $\mathfrak{P}^{L}({\bf G})$ such that $\beta(x)$ is left adjoint to $\alpha(x)$ for all $x\in \on{Exit}({\bf G})$.
\end{enumerate}
\end{definition}

\begin{remark}\label{rem:adminc}
Let $\alpha\colon \mathcal{F}\rightarrow \mathcal{G}$ be an inclusion of ${\bf G}$-parametrized perverse schobers. Spelling out the condition that $\alpha$ is left admissible yields the requirement that for each morphism $v\rightarrow e$ in $\on{Exit}({\bf G})$ the diagram
\begin{equation}\label{faleq}
\begin{tikzcd}
\mathcal{F}(v) \arrow[d, "\mathcal{F}(v\rightarrow e)"'] \arrow[r, "\alpha(v)"] & \mathcal{G}(v) \arrow[d, "\mathcal{G}(v\rightarrow e)"] \\
\mathcal{F}(e) \arrow[r, "\alpha(e)"]                                           & \mathcal{G}(e)                                         
\end{tikzcd}
\end{equation}
is left adjointable, meaning that the diagram
\[
\begin{tikzcd}
\mathcal{F}(v) \arrow[d, "\mathcal{F}(v\rightarrow e)"'] & \mathcal{G}(v) \arrow[d, "\mathcal{G}(v\rightarrow e)"] \arrow[l, "\beta(v)"'] \\
\mathcal{F}(e)                                           & \mathcal{G}(e) \arrow[l, "\beta(e)"']                                         
\end{tikzcd}
\]
commutes, with $\beta(x)$ left adjoint to $\alpha(x)$ for $x=v,e$. Analogously, the condition that $\alpha$ is right admissible is equivalent to the right adjointability of the diagram \eqref{faleq}. Adjointability of a commutative square is also called the \textit{Beck-Chevalley condition}. 
\end{remark}

\begin{definition}\label{def:SODschobers}
Let ${\bf G}$ be a ribbon graph. A semiorthogonal decomposition $\{\mathcal{F}_1,\mathcal{F}_2\}$ of a ${\bf G}$-parametrized perverse schober $\mathcal{G}$ consists of
\begin{itemize} 
\item 
an inclusion $\alpha_1\colon\mathcal{F}_1\rightarrow \mathcal{G}$ in $\mathfrak{P}^R({\bf G})$ and 
\item an inclusion $\alpha_2\colon\mathcal{F}_2\rightarrow \mathcal{G}$ in $\mathfrak{P}^L({\bf G})$, satisfying that
\item $\alpha_2$ is right admissible and that
\item the cofiber of $\alpha_2$ in $\on{Fun}(\on{Exit}({\bf G}),\mathcal{P}r^L_{\on{St}})$ is given by $\mathcal{F}_1$, with cofiber morphism $\beta\colon \mathcal{G}\to \mathcal{F}_1$ pointwise left adjoint to $\alpha_1$, thus exhibiting $\alpha_1$ as left admissible.
\end{itemize}
\end{definition}

\begin{remark}\label{rem:sod1}
Consider a semiorthogonal decomposition $\{\mathcal{F}_1,\mathcal{F}_2\}$ of the ${\bf G}$-parametrized perverse schober $\mathcal{G}$.
\begin{enumerate}
\item Then for each $x\in \on{Exit}({\bf G})$, the $\infty$-category $\mathcal{G}(x)$ admits a semiorthogonal decomposition $(\mathcal{F}_1(x),\mathcal{F}_2(x))$.
\item The condition that $\mathcal{F}_1$ is the cofiber of $\alpha_2$ in $\on{Fun}(\on{Exit}({\bf G}),\mathcal{P}r^L_{\on{St}})$ is equivalent to the condition that $\mathcal{F}_2$ is the fiber of the pointwise right adjoint $\beta_1\colon\mathcal{G}\rightarrow \mathcal{F}_1$ of $\alpha_1$ in $\on{Fun}(\on{Exit}({\bf G}),\mathcal{P}r^R_{\on{St}})$. This follows from the fact that limits and colimits in functor categories are computed pointwise, see \cite[5.1.2.3]{HTT}. Since limits commute with limits, we find that there exists a fiber sequence in $\mathcal{P}r^L_{\on{St}}$
\[ \mathcal{H}({\bf G},\mathcal{F}_2)\xrightarrow{i} \mathcal{H}({\bf G},\mathcal{G})\xrightarrow{\pi} \mathcal{H}({\bf G},\mathcal{F}_1)\,.\]
Passing to right adjoints yields a cofiber sequence in $\mathcal{P}^R_{\on{St}}$, showing by \Cref{sodquotlem} that $\mathcal{H}({\bf G},\mathcal{G})$ admits a semiorthogonal decomposition \[\big(\mathcal{H}({\bf G},\mathcal{F}_1),\mathcal{H}({\bf G},\mathcal{F}_2)\big)\,.\]
\end{enumerate}
\end{remark}

\begin{definition}\label{def:mndschob}
Let $\mathcal{F}$ be a ${\bf G}$-parametrized perverse schober. Recall that given a vertex $v$ of ${\bf G}$, $\mathcal{F}$ is described near $v$ by a spherical functor $F\colon \mathcal{V}_v\to \mathcal{N}$, where $\mathcal{V}_v$ is called the $\infty$-category of vanishing cycles and $\mathcal{N}$ the $\infty$-category of nearby cycles. We call
\begin{itemize}
\item $\mathcal{F}$ vanishing-monadic, if at each vertex $v$ of ${\bf G}$ the spherical functor $F_v\colon\mathcal{V}_v\rightarrow \mathcal{N}$ describing $\mathcal{F}$ at $v$ is a monadic functor. 
\item $\mathcal{F}$ nearby-monadic, if at each vertex $v$ of ${\bf G}$ the right adjoint of the spherical functor $F_v$ describing $\mathcal{F}$ at $v$ is a monadic functor.
\item $\mathcal{F}$ locally constant if $\mathcal{F}$ has no singularities.
\end{itemize}
\end{definition}

\begin{remark}
A spherical functor between presentable $\infty$-categories is monadic if and only if it is comonadic if and only if it is conservative, as follows from the Lurie--Barr--Beck theorem \cite[4.7.3.5]{HA}.
\end{remark}

\section{Generalized cluster categories of surfaces and perverse schobers}\label{sec4} 

We fix a field $k$ and a marked surface ${\bf S}$ with a \idtr{} $\mathcal{T}$. Theorem 6.1 in \cite{Chr22} constructs a $\mathcal{T}$-parametrized perverse schober $\mathcal{F}_\mathcal{T}$, whose $\infty$-category of global sections $\mathcal{H}(\mathcal{T},\mathcal{F}_\mathcal{T})$ is equivalent to the unbounded derived $\infty$-category of the relative Ginzburg algebra $\mathscr{G}_\mathcal{T}$ associated with $\mathcal{T}$. The potential of the underlying quiver consists of $3$-cycles, one for each triangle of the dual ideal triangulation of $\mathcal{T}$. In \Cref{sec4.1,sec4.2}, we describe the generalized cluster category associated to $\mathscr{G}_\mathcal{T}$ in terms of the global sections of a component of a semiorthogonal decomposition of $\mathcal{F}_\mathcal{T}$. Using this, we are able to deduce in \Cref{sec4.3} that the generalized cluster category is equivalent to the $1$-periodic topological Fukaya category.

\subsection{A semiorthogonal decomposition of perverse schobers}\label{sec4.1}

Consider the $\mathcal{T}$-parametrized perverse schober $\mathcal{F}_\mathcal{T}$ from \cite[Theorem 6.1]{Chr22}. The spherical adjunction underlying $\mathcal{F}_\mathcal{T}$ at any vertex of $\mathcal{T}$ is given by
\[ f^*\colon\mathcal{D}(k)\longleftrightarrow \on{Fun}(S^2,\mathcal{D}(k))\noloc f_*\]
where $f^*$ is the pullback functor along $S^2\rightarrow \ast$. There is a $k$-linear equivalence $\on{Fun}(S^2,\mathcal{D}(k))\simeq \mathcal{D}(k[t_1])$, where $k[t_1]$ is the graded polynomial algebra with $|t_1|=1$, under which the functor $f^*$ is identified with the pullback functor $\phi^*$ along $\phi\colon k[t_1]\xrightarrow{t_1\mapsto 0}k$, see \cite[Prop.~5.5]{Chr22}. In particular, the functor $f^*$ is conservative and the adjunction $f^*\dashv f_*$ thus comonadic. The adjunction $f^*\dashv f_*$ is however not monadic. The Eilenberg-Moore $\infty$-category $\on{Fun}(S^2,\mathcal{D}(k))^{\on{mnd}}$ of the monad $f_*f^*$ can be identified with the full subcategory $\on{Ind}\on{Fun}(S^2,\mathcal{D}^{\on{perf}}(k))\subset \on{Fun}(S^2,\mathcal{D}(k))$, see \Cref{fin=mndlem}. Using this observation, we can obtain a nearby-monadic and vanishing-monadic $\mathcal{T}$-parametrized perverse schober $\mathcal{F}^{\on{mnd}}_{\mathcal{T}}$ by restricting in the construction of $\mathcal{F}_{\mathcal{T}}$ in \cite{Chr21b} the spherical adjunction $f^*\dashv f_*$ to the spherical monadic and comonadic adjunction
\[ (f^*)^{\on{Ind-fin}}\colon \mathcal{D}(k)\longleftrightarrow \on{Ind}\on{Fun}(S^2,\mathcal{D}^{\on{perf}}(k))\noloc f_*^{\on{Ind-fin}}\,.\]
We remark that to give this definition of $\mathcal{F}_{\mathcal{T}}^{\on{mnd}}$, we also use that in the construction of $\mathcal{F}_{\mathcal{T}}$ all appearing autoequivalences of $\on{Fun}(S^2,\mathcal{D}(k))$ preserve the subcategory $\on{Ind}\on{Fun}(S^2,\mathcal{D}^{\on{perf}}(k))\subset \on{Fun}(S^2,\mathcal{D}(k))$. 

There is an apparent inclusion of perverse schobers $\mathcal{F}_{\mathcal{T}}^{\on{mnd}}\rightarrow \mathcal{F}_\mathcal{T}$.

\begin{lemma}\label{lem:incadm}
The inclusion $\mathcal{F}^{\on{mnd}}_\mathcal{T}\rightarrow \mathcal{F}_\mathcal{T}$ is right admissible. 
\end{lemma}

\begin{proof}
By \Cref{rem:adminc}, we need to show that for $1\leq i\leq 3$, the diagram
\begin{equation}\label{eq:2adjbldiag}
\begin{tikzcd}
\mathcal{V}^3_{(f^*)^{\on{Ind-fin}}} \arrow[r, hook, "i_\mathcal{V}"] \arrow[d, "{\varrho_i}"'] & \mathcal{V}^3_{f^*} \arrow[d, "{\varrho_i}"] \\
{\on{Ind}\on{Fun}(S^2,\mathcal{D}^{\on{perf}}(k))} \arrow[r, hook, "i_{\on{mnd}}"]                      & {\on{Fun}(S^2,\mathcal{D}(k))}                     
\end{tikzcd}
\end{equation}
is right adjointable. We only consider the case $i=1$, the other cases can be treated analogously or also follow using \cite[Prop.~3.11]{Chr22} from the $i=1$ case.

It is straightforward to see that the following diagram commutes,
\[
\begin{tikzcd}
\D(k) \arrow[d, "(f^*)^{\on{Ind-fin}}"'] \arrow[r, "\on{id}"]                                      & \D(k) \arrow[d, "f^*"] \arrow[r, "\on{id}"]                                     & \D(k) \arrow[d, "(f^*)^{\on{Ind-fin}}"]                                 \\
{\on{Ind}\on{Fun}(S^2,\mathcal{D}^{\on{perf}}(k))} \arrow[r, "i_{\on{mnd}}"] \arrow[d, "\on{id}"'] & {\on{Fun}(S^2,\mathcal{D}(k))} \arrow[d, "\on{id}"] \arrow[r, "\pi_{\on{mnd}}"] & {\on{Ind}\on{Fun}(S^2,\mathcal{D}^{\on{perf}}(k))} \arrow[d, "\on{id}"] \\
{\on{Ind}\on{Fun}(S^2,\mathcal{D}^{\on{perf}}(k))} \arrow[r, "i_{\on{mnd}}"]                       & {\on{Fun}(S^2,\mathcal{D}(k))} \arrow[r, "\pi_{\on{mnd}}"]                      & {\on{Ind}\on{Fun}(S^2,\mathcal{D}^{\on{perf}}(k))}                     
\end{tikzcd}
\]
where $\pi_{\on{mnd}}$ denotes the right adjoint of $i_{\on{mnd}}$. The right adjoint of the functor $i_{\mathcal{V}}$ between lax limits is obtained by passing componentwise to the right adjoint, see \cite[Remark A.2.5]{CDW23}, and thus acts on the third component via $\pi_{\on{mnd}}$. This shows the desired right adjointability.
\end{proof}

\begin{definition}\label{def:gclst}
We define the functor $\mathcal{F}_{\mathcal{T}}^{\on{clst}}\colon \on{Exit}(\mathcal{T})\rightarrow \on{St}$ as the cofiber of the inclusion $\mathcal{F}^{\on{mnd}}_{\mathcal{T}}\hookrightarrow \mathcal{F}_{\mathcal{T}}$ in $\on{Fun}(\on{Exit}({\mathcal{T}}),\on{St})$.
\end{definition}

\begin{lemma}\label{lem:locadj}
The functor $\mathcal{F}_{\mathcal{T}}^{\on{clst}}$ is a $\mathcal{T}$-parametrized perverse schober. At each vertex of $\mathcal{T}$, it is described by the trivial spherical adjunction 
\[ 0\leftrightarrow \mathcal{D}(k[t_1^\pm])\,.\]
\end{lemma}

\begin{proof}
Using \Cref{lem:laurent}, this is straightforward to check by using that pushouts in the functor category $\on{Fun}(\on{Exit}(\mathcal{T}),\on{St})$ are computed pointwise in $\on{Exit}(\mathcal{T})$.
\end{proof}

\begin{remark}
The abbreviation $\on{clst}$ stands for 'cluster', or rearranging the letters to $\on{lcst}$, it stands for 'locally constant'. Note that the perverse schober $\mathcal{F}_{\trivalent}^{\on{clst}}$ is locally constant in the sense of \Cref{def:mndschob}, meaning that $\mathcal{F}_{\trivalent}^{\on{clst}}$ has no singularities.
\end{remark}

\begin{proposition}\label{prop:sodG}
The pair $\{\mathcal{F}_{\mathcal{T}}^{\on{clst}},\mathcal{F}^{\on{mnd}}_{\mathcal{T}}\}$ forms a semiorthogonal decomposition of the $\mathcal{T}$-parametrized perverse schober $\mathcal{F}_{\mathcal{T}}$.
\end{proposition}

\begin{proof} 
The inclusion $\alpha\colon \mathcal{F}_{\mathcal{T}}^{\on{mnd}}\rightarrow \mathcal{F}_{\mathcal{T}}$ is by \Cref{lem:incadm} right admissible. A similar argument as in the proof of \Cref{lem:incadm} further shows that the cofiber morphism $\mathcal{F}_{\mathcal{T}}\rightarrow \mathcal{F}_{\mathcal{T}}^{\on{clst}}$ arising from the definition of $\mathcal{F}_{\mathcal{T}}^{\on{clst}}$ is the pointwise left adjoint of a left admissible inclusion $\mathcal{F}_{\mathcal{T}}^{\on{clst}}\rightarrow \mathcal{F}_{\mathcal{T}}$. This shows that $\{\mathcal{F}_{\mathcal{T}}^{\on{clst}},\mathcal{F}^{\on{mnd}}_{\mathcal{T}}\}$ indeed forms a semiorthogonal decomposition of $\mathcal{F}_{\mathcal{T}}$.
\end{proof}

\subsection{The generalized cluster category of a marked surface}\label{sec4.2} 

\begin{definition}\label{def:clcat}
Let $\mathcal{D}$ be a compactly generated $k$-linear $\infty$-category. 
\begin{enumerate}[(1)]
\item We define the finite subcategory $\mathcal{D}^{\on{fin}}\subset \mathcal{D}$ as the full subcategory consisting of objects $Y\in \mathcal{D}$, satisfying that $\on{Mor}_{\mathcal{D}}(X,Y)\in \mathcal{D}(k)$ is perfect for all compact objects $X\in \mathcal{D}^{\on{c}}$. We note if $\D$ admits a compact generator $Z$, then $Y$ is finite if and only if $\on{Mor}_{\mathcal{D}}(Z,Y)$ is perfect.
\item Suppose that $\mathcal{D}$ is smooth. Then the evaluation functor of $\D$ admits a left adjoint, which we identify with an endofunctor $\on{id}_\D^!$ of $\D$ (also known as the inverse dualizing bimodule). Suppose further that $\on{id}_\D^!$ is an equivalence. Then the finite objects in $\D$ are compact, i.e.~$\mathcal{D}^{\on{fin}}\subset \D^{\on{c}}$, see \cite[Lem.~2.23]{Chr23}. Passing to $\on{Ind}$-completions yields the fully faithful inclusion $\on{Ind}\mathcal{D}^{\on{fin}}\subset \mathcal{D}$. We define the generalized cluster category of $\mathcal{D}$ as the $\on{Ind}$-complete Verdier quotient $\mathcal{D}/\on{Ind}\mathcal{D}^{\on{fin}}$ (i.e.~cofiber in $\mathcal{P}r^L_{\on{St}}$). 
\end{enumerate}
\end{definition}

\begin{remark}\label{rem:clsod}
Let $\mathcal{D}$ be as in part (2) of \Cref{def:clcat}. Then $\mathcal{D}$ admits a semiorthogonal decomposition $(\mathcal{D}/\on{Ind}\mathcal{D}^{\on{fin}},\on{Ind}\mathcal{D}^{\on{fin}})$ into its generalized cluster category and its $\on{Ind}$-finite part, see \Cref{sodquotlem}.
\end{remark}

The goal of this section is to prove the following theorem.

\begin{theorem}\label{thm:clstcat}
The generalized cluster category of $\mathcal{D}(\mathscr{G}_\mathcal{T})$ is equivalent to the $\infty$-category $\mathcal{H}(\mathcal{T},\mathcal{F}_{\mathcal{T}}^{\on{clst}})$ of global sections of the perverse schober $\mathcal{F}_{\mathcal{T}}^{\on{clst}}$. 
\end{theorem}

\begin{remark}
A shown in \cite[Section 6.4]{Chr22}, the $\infty$-category $\mathcal{D}(\mathscr{G}_\mathcal{T})\simeq \mathcal{H}(\mathcal{T},\mathcal{F}_{\mathcal{T}})$ does not depend on the choice of \idtr{} $\mathcal{T}$ of the marked surface ${\bf S}$, up to equivalence of $\infty$-categories. The same thus also holds for its associated generalized cluster category. We therefore write $\mathcal{C}_{\bf S}\coloneqq \mathcal{H}(\mathcal{T},\mathcal{F}_{\mathcal{T}}^{\on{clst}})$ and call $\mathcal{C}_{\bf S}$ the generalized cluster category of the marked surface ${\bf S}$.
\end{remark}

To prove \Cref{thm:clstcat}, we show that the $\infty$-category of global sections of $\mathcal{F}_{\mathcal{T}}^{\on{mnd}}$ is equivalent to $\on{Ind}\mathcal{H}(\mathcal{T},\mathcal{F}_{\mathcal{T}})^{\on{fin}}$ and then make use of the semiorthogonal decomposition $\{\mathcal{F}_{\mathcal{T}}^{\on{clst}},\mathcal{F}_{\mathcal{T}}^{\on{mnd}}\}$ of $\mathcal{F}_{\mathcal{T}}$. The argument for the former relies on a computational argument using compact generators arising from curves described in \cite{Chr21b}.

\begin{definition}
We denote by $\mathcal{H}(\mathcal{T},\mathcal{F}_{\mathcal{T}})^{\on{Ind-fin}}\subset \mathcal{H}(\mathcal{T},\mathcal{F}_{\mathcal{T}})$ the full subcategory of global sections $X\in \mathcal{H}(\mathcal{T},\mathcal{F}_{\mathcal{T}})$ satisfying that $\on{ev}_e(X)\in \on{Ind}\on{Fun}(S^2,\mathcal{D}(k)^{\on{perf}})$ for all edges $e$ of $\mathcal{T}$.
\end{definition}

\begin{lemma}\label{lem:curvgen}
The $\infty$-category $\mathcal{H}(\mathcal{T},\mathcal{F}_{\mathcal{T}})^{\on{Ind-fin}}$ is compactly generated by the objects $M_{\gamma}^L$ associated to finite matching data $(\gamma,f^*(k))$ with $\gamma$ pure and $f^*(k)\in \on{Fun}(S^2,\mathcal{D}(k))^{\on{fin}}\subset \on{Fun}(S^2,\mathcal{D}(k))$, as defined in \cite{Chr21b}. 
\end{lemma}

\begin{proof}
Using that $f^*(k)$ is a compact generator of $\on{Ind}\on{Fun}(S^2,\mathcal{D}(k))^{\on{fin}}$, \cite[Prop.~5.19]{Chr21b} shows that a certain collection of maching data with local value $f^*(k)$ and pure, finite underlying matching curves compactly generate $\mathcal{H}(\mathcal{T},\mathcal{F}_{\mathcal{T}})^{\on{Ind-fin}}$. We remark that for the fact that the matching curves giving rise to this compact generator are finite, it is crucial that the marked surface has a marked point on each boundary component. 
\end{proof}

\begin{lemma}\label{lem:cptgen}
For $e$ an edge of $\mathcal{T}$, consider the left adjoint 
\[ \on{ev}^L_e\colon \mathcal{D}(k[t_1])\simeq \mathcal{F}_{\mathcal{T}}(e)\rightarrow \mathcal{H}(\mathcal{T},\mathcal{F}_{\mathcal{T}})\] 
of the evaluation functor $\on{ev}_e$ at $e$. The object $\bigoplus_{e\in \mathcal{T}_1}\on{ev}^L_e(k[t_1])\in \mathcal{H}(\mathcal{T},\mathcal{F}_{\mathcal{T}})$ is identified under the equivalence $\mathcal{H}(\mathcal{T},\mathcal{F}_{\mathcal{T}})\simeq \mathcal{D}(\mathscr{G}_\mathcal{T})$ with the relative Ginzburg algebra $\mathscr{G}_\mathcal{T}\in \mathcal{D}(\mathscr{G}_\mathcal{T})$. In particular, $\bigoplus_{e\in \mathcal{T}_1}\on{ev}^L_e(k[t_1])$ is a compact generator.
\end{lemma}

\begin{proof}
This is shown in \cite[Prop.~6.7]{Chr22}.
\end{proof}

\begin{proposition}\label{prop:indfin}
There exist equivalences of $\infty$-categories 
\[ \mathcal{H}(\mathcal{T},\mathcal{F}_{\mathcal{T}}^{\on{mnd}})\simeq \mathcal{H}(\mathcal{T},\mathcal{F}_{\mathcal{T}})^{\on{Ind-fin}}\simeq \on{Ind}\mathcal{D}(\mathscr{G}_\mathcal{T})^{\on{fin}}\,.\]
\end{proposition}

\begin{proof}
It follows from the definition of $\mathcal{F}_{\mathcal{T}}^{\on{mnd}}$, that the image on global sections of the inclusion $\mathcal{F}_{\mathcal{T}}^{\on{mnd}}\rightarrow \mathcal{F}_{\mathcal{T}}$ consists of global sections which evaluate on each edge of $\mathcal{T}$ to an object in $\on{Ind}\on{Fun}(S^2,\mathcal{D}(k)^{\on{perf}})\subset \on{Fun}(S^2,\mathcal{D}(k))$. We thus see that $\mathcal{H}(\mathcal{T},\mathcal{F}_{\mathcal{T}}^{\on{mnd}})\simeq \mathcal{H}(\mathcal{T},\mathcal{F}_{\mathcal{T}})^{\on{Ind-fin}}$.

We proceed with showing the equivalence $\mathcal{H}(\mathcal{T},\mathcal{F}_{\mathcal{T}})^{\on{Ind-fin}}\simeq \on{Ind}\mathcal{D}(\mathscr{G}_\mathcal{T})^{\on{fin}}$. We can realize both of these $\infty$-categories as presentable, stable subcategories of $\mathcal{H}(\mathcal{T},\mathcal{F}_{\mathcal{T}})$. An object of $\mathcal{H}(\mathcal{T},\mathcal{F}_{\mathcal{T}})$ is by \Cref{lem:cptgen} finite if and only if its evaluations to the edges of $\mathcal{T}$ are finite in $\on{Fun}(S^2,\mathcal{D}(k))$, which is equivalent to lying in $\on{Fun}(S^2,\mathcal{D}(k)^{\on{perf}})\subset \on{Fun}(S^2,\mathcal{D}(k))$. It follows that $\mathcal{D}(\mathscr{G}_\mathcal{T})^{\on{fin}}\subset \mathcal{H}(\mathcal{T},\mathcal{F}_{\mathcal{T}})^{\on{Ind-fin}}$. Using that the objects in $\mathcal{D}(\mathscr{G}_\mathcal{T})^{\on{fin}}$ are compact in $\mathcal{H}(\mathcal{T},\mathcal{F}_{\mathcal{T}})$, we further get $\on{Ind}\mathcal{D}(\mathscr{G}_\mathcal{T})^{\on{fin}}\subset \mathcal{H}(\mathcal{T},\mathcal{F}_{\mathcal{T}})^{\on{Ind-fin}}$. \Cref{lem:curvgen}, in combination with the fact that the global sections associated to finite, pure matching data with finite local value, in the sense of \cite{Chr21b}, lie by construction in $\mathcal{D}(\mathscr{G}_\mathcal{T})^{\on{fin}}\subset \mathcal{H}(\mathcal{T},\mathcal{F}_{\mathcal{T}})^{\on{Ind-fin}}$, now implies that both $\mathcal{H}(\mathcal{T},\mathcal{F}_{\mathcal{T}})^{\on{Ind-fin}}$ and $\on{Ind}\mathcal{D}(\mathscr{G}_\mathcal{T})^{\on{fin}}$ are compactly generated by the same set of objects and thus equivalent.
\end{proof}

\begin{proof}[Proof of \Cref{thm:clstcat}.]
Combine \Cref{prop:sodG}, \Cref{rem:sod1}.2 and \Cref{prop:indfin}.
\end{proof}

The dg-algebra $k[t_2^\pm]$ is commutative and the $\infty$-category $\mathcal{D}(k[t_1^\pm])$ comes with an apparent $k[t_2^\pm]$-linear structure. This can be used to construct a factorization of the perverse schober $\mathcal{F}_{\mathcal{T}}^{\on{clst}}$ through $\on{LinCat}_{k[t_2^\pm]}\rightarrow \on{St}$.  In particular, its limit $\mathcal{C}_{\bf S}=\mathcal{H}(\mathcal{T},\mathcal{F}_{\mathcal{T}}^{\on{clst}})$ inherits a $k[t_2^\pm]$-linear structure.

\begin{proposition}\label{prop:smprp}
The $k[t_2^\pm]$-linear $\infty$-category $\mathcal{C}_{\bf S}$ is smooth and proper.
\end{proposition}

\begin{proof}
Note that $\mathcal{C}_{\bf S}$ is equivalent to the colimit of the right adjoint diagram of $\mathcal{F}_{\mathcal{T}}^{\on{clst}}$ in $\on{LinCat}_{k[t_2^\pm]}$. The smoothness thus follows from the fact that finite colimits of smooth $k[t_2^\pm]$-linear $\infty$-categories along compact objects preserving functors are again smooth, see \cite[Cor.~3.11]{Chr23}. 

To see that $\mathcal{C}_{\bf S}$ is proper, consider the compact generator $\bigoplus_{e\in \mathcal{T}_1}\on{ev}^L_e(k[t_1^\pm])$ of $\C_{\bf S}$, given by the sum of the images of the left adjoints of the evaluation functors $\on{ev}_e\colon \mathcal{H}(\mathcal{T},\mathcal{F}_{\mathcal{T}}^{\on{clst}})\rightarrow \mathcal{F}_{\mathcal{T}}^{\on{clst}}(e)\simeq \mathcal{D}(k[t_1^\pm])$. Proposition 5.19 in \cite{Chr21b} shows that $\on{ev}^L_e(k[t_1^\pm])\simeq M_{c_e}^{k[t_1^\pm]}$ for some matching datum $(c_e,k[t_1^\pm])$, where $c_e$ is a finite, pure matching curve (or equivalently $\on{ev}^L_e(k[t_1^\pm])\simeq M_{c_e}$ in the notation of \Cref{sec5.2} below for some matching curve $c_e$). Using \Cref{thm:hom1,thm:hom2}, we see that $\on{End}(\bigoplus_{e\in \mathcal{T}_1}\on{ev}^L_e(k[t_1^\pm]))$ is perfect in $\mathcal{D}(k[t_2^\pm])$, showing that $\mathcal{C}_{\bf S}$ is proper.
\end{proof}

\subsection{Relation to the 1-periodic topological Fukaya category}\label{sec4.3}

The perverse schober $\mathcal{F}_{\mathcal{T}}^{\on{clst}}$ is locally constant, i.e.~has no singularities, and its generic stalk is the derived $\infty$-category $\mathcal{D}(k[t_1^\pm])$ of $1$-periodic chain complexes. It is thus locally described by the $1$-periodic derived category of the $A_2$-quiver. Global sections of locally constant perverse schobers may be referred to as topological Fukaya categories, as these recover in special cases the constructions given in \cite{DK18,DK15,HKK17}. We refer to the $\infty$-category of global sections $\mathcal{C}_{\bf S}=\mathcal{H}({\mathcal{T}},\mathcal{F}_{\mathcal{T}}^{\on{clst}})$ as the topological Fukaya category of ${\bf S}$ with values in the derived $\infty$-category of $1$-periodic chain complexes, or the $1$-periodic topological Fukaya category for short. 

There is a notion of monodromy for parametrized perverse schobers \cite{Chr23} and this monodromy of $\mathcal{F}_{\mathcal{T}}^{\on{clst}}$ along any closed curve in ${\bf S}^\circ$ is trivial, this is shown in the proof of Theorem 6.5 in \cite{Chr21b}. Furthermore, any locally constant parametrized perverse schober is fully determined up to equivalence by its generic stalk and monodromy, see \cite[Prop.~4.26]{Chr23}. We are hence justified in calling the generalized cluster category $\mathcal{C}_{\bf S}$ {\textit{the}} $1$-periodic topological Fukaya category of ${\bf S}$.

The $\infty$-category $\mathcal{C}_{\bf S}^{\on{c}}$ can further be considered as an $\infty$-categorical avatar of the topological Fukaya (dg-)category in the sense of \cite{DK18} with coefficients in the cyclic $2$-Segal object arising from the Waldhausen $S_\bullet$-construction of the $2$-periodic dg-category $\on{dgMod}_{k[t_1^\pm]}$ of dg $k[t_1^\pm]$-modules. In particular, Corollary 3.4.7 of \cite{DK18} shows the following.

\begin{theorem}\label{thm:mcgact}
The stable $\infty$-category $\mathcal{C}_{\bf S}$ is acted upon by automorphisms in $\on{ho}\on{LinCat}_{k}$ by the mapping class group of the surface ${\bf S}$ of isotopy classes of orientation preserving diffeomorphisms ${\bf S}\rightarrow {\bf S}$ restricting to the identity on $\partial {\bf S}$.
\end{theorem}

Finally, we also record a description of the generalized cluster category $\mathcal{C}_{\bf S}$ in terms of a gentle algebra. 

\begin{proposition}\label{prop:derivedgtl}
There exists an ungraded gentle algebra $\on{gtl}$ and an equivalence of $k$-linear $\infty$-categories $\mathcal{C}_{\bf S} \simeq \mathcal{D}(\on{gtl})\otimes_k \mathcal{D}(k[t_1^\pm])$.
\end{proposition}

\begin{proof}
Choose a \idtr{} $\mathcal{T}$ of ${\bf S}$. As observed in \cite{HKK17}, the $\infty$-category of global sections of a locally constant perverse schober with generic stalk $\mathcal{D}(k)$ on ${\bf S}$ parametrized by $\mathcal{T}$ admits a formal generator whose endomorphism algebra is a finite dimensional, and in general graded, gentle algebra, denoted $\on{gtl}'$. Using that tensoring with $\mathcal{D}(k[t_1^\pm])$ (with respect to the symmetric monoidal structure of $\on{LinCat}_k$) preserves colimits in $\on{LinCat}_k$, we find an equivalence $\mathcal{C}_{\bf S} \simeq \mathcal{D}(\on{gtl}')\otimes_k \mathcal{D}(k[t_1^\pm])$ in $\on{LinCat}_k$. Let $\on{gtl}$ be the ungraded gentle algebra obtained from $\on{gtl}'$ by discarding the grading. It is easy to see that $\mathcal{D}(\on{gtl}')\otimes_k \mathcal{D}(k[t_1^\pm])\simeq \mathcal{D}(\on{gtl})\otimes_k \mathcal{D}(k[t_1^\pm])$ in $\on{LinCat}_k$, showing the claim.
\end{proof}

Topological Fukaya categories admit relative Calabi--Yau structures if the monodromy of the perverse schober acts trivially on a relative Calabi--Yau structure of the generic stalk, see \cite[Thm.~5.11]{Chr23}. In our setting, this gives the following.

\begin{theorem}[$\!\!${\cite[Thm.~6.1]{Chr23}}]\label{clcatcythm}\label{thm:cyclcat}
Assume that $\on{char}(k)\neq 2$. The $k[t_2^\pm]$-linear functor
\[ \prod_{e\in \mathcal{T}_1^\partial}\on{ev}_e\colon \mathcal{C}_{\bf S}\longrightarrow  \prod_{e\in \mathcal{T}_1^\partial}\mathcal{F}_{\mathcal{T}}^{\on{clst}}(e) \]
admits a $2$-Calabi--Yau structure in the sense of \Cref{def:cystr}.
\end{theorem}

\begin{corollary}\label{cor:sphbdry}
Assume that $\on{char}(k)\neq 2$. The adjunction 
\[\partial \mathcal{F}_{\mathcal{T}}^{\on{clst}}\colon \prod_{e\in \mathcal{T}_1^\partial}\mathcal{F}_{\mathcal{T}}^{\on{clst}}(e)\longleftrightarrow  \mathcal{C}_{\bf S}\noloc \prod_{e\in \mathcal{T}_1^\partial}\on{ev}_e\]
is spherical.
\end{corollary}

\begin{proof}
The twist functor of the adjunction $\partial \mathcal{F}^{\on{clst}}_{\trivalent}\dashv \prod_{e\in \trivalent_1^\partial}\on{ev}_e$ is by \Cref{thm:cyclcat} a suspension of the inverse Serre functor $\on{id}_{\mathcal{C}_{\bf S}}^!$ and thus invertible. The cotwist functor of the adjunction $\partial \mathcal{F}^{\on{clst}}_{\trivalent}\dashv \prod_{e\in \trivalent_1^\partial}\on{ev}_e$ can be readily determined using the equivalence $\on{ev}^L_e(k[t_1^\pm])\simeq M_{c_e}^{k[t_1^\pm]}$ from the proof of \Cref{prop:smprp}, and shown to be an equivalence. It acts by permuting cyclically the copies of $\mathcal{F}_\mathcal{T}^{\on{clst}}(e)\simeq \mathcal{D}(k[t_1^\pm])$ corresponding to the external edges of each boundary circle.
\end{proof}

\section{Exact structures from relative Calabi--Yau structures}\label{sec:exactstr}

Consider an exact functor $G\colon \mathcal{B}\rightarrow \mathcal{A}$ between stable $\infty$-categories, meaning a functor which preserves finite limits and colimits. We can pull back the split-exact structure on $\mathcal{A}$ to an exact structure on the $\infty$-category $\mathcal{B}$, such that a sequence in $\mathcal{B}$ is exact if and only if its image under $G$ is a split-exact sequence in $\mathcal{A}$, see \Cref{ex:splitpb}. The goal of this section is to show that if $G$ carries a right $2$-Calabi--Yau structure and is spherical, then the exact structure on $\mathcal{B}$ is Frobenius and its extriangulated homotopy category is $2$-Calabi--Yau. 

In \Cref{sec6.1}, we recall the notion of an exact $\infty$-category and describe pullbacks of exact structures. In \Cref{sec6.2}, we recall what is an extriangulated category and why it is natural to consider the extriangulated homotopy categories of exact $\infty$-categories. In \Cref{subsec:exactfromCY}, we prove that the Frobenius extriangulated structure arising from a functor carrying a right $2$-Calabi--Yau structure is extriangulated $2$-Calabi--Yau.

\subsection{Exact \texorpdfstring{$\infty$}{infinity}-categories}\label{sec6.1}

\begin{definition}[$\!\!$\cite{Bar15}]
An exact $\infty$-category is a triple $(\mathcal{C},\mathcal{C}_{\dagger},\mathcal{C}^{\dagger})$, where $\mathcal{C}$ is an additive $\infty$-category and $\mathcal{C}_{\dagger},\mathcal{C}^{\dagger}\subset \mathcal{C}$ are subcategories (called subcategories of inflations and deflations), satisfying that 
\begin{enumerate}[(1)]
\item every morphism $0\rightarrow X$ in $\mathcal{C}$ lies in $\mathcal{C}_{\dagger}$ and every morphism $X\rightarrow 0$ in $\mathcal{C}$ lies in $\mathcal{C}^\dagger$.
\item pushouts in $\mathcal{C}$ along morphisms in $\mathcal{C}_{\dagger}$ exist and lie in $\mathcal{C}_{\dagger}$. Dually, pullbacks in $\mathcal{C}$ along morphisms in $\mathcal{C}^{\dagger}$ exist and lie in $\mathcal{C}^{\dagger}$. 
\item  Given a commutative square in $\mathcal{C}$ of the form
\[
\begin{tikzcd}
X \arrow[r, "a"] \arrow[d, "b"] & Y \arrow[d, "c"] \\
X' \arrow[r, "d"]               & Y'              
\end{tikzcd}
\]
the following are equivalent.
\begin{itemize}
\item The square is pullback, $c\in \mathcal{C}_{\dagger}$ and $d\in \mathcal{C}^{\dagger}$.
\item The square is pushout, $b\in \mathcal{C}_{\dagger}$ and $a\in \mathcal{C}^{\dagger}$.
\end{itemize}
\end{enumerate}
When the exact structure is clear from the context, we also simply refer to $\mathcal{C}$ as an exact $\infty$-category.
\end{definition}

\begin{definition}
An exact sequence $X\rightarrow Y\rightarrow Z$ in an exact $\infty$-category $(\mathcal{C},\mathcal{C}_{\dagger},\mathcal{C}^{\dagger})$ consists of a fiber and cofiber sequence in $\mathcal{C}$
\[
\begin{tikzcd}
X \arrow[r, "a"] \arrow[dr, "\square", phantom] \arrow[d] & Y \arrow[d, "b"] \\
0 \arrow[r]                & Z               
\end{tikzcd}
\]
with $a\in \mathcal{C}_{\dagger}$ and $b\in \mathcal{C}^{\dagger}$. 
\end{definition}

\begin{example}\label{ex:exstr}
\begin{enumerate}
\item Let $\mathcal{C}$ be an additive $\infty$-category. Then there is an exact $\infty$-category $(\mathcal{C},\mathcal{C}_{\dagger},\mathcal{C}^{\dagger})$ with inflations consisting of inclusions of direct summands and deflations consisting of projections onto direct summands, called the split-exact structure. The exact sequences are the split-exact sequences $X\hookrightarrow X\oplus Y \twoheadrightarrow Y$. 
\item Let $\mathcal{C}$ be a stable $\infty$-category. Then $(\mathcal{C},\mathcal{C},\mathcal{C})$ is an exact $\infty$-category.
\end{enumerate}
\end{example}

\begin{definition}\label{def:frob}
Let $(\mathcal{C},\mathcal{C}_{\dagger},\mathcal{C}^{\dagger})$ be an exact $\infty$-category.
\begin{enumerate}[1)]
\item An object $P\in \mathcal{C}$ is called projective if every exact sequence $X\rightarrow Y \rightarrow P$ is split-exact. An object $I\in \mathcal{C}$ is called injective if every exact sequence $I\rightarrow Y \rightarrow Z$ is split-exact.
\item We say that $\mathcal{C}$ has enough projectives if for each object $X\in \mathcal{C}$ there exists an exact sequence $X\rightarrow P\rightarrow Y$ with $P$ projective. Similarly, we say that $\mathcal{C}$ has enough injectives if for each object $Y\in \mathcal{C}$ there exists an exact sequence $Y\rightarrow I \rightarrow X$ with $I$ injective.
\item We say that $\mathcal{C}$ is Frobenius if $\mathcal{C}$ has enough projectives and injectives and the classes of projective and injective objects coincide. 
\end{enumerate}
\end{definition}

The $\infty$-categorical version of the stable $1$-category of a Frobenius exact $1$-category is the following.

\begin{proposition}[$\!\!$\cite{JKPW22}]\label{prop:jkpw}
Let $(\mathcal{C},\mathcal{C}_{\dagger},\mathcal{C}^{\dagger})$ be a Frobenius exact $\infty$-category and $W$ the class of morphisms $f\colon X\rightarrow Y$ which fit into an exact sequence 
\[ X\xlongrightarrow{(f,g)} Y\oplus I\longrightarrow J\] 
with $I$ and $J$ injective and $g$ arbitrary. Then, the $\infty$-categorical localisation $\bar{\mathcal{C}}\coloneqq \mathcal{C}[W^{-1}]$ of $\mathcal{C}$ at $W$ is a stable $\infty$-category.
\end{proposition}

\begin{proof}[Proof idea]
Applying \cite[Prop.~7.5.6]{Cis} and its dual version, one finds that $\bar{\mathcal{C}}$ admits finite limits and colimits and pushouts and pullbacks coincide.
\end{proof}

\begin{definition}
\begin{enumerate}
\item An exact functor $G\colon (\mathcal{C},\mathcal{C}_{\dagger},\mathcal{C}^{\dagger})\rightarrow (\mathcal{D},\mathcal{D}_\dagger,\mathcal{D}^\dagger)$ between exact $\infty$-categories consists of a functor $G\colon \mathcal{C}\rightarrow \mathcal{D}$ which preserves zero objects, inflations, deflations, as well as exact sequences\footnote{This is equivalent to the a priori slightly stronger assumption of preserving pushouts along inflations and pullbacks along deflations.}. 
\item A sub-exact structure of an exact $\infty$-category $(\mathcal{C},\mathcal{C}_{\dagger},\mathcal{C}^{\dagger})$ consists of an exact $\infty$-category $(\mathcal{C},\mathcal{C}_{\ddagger},\mathcal{C}^{\ddagger})$, satisfying that $\mathcal{C}_{\ddagger}\subset \mathcal{C}_{\dagger}$ and $\mathcal{C}^{\ddagger}\subset \mathcal{C}^{\dagger}$. 
\end{enumerate}
\end{definition}

Sub-exact structures can be pulled back along exact functors as follows. 

\begin{lemma}\label{lem:expb}
Let $G\colon (\mathcal{C},\mathcal{C}_{\dagger},\mathcal{C}^{\dagger})\rightarrow (\mathcal{D},\mathcal{D}_\dagger,\mathcal{D}^\dagger)$ be an exact functor between exact $\infty$-categories and $(\mathcal{D},\mathcal{D}_{\ddagger},\mathcal{D}^{\ddagger})$ a sub-exact structure of $(\mathcal{D},\mathcal{D}_\dagger,\mathcal{D}^\dagger)$. Then there exists a sub-exact structure $(\mathcal{C},\mathcal{C}_\ddagger,\mathcal{C}^{\ddagger})$ of $(\mathcal{C},\mathcal{C}_{\dagger},\mathcal{C}^{\dagger})$, such that $G$ defines an exact functor $G\colon (\mathcal{C},\mathcal{C}_{\ddagger},\mathcal{C}^{\ddagger})\rightarrow (\mathcal{D},\mathcal{D}_\ddagger,\mathcal{D}^\ddagger)$.
\end{lemma}

\begin{proof}
We set $\mathcal{C}_{\ddagger}\subset \mathcal{C}_{\dagger}$ to be the subcategory of morphisms whose image under $G$ lies in $\mathcal{D}_{\ddagger}$. We define $\mathcal{C}^{\ddagger}$ similarly. It is straightforward to verify that $(\mathcal{C},\mathcal{C}_{\ddagger},\mathcal{C}^{\ddagger})$ is an exact $\infty$-category and that $G\colon (\mathcal{C},\mathcal{C}_{\ddagger},\mathcal{C}^{\ddagger})\rightarrow (\mathcal{D},\mathcal{D}_\ddagger,\mathcal{D}^\ddagger)$ is an exact functor.
\end{proof}

\begin{example}\label{ex:splitpb}
Suppose that $\mathcal{C}$ and $\mathcal{D}$ are stable $\infty$-categories and $G\colon \mathcal{C}\rightarrow \mathcal{D}$ is an exact functor in the usual sense, i.e.~a functor that preserves finite limits and colimits. Then $G$ defines an exact functor $G\colon (\mathcal{C},\mathcal{C},\mathcal{C})\rightarrow (\mathcal{D},\mathcal{D},\mathcal{D})$ between exact $\infty$-categories. Applying \Cref{lem:expb}, we obtain an additional exact structure $(\mathcal{C},\mathcal{C}_{\dagger},\mathcal{C}^{\dagger})$ on $\mathcal{C}$ by pulling back the split-exact structure $(\mathcal{D},\mathcal{D}_{\dagger},\mathcal{D}^{\dagger})$ on $\mathcal{D}$. A fiber and cofiber sequence in $\mathcal{C}$ is exact in $(\mathcal{C},\mathcal{C}_{\dagger},\mathcal{C}^{\dagger})$ if and only if its image under $G$ is split-exact.
\end{example}

\begin{proposition}\label{prop:frobex}
Let $G\colon \mathcal{C}\rightarrow \mathcal{D}$ be an exact functor between stable $\infty$-categories. If $G$ is a spherical functor, then the exact $\infty$-category $(\mathcal{C},\mathcal{C}_{\dagger},\mathcal{C}^{\dagger})$ from \Cref{ex:splitpb} is Frobenius. The subcategory of injective and projective objects of $\mathcal{C}$ consists of the additive closure of the essential image of the right adjoint $H$ of $G$.
\end{proposition}

\begin{remark}
A result similar to \Cref{prop:frobex} appears for triangulated categories in \cite{BS21}, see Theorem 4.23 and Remark 4.24 in loc.~cit.
\end{remark}

\begin{proof}[Proof of \Cref{prop:frobex}]
Let $F$ be the left adjoint of $G$ and $H$ the right adjoint of $G$. By the sphericalness of $G$, we have that $H\simeq F\circ T_{\mathcal{D}}$, with $T_\mathcal{D}$ the cotwist functor of the adjunction $G\dashv H$, see \cite[Cor.~2.5.16]{DKSS21}. In particular, it follows that the essential images of $F$ and $H$, and hence also their additive closures in $\mathcal{C}$, agree. 

We begin by showing that every object of the form $F(d)\in \mathcal{C}$ with $d\in \mathcal{D}$ is injective. A dual argument shows that every object of the form $H(d)\in \mathcal{C}$ is projective. Let $c\in \mathcal{C}$ and consider an extension $\alpha \in \on{Ext}^1_{\mathcal{C}}(F(d),c)$. The extension $\alpha$ corresponds to a fiber and cofiber sequence $c\rightarrow z \rightarrow F(d)$ in $\mathcal{C}$. The image of this fiber and cofiber sequence under $G$ splits (i.e.~defines an exact sequence) if and only if $G(\alpha)\simeq 0$. The commutative diagram,
\[
\begin{tikzcd}
                                                                           & {\on{Ext}^{1}_{\mathcal{D}}(GF(d),G(c))} \arrow[rd, "\on{u}"] &                                      \\
{\on{Ext}^{1}_{\mathcal{C}}(F(d),c)} \arrow[rr, "\simeq"'] \arrow[ru, "G"] &                                                               & {\on{Ext}^{1}_{\mathcal{D}}(d,G(c))}
\end{tikzcd}
\]
with $\on{u}$ the unit of $F\dashv G$, shows that $G(\alpha)\simeq 0$ if and only if $\alpha\simeq 0$. This is the case, if and only if the fiber and cofiber sequence $c\rightarrow z \rightarrow F(d)$ already splits. Since $c$ and $\alpha$ were chosen arbitrarily, it follows that $F(d)$ is injective.

Next, we show that $\mathcal{C}$ has enough projective objects. Let $c\in \mathcal{C}$. We have a fiber and cofiber sequence $c\xrightarrow{\on{u}_c}HG(c)\rightarrow T_{\mathcal{C}}(c)$, where $\on{u}_c$ is a unit map and $T_{\mathcal{C}}$ the twist functor of $G\dashv H$. We apply $G$ and extend to the diagram
\[
\begin{tikzcd}
                                                                      & T_\mathcal{D}G(c) \arrow[rd, "\square", phantom] \arrow[d] \arrow[r]    & 0 \arrow[d]    \\
G(c) \arrow[d] \arrow[r, "G\on{u}_c"'] \arrow[rd, "\square", phantom] & GHG(c) \arrow[r, "\on{cu}_{G(c)}"'] \arrow[d] \arrow[rd, "\square", phantom] & G(c) \arrow[d] \\
0 \arrow[r]                                                           & GT_\mathcal{C}(c) \arrow[r]                                             & 0             
\end{tikzcd}
\]
with $\on{cu}_{G(c)}$ a counit map and all squares biCartesian. This shows that the sequence $G(c)\xrightarrow{G\on{u}_c}GHG(x)\rightarrow GT_{\mathcal{C}}(c)$ in $\mathcal{D}$ splits and hence that that $c\xrightarrow{\on{u}_c}HG(x)\rightarrow T_{\mathcal{C}}(c)$ is exact. This shows that $\mathcal{C}$ has enough projective objects.

Finally, we show that any injective or projective object lies in the additive closure (closure under taking direct sums of direct summands) of the essential image of $H$. Let $I$ be injective. Then we have an exact sequence $I\rightarrow HG(I)\rightarrow T_{\mathcal{C}}(I)$ in $\mathcal{C}$. Since $I$ is injective, this sequence splits, showing that $I$ is a direct summand of $HG(I)$, as desired. If $P$ is projective, we similarly find an exact sequence $T_{\mathcal{C}}^{-1}(P)\rightarrow FG(P)\rightarrow P$, showing that $P$ is a direct summand of $FG(P)\simeq HT_\mathcal{D}^{-1}G(P)$, as desired, concluding the proof. 
\end{proof}

\subsection{Extriangulated categories}\label{sec6.2}

Extriangulated categories were introduced by Nakaoka-Palu in \cite{NP19} as a simultaneous generalization of triangulated and exact $1$-categories. An extriangulated category $(C,\mathbb{E},\mathfrak{s})$ consists of 
\begin{itemize}
\item an additive $1$-category $C$,
\item an additive bifunctor $\mathbb{E}\colon C^{\on{op}}\times C\rightarrow \on{Ab}$ to the additive $1$-category of abelian groups, of which we think as describing an interesting class of extensions in $C$ and 
\item a sequence $\left(Y\rightarrow Z\rightarrow X\right)=\mathfrak{s}(\alpha)$, called realization, associated to each $\alpha\in \mathbb{E}(X,Y)$,
\end{itemize}
subject to number of conditions, see \cite{NP19}. We will refer to the realization sequences as exact sequences in $C$, though Nakaoka-Palu call them conflation sequences. 

Examples of extriangulated categories are, besides triangulated and exact categories, extension closed subcategories of triangulated categories. A further natural source of extriangulated categories are the homotopy $1$-categories of exact $\infty$-categories. To see this, note that by \cite{Kle20} any exact $\infty$-category admits a stable hull, meaning that it can be embedded as an extension closed subcategory in a stable $\infty$-category. Passing to homotopy categories, one obtains an extriangulated structure on the homotopy category, as it is an extension closed subcategory of the triangulated homotopy category of the stable hull. An independent and more direct proof that the homotopy category is extriangulated also appears in \cite{NP20}.  

In the remainder of this article, we will assume all extriangulated categories to be $k$-linear, with $k$ a field, meaning that the $1$-category $C$ is $k$-linear and $\mathbb{E}$ factors through $\on{Vect}_k\to \on{Ab}$.

Given a $k$-linear exact $\infty$-category $\mathcal{C}$, the additive functor of extensions $\mathbb{E}\colon \on{ho}\mathcal{C}^{\on{op}}\times \on{ho}\mathcal{C}\rightarrow \on{Vect}_k$ in the extriangulated structure of $\on{ho}\mathcal{C}$ describes isomorphism classes of exact sequences in $\mathcal{C}$. The functor $\mathbb{E}$ can be useful for studying the exact $\infty$-category $\mathcal{C}$. For instance, one can use it to formulate a notion of cluster tilting object in the exact $\infty$-category $\mathcal{C}$, and this is simply an object which becomes cluster tilting in the extriangulated homotopy category $\on{ho}\mathcal{C}$ in the sense of \Cref{def:ctobj} below.

The definition of a Frobenius extriangulated category is analogous to the definition of a Frobenius exact $\infty$-category.

\begin{definition}\label{def:frobex}
Let $(C,\mathbb{E},\mathfrak{s})$ be an extriangualated category.
\begin{enumerate}[1)]
\item An object $P\in C$ is called projective if $\mathbb{E}(P,X)\simeq 0$ for all $X\in C$ and injective if $\mathbb{E}(X,P)\simeq 0$ for all $X\in C$.
\item We say that $C$ has enough projectives if for each object $X\in C$ there exists an exact sequence $X\rightarrow P\rightarrow Y$ with $P$ projective. Similarly, we say that $C$ has enough injectives if for each object $Y\in C$ there exists an exact sequence $Y\rightarrow I \rightarrow X$ with $I$ injective.
\item We say that $C$ is Frobenius if $C$ has enough projectives and injectives and the classes of projective and injective objects coincide. 
\end{enumerate}
\end{definition}

The inflations, deflations and exact sequences in an exact $\infty$-category can be read off from the extriangulated structure of its homotopy category. Similarly, the condition of being a projective or injective object in an exact $\infty$-category can be tested in the extriangulated homotopy category. We also get the following.

\begin{lemma}
A ($k$-linear) exact $\infty$-category is Frobenius if and only if its extriangulated homotopy category is Frobenius.
\end{lemma}

\begin{remark}
Given a Frobenius extriangulated category $C$, its stable category $\bar{C}$ is defined as the quotient category by the ideal of morphisms which factor through an injective object. The stable category $\bar{C}$ inherits the structure of a triangulated category, see Corollary 7.4 and Remark 7.5 in \cite{NP19}.

If $(\mathcal{C},\mathcal{C}_{\dagger},\mathcal{C}^{\dagger})$ is a ($k$-linear) Frobenius exact $\infty$-category, the triangulated homotopy category of its associated stable $\infty$-category $\bar{\mathcal{C}}$ can be identified with the triangulated stable category $\overline{\on{ho}\mathcal{C}}$ of the Frobenius extriangulated homotopy category $\on{ho}\mathcal{C}$.
\end{remark} 

\begin{definition}\label{def:extcy}
An extriangulated category $(C,\mathbb{E},\mathfrak{s})$ is called extrianguled $2$-Calabi--Yau if there exists an isomorphism of vector spaces $\mathbb{E}(X,Y)\simeq \mathbb{E}(Y,X)^\ast$, bifunctorial in $X$ and $Y$.
\end{definition}

We conclude this section with the definition of a cluster tilting object in an extriangulated category. 

\begin{definition}\label{def:ctobj}
Let $T$ be an object in an extriangulated category $(C,\mathbb{E},\mathfrak{s})$.
\begin{itemize}
\item We denote by $\on{Add}(T)\subset C$ the additive closure of $T$, given by the subcategory given by objects equivalent to finite direct sums of direct summands of $T$. 
\item $T$ is called basic if it can be decomposed into finitely many indecomposable direct summands which are pairwise non-isomorphic.
\item $T$ is called rigid if 
\[\mathbb{E}(T,T)\simeq 0\,.\]
Further, a basic rigid object $T$ is called maximal rigid, if there exist no object $X\in C$, satisfying that $T\oplus X$ is rigid and $X\not \in \on{Add}(T)$.
\item $T$ is called cluster tilting if it is basic rigid object and for every $X\in C$ there exists an exact sequence $T_1\to T_0\to X$ with $T_0,T_1\in \on{Add}(T)$. 
\end{itemize}
\end{definition}

\begin{lemma}\label{lem:2termresolution}
Let $(C,\mathbb{E},\mathfrak{s})$ be a Frobenius extriangulated category with finite-dimensional Homs and $T\in C$ a basic rigid object. 
\begin{enumerate}[(1)]
\item Then $T$ is a cluster tilting object if and only if for every object $X\in C$ the condition $\mathbb{E}(X,T)\simeq 0$ implies that $X\in \on{Add}(T)$. 
\item Dually, the following two statements are equivalent:
\begin{itemize}
\item For every $X\in C$ there exists an exact sequence $X\to T_0\to T_1$ with $T_0,T_1\in \on{Add}(T)$. 
\item For every $X\in C$, the condition $\mathbb{E}(T,X)\simeq 0$ implies that $X\in \on{Add}(T)$. 
\end{itemize}
\end{enumerate}
Note that if $C$ is extriangulated $2$-Calabi--Yau, this shows that $T$ is a cluster tilting object if and only if the equivalent conditions of (2) hold.
\end{lemma}

\begin{proof}
The proof of part (2) is analogous to the proof of part (1), we thus only prove part (1). 

Assume that $T$ is a cluster tilting object. Choose an exact sequence $T_1\to T_0\to X$ in $C$ with $T_0,T_1\in \on{Add}(T)$. The exact sequence of vector spaces, see \cite[Thm.~3.5]{GNP21}, 
\[ \on{Hom}(T_0,T) \to \on{Hom}(T_1,T)\to \mathbb{E}(X,T)\to \underbrace{\mathbb{E}(T_0,T)}_{\simeq 0}\]
shows that if $\mathbb{E}(X,T)\simeq 0$, then the sequence $T_1\to T_0 \to X$ splits (since the connecting homomorphism vanishes), so that $X$ is a direct summand of $T_0$.

For the converse direction, assume that for all $X\in C$, the condition $\mathbb{E}(X,T)\simeq 0$ implies that $X\in \on{Add}(T)$. Since $C$ is Frobenius, we can find a deflation $P\to X$ with $P$ projective. We define $T_0=P\oplus (T\otimes \on{Hom}_C(T,X))\in \on{Add}(T)$. The apparent morphisms $T_0\to X$ is again a deflation by \cite[Cor.~3.16]{NP19}. It remains to show that the term $T_0$ appearing in the exact sequence $T_1\to T_0\to X$ lies in $\on{Add}(T)$. We have an exact sequence of vector spaces
\[ \on{Hom}(T,T_0)\to \on{Hom}(T,X) \to \mathbb{E}(T,T_1)\to \underbrace{\mathbb{E}(T,T_0)}_{\simeq 0} \]
see again \cite[Thm.~3.5]{GNP21}. The first morphisms is surjective by the definition of $T_0$, which shows that $\mathbb{E}(T,T_1)\simeq 0$. It follows that $T_1\in \on{Add}(T)$, concluding the proof.
\end{proof}

\subsection{Exact structures via relative Calabi--Yau structures}\label{subsec:exactfromCY}

Let $\kappa=k$ be a field or $\kappa=k[t_{2n}^\pm]$ the commutative dg-algebra of graded Laurent polynomials over a field. We fix a functor $G\colon \mathcal{B}\rightarrow \mathcal{A}$ between $\kappa$-linear smooth and proper $\infty$-categories, which is spherical and carries a right $n$-Calabi--Yau structure. 

We begin this section by exhibiting a relative version of the triangulated $n$-Calabi--Yau property $\on{Ext}^i(X,Y)\simeq \on{Ext}^{n-i}(Y,X)^\ast$ for $\mathcal{B}$, using the relative right Calabi--Yau structure of $\mathcal{B}$. This relative Calabi--Yau property then yields a $2$-Calabi--Yau property on the Frobenius extriangulated structure on $\on{ho}\mathcal{B}$.

\begin{definition}\label{def:cyfun}
Let $X,Y\in \mathcal{B}^{\on{c}}$ be compact objects.
\begin{enumerate}
\item We denote by $\on{Mor}_{\mathcal{B}}^{\on{CY}}(X,Y)\subset \on{Mor}_{\mathcal{B}}(X,Y)\in \mathcal{D}(\kappa)$ the maximal direct summand\footnote{Equivalently this is the kernel of \eqref{eq:GMor} in the abelian homotopy $1$-category $\on{ho}\D(\kappa)$.} satisfying that the composite with 
\begin{equation}\label{eq:GMor} 
G\colon \on{Mor}_{\mathcal{B}}(X,Y)\longrightarrow \on{Mor}_{\mathcal{A}}(G(X),G(Y))
\end{equation} 
yields the zero morphism in $\mathcal{D}(\kappa)$. We call $\on{Mor}_{\mathcal{B}}^{\on{CY}}(X,Y)$ the \textit{Calabi--Yau morphism object}. 
\item We denote by $\on{Ext}_{\mathcal{B}}^{i,\on{CY}}(X,Y)\coloneqq \on{H}_0\on{Mor}_{\mathcal{B}}^{\on{CY}}(X,Y[i])\in \on{N}(\on{Vect}_k)$ the $k$-vector space of \textit{Calabi--Yau extensions}. 
\end{enumerate}
\end{definition}

\begin{lemma}\label{lem:cyfun}
\begin{enumerate}[(1)]
\item The Calabi--Yau morphism objects assemble into a functor 
\[ \on{Mor}_{\mathcal{B}}^{\on{CY}}(\mhyphen,\mhyphen)\colon \mathcal{B}^{\on{c},\on{op}}\times \mathcal{B}^{\on{c}}\rightarrow \mathcal{D}(\kappa)\,.\]
There further exists a natural transformation $\on{Mor}_{\mathcal{B}}^{\on{CY}}(\mhyphen,\mhyphen)\rightarrow \on{Mor}_{\mathcal{B}}(\mhyphen,\mhyphen)$ which at a point $(X,Y)\in \mathcal{B}^{\on{c},\on{op}}\times \mathcal{B}^{\on{c}}$ is given by the inclusion of the direct summand $\on{Mor}_{\mathcal{B}}^{\on{CY}}(X,Y)\subset \on{Mor}_{\mathcal{B}}(X,Y)$.
\item The Calabi--Yau extensions form the maximal subfunctor 
\[ \on{Ext}^{i,\on{CY}}_{\mathcal{B}}(\mhyphen,\mhyphen)\subset \on{Ext}^i_{\mathcal{B}}(\mhyphen,\mhyphen)\colon \mathcal{B}^{\on{c},\on{op}}\times \mathcal{B}^{\on{c}}\rightarrow \on{N}(\on{Vect}_k)\]
satisfying that $\on{Ext}^{i,\on{CY}}_{\mathcal{B}}(G(\mhyphen),G(\mhyphen))\simeq 0$. 
\end{enumerate}

\end{lemma}

\begin{proof}
Consider the mapping space functor $\on{Map}_{\mathcal{B}}(\mhyphen,\mhyphen)\colon \mathcal{B}^{\on{c},\on{op}}\times \mathcal{B}^{\on{c}}\rightarrow \mathcal{S}$ and the functor $\pi_0\on{Map}_{\mathcal{B}}(\mhyphen,\mhyphen)\colon \mathcal{B}^{\on{c},\on{op}}\times \mathcal{B}^{\on{c}}\rightarrow \mathcal{S}$, obtained from composing $\on{Map}_{\mathcal{B}}(\mhyphen,\mhyphen)$ with the functor $\mathcal{S}\xrightarrow{\pi_0}N(\on{Set})\hookrightarrow \mathcal{S}$ taking connected components. Given $X,Y\in \mathcal{B}$, we denote by $(\pi_0\on{Map}_{\mathcal{B}})^{\on{CY}}(X,Y)\subset \pi_0\on{Map}_{\mathcal{B}}(X,Y)$ the subset of homotopy classes of morphisms, satisfying that their image under $G$ is a zero morphism in $\mathcal{A}$. Since zero morphisms are closed under composition, we find that $(\pi_0\on{Map}_{\mathcal{B}})^{\on{CY}}(\mhyphen,\mhyphen)$ defines a subfunctor of $\pi_0\on{Map}_{\mathcal{B}}(X,Y)$. We define the Calabi--Yau mapping space functor $\on{Map}_{\mathcal{B}}^{\on{CY}}(\mhyphen,\mhyphen)\colon \mathcal{B}^{\on{c},\on{op}}\times \mathcal{B}^{\on{c}}\rightarrow \mathcal{S}$ as the pullback of the diagram 
\[
\begin{tikzcd}
                                                                          & {\on{Map}_{\mathcal{B}}(\mhyphen,\mhyphen)} \arrow[d] \\
{(\pi_0\on{Map}_{\mathcal{B}})^{\on{CY}}(\mhyphen,\mhyphen)} \arrow[r] & {\pi_0\on{Map}_{\mathcal{B}}(\mhyphen,\mhyphen)}     
\end{tikzcd}
\]
in $\on{Fun}(\mathcal{B}^{\on{c},\on{op}}\times \mathcal{B}^{\on{c}},\mathcal{S})$. Using the fully-faithfulness of the Yoneda embedding, we can define a functor $\on{Mor}_{\mathcal{B}}^{\on{CY}}(\mhyphen,\mhyphen)\colon \mathcal{B}^{\on{c},\on{op}}\times \mathcal{B}^{\on{c}}\rightarrow \mathcal{D}(\kappa)$ via 
\[ \on{Map}_{\mathcal{D}(\kappa)}(C,\on{Mor}_{\mathcal{B}}^{\on{CY}}(\mhyphen,\mhyphen))\simeq \on{Map}_{\mathcal{B}}^{\on{CY}}(\mhyphen\otimes C,\mhyphen)\,.\]
The natural transformation $\on{Map}_{\mathcal{B}}^{\on{CY}}(\mhyphen,\mhyphen)\subset \on{Map}_{\mathcal{B}}(\mhyphen,\mhyphen)$ induces a natural transformation $\eta\colon \on{Mor}^{\on{CY}}_{\mathcal{B}}(\mhyphen,\mhyphen)\rightarrow \on{Mor}_{\mathcal{B}}(\mhyphen,\mhyphen)$. Note that for any pair $(X,Y)\in \mathcal{B}^{\on{c},\on{op}}\times \mathcal{B}^{\on{c}}$, the natural transformation $\eta$ evaluates on the $i$-th homology group as the inclusion \begin{equation}\label{eq:hieta}(\pi_0\on{Map}_{\mathcal{B}})^{\on{CY}}(X[i],Y)\subset \pi_0\on{Map}_{\mathcal{B}}(X[i],Y)\,.\end{equation}
Using that any object in $\mathcal{D}(\kappa)$ is equivalent to its homology, which is the direct sum of suspensions of copies of $\kappa$, we find that the inclusion $\on{Mor}_{\mathcal{B}}^{\on{CY}}(X,Y)\subset \on{Mor}_{\mathcal{B}}(X,Y)$ of the Calabi--Yau morphism object evaluates on the $i$-th homology group to the inclusion \eqref{eq:hieta}. This implies that the functor $\on{Mor}_{\mathcal{B}}^{\on{CY}}(\mhyphen,\mhyphen)$ indeed describes the Calabi--Yau morphism objects and $\eta$ evaluates pointwise to their inclusion. This shows part (1). 

Part (2) follows from the fact that on $i$-th homology, $\eta$ exhibits $\on{Ext}^{i,\on{CY}}_{\mathcal{B}}(\mhyphen,\mhyphen)$ as the desired maximal subfunctor of $\on{Ext}^{i}_{\mathcal{B}}(\mhyphen,\mhyphen)$, by the description in \eqref{eq:hieta}. 
\end{proof}

Recall from \Cref{def:serre}, that given $A\in \mathcal{D}(\kappa)$, we write $A^*=\on{Mor}_{\mathcal{D}(\kappa)}(A,\kappa)$. Given a $k$-vector space $B$, we write $B^*=\on{Hom}_{\on{Vect}_k}(B,k)$.

\begin{proposition}\label{prop:cydual}
Let $X,Y\in \mathcal{B}^c$ be compact objects.
\begin{enumerate}[(1)]
\item There exists an equivalence in $\mathcal{D}(\kappa)$
\[ 
\on{Mor}_{\mathcal{B}}^{\on{CY}}(X,Y) \simeq \on{Mor}_{\mathcal{B}}^{\on{CY}}(Y[-n],X)^\ast\,,
\]
bifunctorial in $X$ and $Y$.
\item There exists an equivalence in $N(\on{Vect}_k)$
\[ \on{Ext}^{i,\on{CY}}_\mathcal{B}(X,Y)\simeq \on{Ext}^{n-i,\on{CY}}_\mathcal{B}(Y,X)^{\ast}\,,\]
bifunctorial in $X,Y$. 
\end{enumerate}
\end{proposition}

Part (2) of \Cref{prop:cydual} should be seen as a relative version of the triangulated $n$-Calabi--Yau property. We postpone the proof of \Cref{prop:cydual} to the end of this subsection.

Let $(\mathcal{B},\mathcal{B}_\dagger,\mathcal{B}^\dagger)$ be the exact $\infty$-category obtained by pulling back the split-exact structure on $\mathcal{A}$ along $G$, see \Cref{ex:splitpb}. We denote by $(\on{ho}\mathcal{B}^{\on{c}},\on{Ext}_{\mathcal{B}}^{1,\on{CY}},\mathfrak{s})$ the arising extriangulated homotopy category. Here, we use that $\on{Ext}_{\mathcal{B}}^{1,\on{CY}}$ describes the extensions in this extriangulated category, as follows from \Cref{lem:ext}. To be precise, we also abuse notation by labeling by $\on{Ext}_{\mathcal{B}}^{1,\on{CY}}$ the second functor in the factorization 
\[ \on{Ext}_{\mathcal{B}}^{1,\on{CY}}\colon \mathcal{B}^{\on{c},\on{op}}\times \mathcal{B}^{\on{c}}\longrightarrow \on{ho}\mathcal{B}^{\on{c},\on{op}}\times \on{ho}\mathcal{B}^{\on{c}}\longrightarrow N(\on{Vect}_k)\,.\]

\begin{lemma}\label{lem:ext}
Let $X,Y\in \mathcal{B}^{\on{c}}$. Consider a fiber and cofiber sequence $X\rightarrow Z\xrightarrow{\alpha} Y$ in $\mathcal{B}$ and let $\beta\colon X\rightarrow Y[1]$ be the cofiber morphism of $\alpha$. Then $\beta$ lies in $\on{Ext}^{1,\on{CY}}_{\mathcal{B}}(X,Y)\subset \on{Ext}^{1}_{\mathcal{B}}(X,Y)$ if and only if the image of the fiber and cofiber sequence under $G$ splits. 
\end{lemma}

\begin{proof}
We show that the fiber and cofiber sequence $G(X)\rightarrow G(Z)\xrightarrow{G(\alpha)}G(Y)$ splits if and only if the cofiber morphism $G(\beta)$ vanishes. By definition, the latter is equivalent to $\beta$ being a Calabi--Yau extension. The forward implication is clear. For the converse, suppose that $G(\beta)$ is zero. Then its fiber morphism, given by $G(\alpha)$, is equivalent to $G(X)\oplus G(Y)\xrightarrow{(0,\on{id})}G(Y)$. This shows that the fiber and cofiber sequence splits.
\end{proof}

\begin{proposition}\label{prop:2cyextr}
The extriangulated category $(\on{ho}\mathcal{B}^{\on{c}},\on{Ext}_{\mathcal{B}}^{1,\on{CY}},\mathfrak{s})$ is Frobenius and extriangled $2$-Calabi--Yau. The exact $\infty$-category $(\mathcal{B}^{\on{c}},\mathcal{B}_\dagger^{\on{c}},\mathcal{B}^{\on{c},\dagger})$ is hence also Frobenius.
\end{proposition}

\begin{proof}
The Frobenius property is shown in \Cref{prop:frobex}. The statement about the $2$-Calabi--Yau property follows directly from \Cref{prop:cydual}.
\end{proof}

\begin{remark}\label{rem:shift}
The proof of \Cref{prop:frobex} shows that the suspension functor of the stable $\infty$-category $\bar{\mathcal{B}^{\on{c}}}$ arising from $(\mathcal{B}^{\on{c}},\mathcal{B}_\dagger^{\on{c}},\mathcal{B}^{\on{c},\dagger})$, see \Cref{prop:jkpw}, acts on objects as the twist functor of the spherical adjunction $G\dashv H$, or by the relative right $2$-Calabi--Yau structure equivalently as the delooping (or negative shift) of the Serre functor of $\mathcal{B}$.

It is an interesting problem to determine whether $\bar{\mathcal{B}^{\on{c}}}$ inherits a right $2$-Calabi--Yau structure from the relative right $2$-Calabi--Yau structure of $\mathcal{B}$.
\end{remark}

We conclude this section with the proof of \Cref{prop:cydual}.

\begin{proof}[Proof of \Cref{prop:cydual}.]
Part (2) follows from part (1). We thus show part (1). Let $H$ be the right adjoint of $G$ and $\on{u}\colon \on{id}_{\mathcal{B}}\rightarrow HG$ be the unit. Note that there  is a commutative triangle as follows. 
\begin{equation}\label{eq:cutri}
\begin{tikzcd}
{\on{Map}_{\mathcal{B}}(\mhyphen_1,\mhyphen_2)} \arrow[r, "G"] \arrow[rd, "\on{u}\circ\, \mhyphen"'] & {\on{Map}_{\mathcal{A}}(G(\mhyphen_1),G(\mhyphen_2))} \arrow[d, "\simeq"] \\
                                                                                                  & {\on{Map}_{\mathcal{B}}(\mhyphen_1,HG(\mhyphen_2))}                      
\end{tikzcd}
\end{equation}
We thus have a commutative diagram in $\on{Fun}(\mathcal{B}^{\on{c},\on{op}}\times\mathcal{B}^{\on{c}},\mathcal{D}(\kappa))$
\[
\begin{tikzcd}
{\on{Mor}_{\mathcal{B}}^{\on{CY}}(\mhyphen_1,\mhyphen_2)} \arrow[d] \arrow[r] & {\on{Mor}_{\mathcal{B}}(\mhyphen_1,\mhyphen_2)} \arrow[d, "\on{u}\circ\, \mhyphen"] \\
0 \arrow[r]                                                                   & {\on{Mor}_{\mathcal{B}}(\mhyphen_1,HG(\mhyphen_2)))}                             
\end{tikzcd}
\]
which induces by the relative Calabi--Yau structure a natural transformation 
\[\xi\colon \on{Mor}_{\mathcal{B}}^{\on{CY}}(\mhyphen_1,\mhyphen_2)\rightarrow \on{fib}(\on{u}\circ\,\mhyphen)\simeq  \on{Mor}_{\mathcal{B}}(\mhyphen_1,\on{id}_{\mathcal{B}}^\ast(\mhyphen_2)[-n])\,.\] 
Consider the natural transformation $\nu\colon \on{Mor}_{\mathcal{B}}^{\on{CY}}(\mhyphen_1,\mhyphen_2)\rightarrow \on{Mor}_{\mathcal{B}}^{\on{CY}}(\mhyphen_2[-n],\mhyphen_1)^\ast$ appearing in the following commutative diagram.
\[
\begin{tikzcd}
{\on{Mor}_{\mathcal{B}}(\mhyphen_1,\on{id}_{\mathcal{B}}^\ast(\mhyphen_2)[-n])} \arrow[r, "\simeq"]                    & {\on{Mor}_{\mathcal{B}}(\on{id}_{\mathcal{B}}^*(\mhyphen_2)[-n],\on{id}_{\mathcal{B}}^*(\mhyphen_1))^\ast} \arrow[d]                     &                                                                    \\
{\on{Mor}_{\mathcal{B}}^{\on{CY}}(\mhyphen_1,\mhyphen_2)} \arrow[u, "\xi"] \arrow[r] \arrow[rr, "\nu"', bend right=10] & {\on{Mor}_{\mathcal{B}}^{\on{CY}}(\on{id}_{\mathcal{B}}^*(\mhyphen_2)[-n],\on{id}_{\mathcal{B}}^*(\mhyphen_1))^\ast} \arrow[r, "\simeq"] & {\on{Mor}_{\mathcal{B}}^{\on{CY}}(\mhyphen_2[-n],\mhyphen_1)^\ast}
\end{tikzcd}
\] To prove part (1), we show that $\nu$ is a natural equivalence, meaning that it evaluates at any pair $X,Y\in \mathcal{B}^c$ to an equivalence. Consider the sequence $HG(Y)[-1]\rightarrow \on{id}_{\mathcal{B}}^*(Y)[-n]\rightarrow Y\xrightarrow{\on{u}} HG(Y)$ in $\mathcal{B}^{\on{c}}$, where any three consecutive terms form a fiber and cofiber sequence. Applying $\on{Mor}_{\mathcal{B}}(X,\mhyphen)$, we obtain the upper sequence in $\mathcal{D}(\kappa)$ in the following diagram,
\[ 
\begin{tikzcd}[column sep=small]
{\on{Mor}_{\mathcal{B}}(X,HG(Y)[-1])} \arrow[r] & {\on{Mor}_{\mathcal{B}}(X,\on{id}_{\mathcal{B}}^*(Y)[-n])} \arrow[r, "\alpha"] & {\on{Mor}_{\mathcal{B}}(X,Y)} \arrow[r]                           & {\on{Mor}_{\mathcal{B}}(X,HG(Y))} \\
0 \arrow[r] \arrow[u]                            & {\on{Mor}_{\mathcal{B}}^{\on{CY}}(X,Y)} \arrow[r, "\on{id}"] \arrow[u, hook]  & {\on{Mor}_{\mathcal{B}}^{\on{CY}}(X,Y)} \arrow[r] \arrow[u, hook] & 0 \arrow[u]                      
\end{tikzcd}
\]
which arises from extending the rightmost square by using that any three consecutive horizontal terms form a fiber and cofiber sequence. Note also that the inclusions $\on{Mor}_{\mathcal{B}}^{\on{CY}}(X,Y)\subset \on{Mor}_{\mathcal{B}}(X,\on{id}_{\mathcal{B}}^*(Y)[-n]),\on{Mor}_{\mathcal{B}}(X,Y)$ split, as any morphism in $\mathcal{D}(\kappa)$ splits (into a direct sum of equivalences and zero morphisms). It follows that the Calabi--Yau morphism object is the maximal simultaneous direct summand of both $\on{Mor}_{\mathcal{B}}(X,Y)$ and $\on{Mor}_{\mathcal{B}}(X,\on{id}_{\mathcal{B}}^*(Y)[-n])$, being preserved by $\alpha$. 

Using the natural equivalence $\on{Mor}_{\mathcal{B}}(\mhyphen_1,\mhyphen_2)\simeq \on{Mor}_{\mathcal{B}}(\mhyphen_2,\on{id}_{\mathcal{B}}^*(\mhyphen_1))^\ast$, we see that the upper part of the above diagram is equivalent to the upper part of the following diagram.

\begin{adjustwidth}{-0.9in}{-0.9in}
\begin{equation}\label{eq:dualcy}
\begin{tikzcd}[column sep=small]
{\on{Mor}_{\mathcal{B}}(HG(Y)[-1],\on{id}_{\mathcal{B}}^*(X))^\ast} \arrow[r, "\delta"] \arrow[d] & {\on{Mor}_{\mathcal{B}}(\on{id}_{\mathcal{B}}^*(Y)[-n],\on{id}_{\mathcal{B}}^*(X))^\ast} \arrow[r, "\beta"] \arrow[d, two heads]  & {\on{Mor}_{\mathcal{B}}(Y,\on{id}_{\mathcal{B}}^*(X))^\ast} \arrow[r] \arrow[d, two heads]                            & {\on{Mor}_{\mathcal{B}}(HG(Y),\on{id}_{\mathcal{B}}^*(X))^\ast} \arrow[d] \\
0 \arrow[r]                                                                                      & {\on{Mor}_{\mathcal{B}}^{\on{CY}}(\on{id}_{\mathcal{B}}^*(Y)[-n],\on{id}_{\mathcal{B}}^*(X))^\ast} \arrow[r, "\on{id}"] \arrow[r] & {\on{Mor}_{\mathcal{B}}^{\on{CY}}(\on{id}_{\mathcal{B}}^*(Y)[-n],\on{id}_{\mathcal{B}}^*(X))^\ast} \arrow[r] \arrow[r] & 0                                                                        
\end{tikzcd}\end{equation}
\end{adjustwidth}

The left adjoint $F$ of $G$ is given by $(\on{id}_{\mathcal{B}}^*[1-n])^{-1}\circ H$, where $\on{id}_{\mathcal{B}}^*[1-n]$ is equivalent to the twist functor of the adjunction $G\dashv H$, see \cite[Corollary 2.5.16]{DKSS21}. Furthermore by \cite[Lemma 2.11]{Chr20}, the morphism (from the above sequence in $\mathcal{B}^c)$
\[ FG\circ \on{id}_{\mathcal{B}}^*(Y)[-n]\simeq HG(Y)[-1]\rightarrow \on{id}_{\mathcal{B}}^*(Y)[-n]\]
is a counit morphism of the adjunction $F\dashv G$, and thus adjoint under $FG\dashv HG$ to the unit of $G\dashv H$. It follows that the morphism $\delta$ in the above sequence is equivalent to the dual of the morphism
\begin{align*} 
{\on{Mor}_{\mathcal{B}}(\on{id}_{\mathcal{B}}^*(Y)[-n],\on{id}_{\mathcal{B}}^*(X))} \xlongrightarrow{\on{u}\circ \mhyphen}& {\on{Mor}_{\mathcal{B}}( \on{id}_{\mathcal{B}}^*(Y)[-n],HG\circ\on{id}_{\mathcal{B}}^*(X))}\,.\\
 \simeq  &~ {\on{Mor}_{\mathcal{B}}(FG\circ \on{id}_{\mathcal{B}}^*(Y)[-n],\on{id}_{\mathcal{B}}^*(X))} \\ 
  \simeq  &~ {\on{Mor}_{\mathcal{B}}(HG(Y)[-1],\on{id}_{\mathcal{B}}^*(X))} 
\end{align*}
arising from postcomposition with the unit $\on{u}$ of $G\dashv H$. By the commutativity of the diagram \eqref{eq:cutri}, this shows that the Calabi--Yau morphism object $\on{Mor}_{\mathcal{B}}(\on{id}_{\mathcal{B}}^*(Y)[-n],\on{id}_{\mathcal{B}}^*(X))$ fits into the diagram \eqref{eq:dualcy} as indicated. Again, the vertical morphisms split and ${\on{Mor}_{\mathcal{B}}(\on{id}_{\mathcal{B}}^*(Y)[-n],\on{id}_{\mathcal{B}}^*(X))}$ forms the maximal direct summand which is preserved by $\beta$. This shows that the composite
\begin{adjustwidth}{-0.3in}{-0.3in} 
\[ \on{Mor}_{\mathcal{B}}^{\on{CY}}(X,Y)\hookrightarrow \on{Mor}_{\mathcal{B}}(X,\on{id}_{\mathcal{B}}^*(Y)[-n])\simeq \on{Mor}_{\mathcal{B}}(\on{id}_{\mathcal{B}}^*(Y)[-n],\on{id}_{\mathcal{B}}^*(X))^\ast\twoheadrightarrow \on{Mor}_{\mathcal{B}}^{\on{CY}}(\on{id}_{\mathcal{B}}^*(Y)[-n],\on{id}_{\mathcal{B}}^*(X))^* \]
\end{adjustwidth}
and hence also $\nu$ evaluated at $(X,Y)$ are equivalences. This concludes the proof of part (1) and the proof.
\end{proof}

\section{The geometric model}\label{sec5}

In this section, we describe a geometric model for the generalized cluster category $\mathcal{C}_{\bf S}$, including a classification of all indecomposable objects in terms so called matching curves, which we introduce in \Cref{sec5.1}. In \Cref{sec5.2}, we recall the results from \cite{Chr21b}, which associate objects of $\mathcal{C}_{\bf S}$ to matching curves and describe the Homs in terms of the intersections of the curves. In \Cref{sec5.5}, we compare the computation of cones in $\mathcal{C}_{\bf S}$ of morphisms arising from intersections with the Kauffman Skein relation. The more technical aspects are discussed in \Cref{sec5.3,sec5.4}. In \Cref{sec5.3}, we give a review of the construction of the objects from matching curves in \cite{Chr21b}. In the final \Cref{sec5.4}, we prove the geometrization \Cref{thm:geom}, which states that every compact object in $\mathcal{C}_{\bf S}$ decomposes uniquely into the direct sum of objects associated to matching curves.

For the remainder of this article, we will always assume that $k$ is an algebraically closed field. For the entirety of this section, we also fix a marked surface ${\bf S}$ and an auxiliary \idtr{} $\mathcal{T}$ of ${\bf S}$. 

\subsection{Matching curves}\label{sec5.1}

\begin{definition}\label{def:curve}
An allowed curve is a continuous map $\gamma\colon U\rightarrow {\bf S}\backslash M$ with $U=[0,1],S^1$, satisfying that 
\begin{enumerate}
\item all existent endpoints of $\gamma$ (possibly none) lie in $\partial {\bf S}\backslash M$.
\item away from the endpoints, $\gamma$ is disjoint from $\partial {\bf S}$.
\item $\gamma$ does not cut out an unmarked disc\footnote{By cutting out, we mean that a connected part of the curve bounds an unmarked marked disc.} in ${\bf S}$.
\item if $U=S^1$, then $\gamma$ is not homotopic to the composite of multiple identical closed curves.
\end{enumerate}
An open matching curve $\gamma$ in ${\bf S}$ is an equivalence class of allowed curves under homotopies relative $\partial {\bf S}\backslash M$ with domain $U=[0,1]$ and considered up to reversal of orientation. We say that the rank of $\gamma$ is $1$.

A closed matching curve $\gamma$ in ${\bf S}$ is an equivalence class of allowed curves under homotopies relative $\partial {\bf S}\backslash M$ with domain $U=[0,1]$ together with choices of an integer $a\geq 1$ and $\lambda\in k^{\times}$. We call $a$ the rank of $\gamma$ and $\lambda$ the monodromy datum\footnote{The corresponding monodromy matrix is the $a\times a$-Jordan block with eigenvalue $\lambda$.}. We consider closed matching curves up to simultaneously reversing the orientation and replacing the monodromy datum $\lambda$ by its inverse $\lambda^{-1}$. We say that two closed matching curves are distinct if either their underlying curves differ.
\end{definition}

\begin{definition}\label{intdef}
Let $\gamma:U\rightarrow {\bf S},\,\gamma'\colon U'\rightarrow {\bf S}$ be two matching curves. We choose representatives of $\gamma$ and $\gamma'$ with the minimal number of intersections. 
\begin{itemize}
\item A crossing of $\gamma$ and $\gamma'$ is an intersection of $\gamma$ and $\gamma'$ away from their endpoints. We denote the set of crossings of $\gamma$ and $\gamma'$ by $i^{\on{cr}}(\gamma,\gamma')$. If $\gamma=\gamma'$, then $i^{\on{cr}}(\gamma,\gamma)$ counts each self-crossing only once.
\item A directed boundary intersection from $\gamma$ to $\gamma'$ is an intersection of both $\gamma$ and $\gamma'$ with the same connected component $B$ of $\partial {\bf S}\backslash M$ such that the intersection of $\gamma'$ with $B$ follows the intersection of $\gamma$ with $B$ in the orientation of $B$ induced by the (clockwise) orientation of ${\bf S}$. We denote the set of directed boundary intersection from $\gamma$ to $\gamma'$ by $i^{\on{bdry}}(\gamma,\gamma')$. If $\gamma=\gamma'$, then $i^{\on{bdry}}(\gamma,\gamma)$ only counts those directed boundary intersections arising from two distinct ends of $\gamma$ lying on the same component of $\partial {\bf S}\backslash M$.
\end{itemize}
\end{definition}

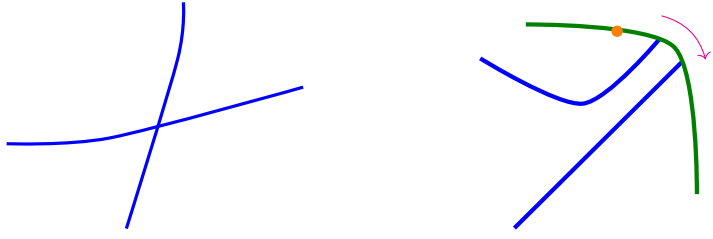
\begin{figure}[ht]\label{figta}
\begin{center}
\begin{tikzpicture}[scale=1.5]
  \draw[color=blue][very thick] plot [smooth] coordinates {(0.75,0) (1.2,1.5) (1.25,2)};
   \draw[color=blue][very thick] plot [smooth] coordinates {(-0.3,0.75) (0.6,0.8) (2.3, 1.25)};
  \node (0) at (1.33, 1.33){};
  \node (1) at (0.67, 0.67){};
\end{tikzpicture}
\quad\quad\quad
\begin{tikzpicture}[scale=1.5]
\draw[color=blue][ultra thick] plot [smooth] coordinates {(-0.9,0.7) (0,0.3) (0.68,0.88)};
\draw[color=blue][ultra thick] plot [smooth] coordinates {(-0.6,-0.8) (0.88,0.68)};
\node (1) at (0.6,1.1){};
\node (2) at (1.1,0.6){} edge [<-, bend right, color=magenta] (1);
\node (3) at (-1.5,-0.5){};
\node (4) at (1.5,-0.5){};
\node (5) at (0,1){};
\draw[color=ao][ultra thick] plot [smooth] coordinates {(-0.5,1)  (0.8,0.8) (1,-0.5)};
\node (0) at (0.3,0.94){};
\fill[color=orange] (0) circle (0.05);
\end{tikzpicture}
\caption{A crossing and  a direct endpoint intersection of two matching curves. The boundary of the surface is depicted in green. The magenta arrow indicates the clockwise direction. The orange vertex describes a marked point.}
\end{center}
\end{figure}

\begin{definition}\label{def:arc}
An open matching curve $\gamma\colon [0,1]\rightarrow {\bf S}$ is called an arc if it has no self-crossings. An arc is called a boundary arc, if it cuts out a monogon. An arc is called an internal arc, if it is not a boundary arc.
\end{definition}

\begin{remark}
The notion of an internal arc in the sense of \Cref{def:arc} coincides with the notion of an arc from \cite[Def.~5.2]{FT12}, except that we use a different convention regarding endpoints. In this paper, endpoints lie in $\partial{\bf S}\backslash M$, whereas in loc.~cit.~endpoints lie in $M$. These two perspectives are however equivalent, as can be seen by expanding the marked points to intervals, contracting their complements to points and using that arcs are considered up to homotopy.
\end{remark}

\begin{definition}[See e.g.~\cite{FST08}]\label{def:idealtr}
Two arcs in ${\bf S}$ are called compatible, if they do not have any crossings. An ideal triangulation of ${\bf S}$ consists of a maximal (by inclusion) collection $\I$ of pairwise compatible arcs in ${\bf S}$.  
\end{definition}

We note that any ideal triangulation of ${\bf S}$ contains all boundary arcs of ${\bf S}$ and that any two ideal triangulations of ${\bf S}$ have the same cardinality.

\subsection{Objects and Homs}\label{sec5.2} 

The generalized cluster category $\mathcal{C}_{\bf S}$ embeds fully faithfully into the $\infty$-category $\mathcal{H}({\mathcal{T}},\mathcal{F}_{\mathcal{T}})$ of global sections of the perverse schober $\mathcal{F}_{\mathcal{T}}$ as the full subcategory spanned by global sections whose value at any edge $e\in \mathcal{T}_1$ lies in $\mathcal{D}(k[t_1^\pm])\subset \mathcal{D}(k[t_1])=\mathcal{F}_\mathcal{T}(e)$. To describe the objects and morphism objects in the generalized cluster category $\mathcal{C}_{\bf S}$, we may thus make use of the partial geometric model for $\mathcal{H}({\mathcal{T}},\mathcal{F}_{\mathcal{T}})$ of \cite{Chr21b}. 

Consider a matching curve $\gamma$ in ${\bf S}$. Such a curve defines by \cite[Lemma 5.8]{Chr21b} a matching datum $(\gamma,k[t_1^\pm])$ in the sense of \cite[Def.~5.9]{Chr21b} with $\gamma$ a finite, pure matching curve in ${\bf S}\backslash M$ in the sense of \cite[Def.~5.5]{Chr21b}. By the results of Section 5.2.~in loc.~cit., we thus get for each matching curve $\gamma$ in ${\bf S}$ a global section $M_{\gamma}\coloneqq M_{\gamma}^{k[t_1^\pm]}\in \mathcal{C}_{\bf S}\subset \mathcal{H}({\mathcal{T}},\mathcal{F}_{\mathcal{T}})$. We further have the following.

\begin{proposition}\label{prop:indec}
Let $\gamma$ be a matching curve in ${\bf S}$. Then the discrete endomorphism ring of $M_{\gamma}\in \mathcal{C}_{\bf S}$ is local and $M_{\gamma}$ thus indecomposable. 
\end{proposition}

\begin{proof}
This is \cite[Cor.~6.13]{Chr21b}.
\end{proof}

The morphism objects between the $M_{\gamma}$'s can be described in terms of intersections of the curves as in the following two Theorems, see \cite[Theorems 6.4 and 6.5]{Chr21b}. 

\begin{theorem}\label{thm:hom1}
Let $\gamma,\gamma'$ be two distinct matching curves in ${\bf S}$. Let $a$ be the rank of $\gamma$ and $a'$ the rank of $\gamma'$. There exists an equivalence in $\mathcal{D}(k[t_2^\pm])$ 
\[ \on{Mor}_{\mathcal{C}_{\bf S}}(M_{\gamma},M_{\gamma'})\simeq k[t_{1}^\pm]^{\oplus aa' i^{\on{cr}}(\gamma,\gamma')\oplus i^{\on{bdry}}(\gamma,\gamma')}\,.\]
\end{theorem}

\begin{theorem}\label{thm:hom2}
\begin{enumerate}[(i)]
\item Let $\gamma\colon [0,1]\rightarrow {\bf S}$ be an open matching curve in ${\bf S}$. There exists an equivalence in $\mathcal{D}(k[t_2^\pm])$
\[ \on{Mor}_{\mathcal{C}_{\bf S}}(M_{\gamma},M_{\gamma})\simeq k[t_{1}^\pm]^{\oplus 1+2i^{\on{cr}}(\gamma,\gamma)+i^{\on{bdry}}(\gamma,\gamma)}\,.\]
\item 
Let $\gamma\colon S^1\rightarrow {\bf S}$ be a closed matching curve in ${\bf S}$ of rank $a$. There exists an equivalence in $\mathcal{D}(k[t_2^\pm])$
\[ \on{Mor}_{\mathcal{C}_{\bf S}}(M_\gamma,M_\gamma)\simeq k[t_{1}^\pm]^{\oplus 2a+2a^2i^{\on{cr}}(\gamma,\gamma)}\,.\]
\end{enumerate}
\end{theorem}

\subsection{Skein relations from mapping cones}\label{sec5.5}

In the following, we collect geometric descriptions of the cones of the morphisms between the objects of $\mathcal{C}_{\bf S}$ associated to the matching curves. Similar descriptions of cones in the derived categories of gentle algebras appear in \cite{OPS18}. We fix two open matching curves $\gamma,\gamma'$ in ${\bf S}$ and distinguish two cases.\\

\noindent {\bf Case 1}: $\gamma$ and $\gamma'$ have a crossing. 

There are two possible \textit{smoothings} of this crossing, each consisting of two matching curves in ${\bf S}$, denoted $\gamma_1,\gamma_2$ and $\gamma_3,\gamma_4$. The first curve in each of the two smoothings is obtained by starting at an endpoint of $\gamma$ and tracing along $\gamma$ up to that crossing, and then tracing along $\gamma'$ in one of two possible directions. Similarly, the second curve in the two smoothings is obtained by starting at the other endpoint of $\gamma$, tracing along $\gamma$ up to the crossing, and then tracing along $\gamma'$ to the other end. This process is locally at the crossing illustrated on the left in \Cref{crosssmooth}. 

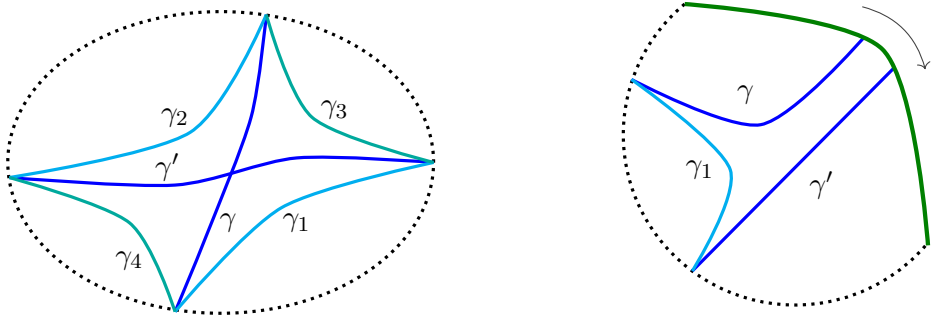
\begin{figure}[ht]
\begin{center}
\begin{tikzpicture}[scale=2]
\draw[dotted, very thick] (0.9,1) ellipse (1.4 and 1);
  \draw[color=blue][very thick] plot [smooth] coordinates {(0.6,0.01) (0.8,0.5) (1.1,1.3) (1.2,1.98)};
   \draw[color=blue][very thick] plot [smooth] coordinates {(-0.49,0.9) (0.6,0.85) (1.4,1.03) (2.3, 1)};
  \draw[color=cyan][very thick] plot [smooth] coordinates {(0.6,0.01) (1.3,0.7) (2.3, 1)};
   \draw[color=cyan][very thick] plot [smooth] coordinates {(-0.49,0.9) (0.7,1.2)  (1.2,1.98)};
 \draw[color=Emerald][very thick] plot [smooth] coordinates {(0.6,0.01) (0.3,0.6) (-0.49,0.9) };
   \draw[color=Emerald][very thick] plot [smooth] coordinates {(2.3, 1) (1.5,1.3)  (1.2,1.98)}; 
  \node (0) at (1.33, 1.33){};
  \node (1) at (0.67, 0.67){};
  \node () at (0.95,0.6){$\gamma$};
  \node () at (0.55,0.97){$\gamma'$};
  \node () at (1.4,0.6){$\gamma_1$};
  \node () at (0.6,1.3){$\gamma_2$};
  \node () at (1.65,1.35){$\gamma_3$};
  \node () at (0.3,0.35){$\gamma_4$};
\end{tikzpicture}\quad\quad\quad
\begin{tikzpicture}[scale=2]
\node (0) at (0,0){};
\draw[very thick, dotted] (-0.5,1.1) arc[start angle=130, end angle=320,radius=1.13];
\draw[color=blue][very thick] plot [smooth] coordinates {(-0.85,0.6) (0,0.3) (0.68,0.88)};
\draw[color=blue][very thick] plot [smooth] coordinates {(-0.45,-0.67) (0.88,0.68)};
\draw[color=cyan][very thick] plot [smooth] coordinates {(-0.85,0.6) (-0.2,0) (-0.45,-0.67)};
\node ()  at (-0.1,0.5){$\gamma$};
\node () at (0.4,-0.1){$\gamma'$};
\node ()  at (-0.4,0){$\gamma_1$};
\node (1) at (0.6,1.1){};
\node (2) at (1.1,0.6){} edge [<-, bend right, color=darkgray] (1);
\node (3) at (-1.5,-0.5){};
\node (4) at (1.5,-0.5){};
\node (5) at (0,1){};
\draw[color=ao][ultra thick] plot [smooth] coordinates {(-0.5,1.1)  (0.8,0.8) (1.1,-0.5)};
\end{tikzpicture}
\caption{On the left: a crossing of two matching curves $\gamma,\gamma$ (in blue) and the two possible smoothings $\gamma_1,\gamma_2$ and $\gamma_3,\gamma_4$. On the right: a directed boundary intersection of two matching curves $\gamma,\gamma'$ and the corresponding smoothed composite $\gamma_1$. The boundary of ${\bf S}$ is depicted in green. Outside of the depicted parts of ${\bf S}$, the matching curves continue identically.}
\label{crosssmooth}\label{compsmoth}\label{figta2}
\end{center}
\end{figure}

\noindent{\bf Case 2}: there is a directed boundary intersection from $\gamma$ to $\gamma'$. 

We can compose $\gamma$ with part of a boundary component of ${\bf S}\backslash M$ and $\gamma'$ to a curve, which we then smooth to a matching curve $\gamma_1$. This process is illustrated on the right in \Cref{compsmoth}.\\

To both types of intersection, \Cref{thm:hom1,thm:hom2} associate a direct summand of the morphism object $\on{Mor}_{\mathcal{C}_{\bf S}}(M_{\gamma},M_{\gamma'})$ and also $\on{Mor}_{\mathcal{C}_{\bf S}}(M_{\gamma},M_{\gamma'})$ in case of a crossing. \Cref{prop:cone} describes the cones of these morphisms in terms of the above smoothings. 

\begin{proposition}\label{prop:cone}
Let $\gamma,\gamma'$ be two matching curves in ${\bf S}$. 
\begin{enumerate}[(1)]
\item Suppose that $\gamma$ and $\gamma'$ have a crossing. 
There exist fiber and cofiber sequences in $\mathcal{C}_{\bf S}$ 
\[ M_{\gamma_1}\oplus M_{\gamma_2}\rightarrow M_{\gamma}\xrightarrow{\alpha}M_{\gamma'}\,,\quad\quad M_{\gamma_3}\oplus M_{\gamma_4}\rightarrow M_{\gamma'}\xrightarrow{\beta}M_{\gamma}\,,\]
with $\gamma_1,\gamma_2$ and $\gamma_3,\gamma_4$ being the two possible smoothings of the crossing. The morphisms $\alpha$ and $\beta$ describe any non-zero degree $0$ elements of the direct summands $ k[t_1^\pm]\subset \on{Mor}_{\mathcal{C}_{\bf S}}(M_{\gamma},M_{\gamma'})$, $k[t_1^\pm]\subset \on{Mor}_{\mathcal{C}_{\bf S}}(M_{\gamma'},M_{\gamma})$ associated to the crossing in \Cref{thm:hom1}.
\item Suppose that there is a directed boundary intersection from $\gamma$ to $\gamma'$. There exist a fiber and cofiber sequence in $\mathcal{C}_{\bf S}$ 
\[ M_{\gamma_1}\rightarrow M_{\gamma}\xrightarrow{\alpha}M_{\gamma'}\,,\]
with $\gamma_1$ the smoothed composite of $\gamma,\gamma'$. The morphism $\alpha$ describes any non-zero degree $0$ element of the direct summand $k[t_1^\pm]\subset \on{Mor}_{\mathcal{C}_{\bf S}}(M_{\gamma},M_{\gamma'})$ associated to the directed boundary intersection in \Cref{thm:hom1}.
\end{enumerate}
\end{proposition}

\begin{proof}
The Proposition follows from a direct computation, using the descriptions of $M_{\gamma}$ and $M_{\gamma'}$ as coCartesian sections of the Grothendieck construction of $\mathcal{F}_{\mathcal{T}}^{\on{clst}}$.
\end{proof}

\begin{remark}\label{rem:ptlmycone}
Part (1) of \Cref{prop:cone} matches the $q=1$ Kauffman Skein relations, or for arcs equivalently the Ptolemy cluster exchange relations, see \Cref{clrel,def:skein}.
\end{remark}

\subsection{Review of the gluing construction of global sections}\label{sec5.3}

Before we state and prove the geometrization Theorem in \Cref{sec5.4}, we recall the construction of the global sections from matching curves in \cite{Chr21b}. 

\begin{definition}
Let $v$ be a vertex of $\mathcal{T}$. A segment at $v$ is an embedded curve $\delta\colon [0,1]\rightarrow {\bf S}$ lying in a contractible neighborhood of $v$ whose endpoints lie on two distinct edges of ${\mathcal{T}}$ incident to $v$ (but not on $v$). 

We consider segments at a vertex $v$ of $\mathcal{T}$ as equivalence classes under homotopies $\Delta\colon [0,1]^2\rightarrow {\bf S}$, which satisfy that $\Delta(t)$ is a segment at $v$ for all $t\in [0,1]$. A segment in ${\bf S}$ is a segment at any vertex $v$ of $\mathcal{T}$.
\end{definition}

It is straightforward to see that every matching curve in ${\bf S}$ arises in a unique way by composing finitely many segments in ${\bf S}$, such that whenever two segments are composed, they lie at distinct vertices of ${\mathcal{T}}$. 

\begin{figure}[ht]
\begin{center}
\begin{tikzpicture}
\node (0) at (0,0){};
\fill (0) circle (0.1);
\draw[color=ao][very thick] 
(-2.25,-0.75)--(2.25,-0.75)
(0,1.5) --(-2.25,-0.75)
(0,1.5) --(2.25,-0.75);
\draw[very thick]
(0,0)--(0,-0.75)
(0,0)--(0.8,0.71)
(0,0)--(-0.8,0.71);
\node (2) at (-2.25,-0.75){};
\node (3) at (2.25,-0.75){};
\node (4) at (0,1.5){};
\fill[color=orange] (2) circle (0.1);
\fill[color=orange] (3) circle (0.1);
\fill[color=orange] (4) circle (0.1);

\draw[color=blue][very thick] plot [smooth] coordinates {(0.8,0.71) (0.2,0.35) (-0.2,0.35) (-0.8,0.71)};
\end{tikzpicture}
\end{center}
\caption{One of the three possible segments (in blue) at a vertex of $\mathcal{T}$.}
\end{figure}
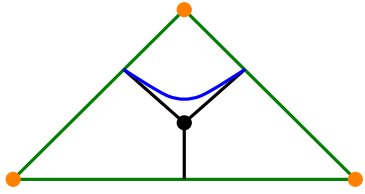

To each segment $\delta$ at a vertex $v$ of $\mathcal{T}$, there is an associated local section $M_{\delta}\in \mathcal{L}(\mathcal{T},\mathcal{F}_{\mathcal{T}}^{\on{clst}})$ in the sense of \Cref{def:glsec}. Such a local section $M$ can be identified with the data of 
\begin{itemize}
\item for each object $x\in \on{Exit}(\mathcal{T})$ an object $M(x)$ in the stable $\infty$-category $\mathcal{F}^{\on{clst}}_{\mathcal{T}}(x)$ and 
\item for each morphism $x\rightarrow y$ in $\on{Exit}(\mathcal{T})$ a morphism $\mathcal{F}^{\on{clst}}_{\mathcal{T}}(x\rightarrow y)(M(x))\rightarrow M(y)$ in $\mathcal{F}^{\on{clst}}_{\mathcal{T}}(y)$. If all these morphisms are equivalences (hence describe coCartesian morphisms), the local section is a global section. 
\end{itemize}
The local section $M_{\delta}$ associated to the segment $\delta$ is a left Kan extension relative the Grothendieck construction $p\colon \Gamma(\mathcal{F}_{\mathcal{T}}^{\on{clst}})\rightarrow \on{Exit}(\mathcal{T})$ along the inclusion $\{v\}\hookrightarrow \on{Exit}(\mathcal{T})$ of $M_{\delta}(v)\in \mathcal{F}_{\mathcal{T}}^{\on{clst}}(v)$. In particular, this local section is fully determined by its value $M_{\delta}(v)$ at $v\in \on{Exit}(\mathcal{T})$. Spelling out the Kan extension, one finds that for each edge $e$ incident to $v$, $M_\delta(e)$ is equivalent to the image of $M_{\delta}(v)$ under the functor $\mathcal{F}_{\mathcal{T}}^{\on{clst}}(v\rightarrow e)$. Further, if $x\in \on{Exit}(\mathcal{T})$ is neither $v$ not an edge incident to $v$, then $M_{\delta}(x)$ vanishes.

Let $v$ be a vertex of $\mathcal{T}$. There are three indecomposable objects in $\mathcal{F}_{\mathcal{T}}^{\on{clst}}(v)\simeq \on{Fun}(\Delta^1,\mathcal{D}(k[t_1^\pm]))$, given as follows:
\begin{align*}
\left(0\rightarrow k[t_1^\pm]\right)\,,~\left(k[t_1^\pm]\xrightarrow{\simeq} k[t_1^\pm]\right)\,,~\left(k[t_1^\pm]\rightarrow 0\right)\,.
\end{align*}
These three objects describe the values of the three local sections associated to the three possible segments at $v$. Let $e_1,e_2,e_3$ be the three edges incident to $v$. Each of the above objects satisfies that applying $\mathcal{F}_{\mathcal{T}}^{\on{clst}}(v\rightarrow e_i)$ yields $k[t_1^\pm]$ in two cases and $0$ in one case. We choose the value $M_{\delta}(v)$ so that the set of edges of $\mathcal{T}$ where $M_{\delta}$ evaluates nontrivially as $k[t_1^\pm]$ consists exactly of the two edges at which $\delta$ begins and ends.  

To produce a global section associated to a matching curve $\gamma$, we glue the local sections associated to the segments of $\gamma$. In the simplest case of gluing sections from two segments, this goes as follows. Let $e$ be an edge of $\mathcal{T}$ with incident vertices $v,v'$ and let $\delta$ and $\delta'$ be two segments at $v$, respectively $v'$, which end at $e$. We have $M_{\delta}(e)\simeq M_{\delta'}(e)\simeq k[t_1^\pm]$. Define $Z_e\in \mathcal{L}(\mathcal{T},\mathcal{F}_{\mathcal{T}}^{\on{clst}})$ via $Z_e(x)=0$ if $x\neq e$ and $Z_e(e)=k[t_1^\pm]$. There are apparent pointwise inclusions of $Z_e$ into $M_{\delta}$ and $M_{\delta'}$, which restrict at $e$ to the identity on $k[t_1^\pm]$. The gluing of $M_{\delta}$ and $M_{\delta'}$ is now defined as the pushout in $\mathcal{L}(\mathcal{T},\mathcal{F}_{\mathcal{T}}^{\on{clst}})$ of the following diagram.
\begin{equation}\label{eq:zmpu}
\begin{tikzcd}
           & Z_e \arrow[ld] \arrow[rd] &             \\
M_{\delta} &                           & M_{\delta'}
\end{tikzcd}
\end{equation}
More generally, given any curve $\eta$ composed of segments, we can glue a section $M_{\eta}\in \mathcal{L}(\mathcal{T},\mathcal{F}_{\mathcal{T}}^{\on{clst}})$ by repeatedly taking pushouts as above. If $\eta$ is an open matching curve, the section $M_{\eta}$ is even a global sections, i.e.~lies in $\mathcal{H}(\mathcal{T},\mathcal{F}_{\mathcal{T}}^{\on{clst}})\subset \mathcal{L}(\mathcal{T},\mathcal{F}_{\mathcal{T}}^{\on{clst}})$.

Consider finally a closed matching curve $\gamma$ of rank $a$ with monodromy datum $\lambda\in k^{\times}$. Let $\mathscr{J}_{\lambda}\colon k[t_1^\pm]^{\oplus a}\simeq k[t_1^\pm]^{\oplus a}$ be the $a\times a$-Jordan block with eigenvalue $\lambda$. We cut $\gamma$ at any edge of $\mathcal{T}$ into an open curve $\eta$ which is composed of segments. The global section $M_{\gamma}\in \mathcal{H}(\mathcal{T},\mathcal{F}_{\mathcal{T}}^{\on{clst}})\subset \mathcal{L}(\mathcal{T},\mathcal{F}_{\mathcal{T}}^{\on{clst}})$ is obtained as the coequalizer of 
\begin{equation}\label{eq:ii}
\begin{tikzcd}
Z_e^{\oplus a} \arrow[r, "\iota_1^{\oplus a}", shift left] \arrow[r, "\iota_2^{\oplus a}\circ \mathscr{J}_{\lambda}^{-1}"', shift right] & M_{\eta}^{\oplus a}
\end{tikzcd}
\end{equation}
in $\mathcal{L}(\mathcal{T},\mathcal{F}_{\mathcal{T}}^{\on{clst}})$. Here, $\iota_1,\iota_2$ are two pointwise inclusions of $Z_e$ into $M_{\eta}$ at $e$ arising from the two ends of $\eta$ at $e$. In the way we state this definition, making different choices of Kan extensions in the construction of $M_{\eta}$ changes the section $M_{\gamma}$. One can show that the difference between the $M_{\gamma}$'s associated to two possible choices is the same as multiplying the monodromy datum $\lambda$ of $\gamma$ by a non-zero factor. One can solve this by fixing choices of these Kan extensions. Alternatively, the proof of the geometrization \Cref{thm:geom} also shows how to extract a unique monodromy datum from a global section of the form $M_{\gamma}$. 

We note that we have defined matching curves as certain curves considered up to reversal of orientation. If $\gamma$ is an open matching curve, it is easy to see that the global section $M_{\gamma}$ does not depend on the orientation of $\gamma$. If $\gamma$ is a closed matching curve, changing the orientation of $\gamma$ in the construction of $M_{\gamma}$ interchanges the two maps in \eqref{eq:ii}. Hence, to obtain the same global section, we need to replace the monodromy datum by its inverse and use that the Jordan normal form of $\mathscr{J}_{\lambda}^{-1}$ is a single Jordan block with eigenvalue $\lambda^{-1}$ .

\subsection{The geometrization Theorem}\label{sec5.4}

The main result of this section is to prove the following Theorem, stating that all compact objects in $\mathcal{C}_{\bf S}$ are geometric, meaning that they arise as direct sums of object associated with matching curves.

\begin{theorem}[The geometrization Theorem]\label{thm:geom}
Let $X\in \mathcal{C}_{\bf S}$ be a compact object. Then there exists a unique and finite set $J$ of matching curves in ${\bf S}$ satisfying that there is an equivalence in $\mathcal{C}_{\bf S}$
\[ X\simeq \bigoplus_{\gamma\in J} M_{\gamma}\,.\]
\end{theorem}

\begin{corollary}\label{thm:KS}
The full subcategory $\mathcal{C}_{\bf S}^c\subset \mathcal{C}_{\bf S}$ of compact objects is Krull-Schmidt, meaning that every objects splits into a direct sum of objects with local endomorphism rings. 
\end{corollary}

\begin{proof}[{Proof of \Cref{thm:KS}.}]
As shown in the proof of \Cref{thm:geom}, every compact object $X$ of $\mathcal{C}_{\bf S}$ splits as the direct sum $X\simeq \bigoplus_{\gamma \in J}M_\gamma$, with $J$ a set of matching curves. The endomorphism rings of these objects are local by \Cref{prop:indec}.
\end{proof}

The proof of the geometrization Theorem makes use of the gluing description of global sections arising from matching curves and the description of global sections in terms of coCartesian sections of the Grothendieck construction. Given a compact object $X$ of $\mathcal{C}_{\bf S}$, considered as a global section of $\mathcal{F}_{\mathcal{T}}^{\on{clst}}$, we can evaluate it at any edge $e$ of the ribbon graph $\mathcal{T}$. If the value is non-zero we can choose a direct sum decomposition of this value with a non-zero direct summand. Let $v$ be a vertex incident to $e$. As we show in \Cref{constr:geom1}, we can almost find a lift along the functor $\mathcal{F}_{\mathcal{T}}^{\on{clst}}(v\rightarrow e)$ of the direct sum decomposition at $e$ to a direct sum decomposition of the value of $X$ to $v$. This then give rise to direct sum decompositions at further edges incident to $v$. Proceeding this way, we choose repeated direct sum decompositions of the values of $X$ at all vertices and edges of $\mathcal{T}$. We then observe that we can further tweak the chosen direct sum decompositions so that they are all compatible and thus glue to a global section, which is a non-zero direct summands of $X$. The summand of $X$ so constructed is obtained from gluing local sections associated to segments and hence geometric, i.e.~arises from a matching curve. Repeating this argument, we find a splitting of $X$ into geometric direct summands.  The uniqueness of the set $C$ in the geometrization Theorem is proven separately.

\begin{construction}\label{constr:geom1}
Let ${\bf S}$ be the $3$-gon and $e$ an edge of a \idtr{} $\mathcal{T}$ of ${\bf S}$. Let $X\in \mathcal{C}_{\bf S}^{\on{c}}$ and let $\on{ev}_e(X)\simeq B\oplus A\oplus B'$ with $A\neq 0$. We construct two direct sum decompositions $X\simeq Z\oplus Y \oplus Z'$ and $X\simeq \tilde{Z}\oplus Y \oplus \tilde{Z}'$, called the counterclockwise and clockwise splittings. These satisfy that the arising equivalences 
\[ \on{ev}_e(Z)\oplus \on{ev}_e(Y)\oplus \on{ev}_e(Z')\simeq B\oplus A\oplus B'\]
\[ \on{ev}_e(\tilde{Z})\oplus \on{ev}_e(Y)\oplus \on{ev}_e(\tilde{Z}')\simeq B\oplus A\oplus B'\]
are lower triangular with invertible diagonals.

The generalized cluster category $\mathcal{C}_{\bf S}$ of ${\bf S}$ is equivalent to the functor category $\on{Fun}(\Delta^1,\mathcal{D}(k[t_1^\pm]))$, whose objects are diagrams $x\rightarrow y$ with $x,y\in \mathcal{D}(k[t_1^\pm])$. In the following, we set $\mathcal{C}_{\bf S}=\on{Fun}(\Delta^1,\mathcal{D}(k[t_1^\pm]))$. The functor $\on{ev}_{e}\colon \mathcal{C}_{\bf S}\rightarrow \mathcal{D}(k[t_1^\pm])$ can be chosen to evaluate a functor $\Delta^1\rightarrow \mathcal{D}(k[t_1^\pm])$ at its value at $1\in \Delta^1$ (recall that $\Delta^1$ has the vertices $0,1$). There are three indomposable objects in $\mathcal{C}_{\bf S}$, corresponding to the three boundary arcs in ${\bf S}$ (there are no further matching curves in ${\bf S}$). Explicitly, these objects can be described as follows.
\begin{align*}
M_1&= \left(0\rightarrow k[t_1^\pm]\right)\\
M_2&= \left(k[t_1^\pm]\xrightarrow{\simeq} k[t_1^\pm]\right)\\
M_3&= \left(k[t_1^\pm]\rightarrow 0\right)
\end{align*}
The homotopy category of $\mathcal{D}(k[t_1^\pm])$ is equivalent to the abelian $1$-category of $k$-vector spaces, see \Cref{lem:vectk}. We thus treat objects in $\mathcal{D}(k[t_1^\pm])$ as vector spaces in the following. 

All compact objects in $\mathcal{D}(k[t_1^\pm])$ are a finite direct sums of copies of $k[t_1^\pm]$. Using that every morphism in $\mathcal{D}(k[t_1^\pm])$ splits, we may assume that the object $X\in \mathcal{C}_{\bf S}$ is equivalent to an object of the form 
\[
k[t_1^\pm]^{\oplus j}\oplus k[t_1^\pm]^{\oplus l}\xrightarrow{\mathscr{M}}k[t_1^\pm]^{\oplus j}\oplus k[t_1^\pm]^{\oplus i}\,,
\]
with $\mathscr{M}=\begin{pmatrix} \on{id}_{k[t_1^\pm]^{\oplus j}} & 0 \\ 0& 0\end{pmatrix}$ and $i,j,l\geq 0$. We thus have $\on{ev}_{e}(X)=k[t_1^\pm]^{\oplus j}\oplus k[t_1^\pm]^{\oplus i}$ and a splitting 
\[ X\simeq \left(k[t_1^\pm]^{\oplus l}\rightarrow 0\right)\oplus \left( k[t_1^\pm]^{\oplus j}\hookrightarrow k[t_1^\pm]^{\oplus j}\oplus k[t_1^\pm]^{\oplus i}\right)\,.\]
We denote the first summand by $O$ and the second by $U$.  Note that $\on{ev}_e(O)=0$.

Let $C=(B\oplus A)\cap k[t_1^\pm]^{\oplus j}$. We define $V$ as the diagram $C\rightarrow B\oplus A$ in $\mathcal{D}(k[t_1^\pm])$. It is easy to see that the apparent morphism $V\rightarrow U$ in $\mathcal{C}_{\bf S}$ admits a retraction. We can thus choose a splitting $U\simeq V\oplus Q'$. By construction, we have $\on{ev}_e(V)=B\oplus A$. Next, we define $D=B\cap k[t_1^\pm]^{\oplus j}$. Then we again obtain a direct summand $Q=\left(D\rightarrow B\right)\hookrightarrow V$. We choose a direct sum complement $V\simeq Q\oplus Y$. We have $\on{ev}_e(Q)\oplus \on{ev}_e(Y)\oplus \on{ev}_e(Q')\simeq B\oplus A \oplus B'$ and this equivalence is lower triangular with invertible diagonals. 

We set $Z=Q$, $Z'=Q'\oplus O$ and $\tilde{Z}=Q\oplus O$, $\tilde{Z}'=Q'$. This gives the desired splittings $X\simeq Z\oplus Y\oplus Z'$ and $X\simeq \tilde{Z}\oplus Y\oplus \tilde{Z}'$. 
\end{construction}

\begin{lemma}\label{lem:geom1}
Consider the setup of \Cref{constr:geom1}  and the counterclockwise and clockwise splittings 
\[ X\simeq Z \oplus Y \oplus Z' \quad \text{ and }\quad X\simeq \tilde{Z}\oplus Y\oplus \tilde{Z}'\,.\]
Denote by $e_1$ the edge of $\mathcal{T}$ following the edge $e$ in the counterclockwise direction and by $e_2$ the edge of $\mathcal{T}$ following the edge $e$ in the clockwise direction.
\begin{enumerate}[(1)]
\item Suppose that $Y\simeq M_1$. Then any autoequivalence $\phi$ of $\on{ev}_{e_1}(Z)\oplus \on{ev}_{e_1}(Y)\oplus \on{ev}_{e_1}(Z')$ given by a lower triangular matrix lifts to an autoequivalence of $Z\oplus Y \oplus Z'$ given by a lower triangular matrix.
\item Suppose that $Y\simeq M_2$. Any autoequivalence $\phi$ of $\on{ev}_{e_2}(\tilde{Z})\oplus \on{ev}_{e_2}(Y)\oplus \on{ev}_{e_2}(\tilde{Z}')$ given by a lower triangular matrix lifts to an autoequivalence of $\tilde{Z}\oplus Y \oplus \tilde{Z}'$ given by a lower triangular matrix.
\end{enumerate}
\end{lemma}

\begin{proof}
We only prove part (1), part (2) is analogous. By construction, we can find splittings $Z\simeq \bigoplus_{i_1}M_1\oplus \bigoplus_{j_1}M_2$ and $Z'\simeq \bigoplus_{i_2}M_1\oplus \bigoplus_{j_2}M_2\oplus \bigoplus_l M_3$. We have  $\on{ev}_{e_1}(M_2)\simeq 0$ and $\on{ev}_{e_1}(M_3)\simeq \on{ev}_{e_1}(M_1)\simeq k[t_1^\pm]$. We further find $\on{Mor}_{\mathcal{C}_{\bf S}}(M_x,M_y)\xrightarrow{\on{ev}_{e_1}}\on{Mor}_{\mathcal{D}(k[t_1^\pm])}(k[t_1^\pm],k[t_1^\pm])$ to be an equivalence for $x=y=1$, for $x=3$ and $y=1$ or for $x=y=3$. It is thus clear that we can find a unique lift of $\phi$, which restricts to the identity on $\bigoplus_{j_1}M_2 \oplus \bigoplus_{j_2}M_2$. 
\end{proof}

By \Cref{lem:geom1}, the clockwise and counterclockwise splitting have the advantage that we can lift certain autoequivalences of their values at an edge to autoequivalences of the splitting. We use these autoequivalences to tweak the choices of local splittings in the proof of the geometrization \Cref{thm:geom}, by replacing some splittings with their image under such an autoequivalence. After the necessary tweaks, we are able to glue the arising two-term splittings $Y\oplus (Z\oplus Z')\simeq X$ or $Y\oplus (\tilde{Z}\oplus \tilde{Z}')\simeq X$ to a splitting of global sections.

\begin{proof}[Proof of \Cref{thm:geom}, part 1: existence of the decomposition.]
Given an edge $e$ of $\trivalent$ and $X\in \mathcal{C}_{\bf S}$, we denote by $\on{ev}_e(X)\in \mathcal{F}_{\trivalent}^{\on{clst}}(e)\simeq \mathcal{D}(k[t_1^\pm])$ the value of the coCartesian section $X$ at $e$. Similarly, given a vertex $v$ of $\trivalent$, we denote by $\on{res}_v(X)$ the restriction of the coCartesian section $X$ to $v$ and its three incident edges. The object $\on{res}_v(X)$ can be considered as an object in $\mathcal{C}_{{\bf S}_v}$, with ${\bf S}_v$ a $3$-gon. Note also that $\mathcal{C}_{{\bf S}_v}\simeq \mathcal{F}_{\trivalent}^{\on{clst}}(v)$, with the equivalence given by evaluation at $v$. We show that if $X\neq 0$, then there is a splitting $X\simeq M_{\gamma}\oplus T$ for a matching datum $(\gamma,k[t_1^\pm])$. By a descending induction on the total dimension over $k$ of $\bigoplus_{e\in \trivalent_1}\on{ev}_e(X)$, the existence of the desired decomposition then follows.

Let $0\neq X\in \mathcal{C}_{\bf S}$ be compact. Let $e_0$ be an external edge of $\trivalent$ with $\on{ev}_{e_0}(X)\neq 0$. If such an external edge does not exist, choose $e_0$ instead to be an internal edge with $\on{ev}_{e_0}(X)\neq 0$. Choose any direct sum decomposition $B_0\oplus A_0\oplus B_0'\simeq \on{ev}_{e_0}(X)$ with $A_0\simeq k[t_1^\pm]$. Let $v_1$ be a trivalent vertex incident to $e_0$. 

With this starting data, we iteratively make choices of edges $e_i$ with incident vertices $v_i$ and $v_{i+1}$ and find splittings $B_i\oplus A_i\oplus B_i'\simeq \on{ev}_{e_i}(X)$ and $T_i\oplus S_i\oplus T_i'\simeq \on{res}_{v_i}(X)$ as follows. 

Suppose the data for $i$ has been chosen. Let $v_{i+1}\neq v_i$ be the other vertex incident to $e_i$. The summand $S_i\coloneqq Y$ appearing in both the clockwise and counterclockwise splittings of $\on{res}_{v_{i+1}}(X)$ is of the form $Y\simeq M_{\delta}$, with $\delta$ a segment starting at the edge $e_i$ and ending at another edge, called $e_{i+1}$. If $e_{i+1}$ follows $e_i$ in the clockwise direction, we use the clockwise splitting from \Cref{constr:geom1} arising from $B_i\oplus A_i\oplus B_i'$ to obtain a splitting $T_{i+1}\oplus S_{i+1}\oplus T_{i+1}'$ of $\on{res}_{v_{i+1}}(X)$. If $e_{i+1}$ instead follows $e_i$ in the counterclockwise direction, we instead use the counterclockwise splitting. By \Cref{constr:geom1}, we have an equivalence $\phi\colon \on{ev}_{e_i}(T_{i+1})\oplus \on{ev}_{e_i}(S_{i+1})\oplus \on{ev}_{e_i}(T_{i+1}')\simeq \on{ev}_{e_i}(X)\simeq B_i\oplus A_i \oplus B_i'$ which is given by a lower triangular matrix with invertible diagonal. By \Cref{lem:geom1}, we can find compatible autoequivalences of $T_j\oplus S_j\oplus T_j'$  for all $1\leq j\leq i-1$ and of $B_j\oplus A_j\oplus B_j'$ for all $0\leq j\leq i-1$, such that the composite at $B_{i} \oplus A_{i}\oplus B_{i}'$ with $\phi$ is a diagonal matrix. We redefine all $S_j,T_j,T_j',A_j,B_j,B_j'$ by their images under these autoequivalences. We then set $A_{i+1}=\on{ev}_{e_{i+1}}(S_{i+1})$ and $B_{i+1}=\on{ev}_{e_{i+1}}(T_{i+1}),\,B_{i+1}'=\on{ev}_{e_{i+1}}(T_{i+1}')$.

We proceed with making these choices, until we come to a stop in one the following cases below. If $e_0$ is external, we always come to a stop in the following case, as otherwise one can find a contradiction to $X$ being a coCartesian section.

{\bf Case 1)} We have started at an external edge $e_0$ and arrive at an external edge $e_N$ with $N\geq 1$. 

In this case, the $S_i$'s, with $1\leq i\leq N$, glue to a global section $S\subset X$, satisfying $\on{res}_{v}(S)=\bigoplus_{v_j=v}S_j$. Similarly, there is a global section $T\subset X$ satisfying $\on{res}_v(T)=\bigcap_{v_j=v} T_j\oplus T_j'$ if there is at least one $1\leq j\leq N$ with $v_j=v$ and $\on{res}_v(T)=\on{res}_v(X)$ if there are no $j$ with $v_j=v$. The arising map $S\oplus T\rightarrow X$ restricts pointwise to an equivalence, and is hence also an equivalence of global sections. At each vertex $v_i$, we have by construction $S_i\simeq M_{\delta_i}$ for a segment $\delta_i$ at $v_i$. We compose these segments to an open pure matching curve $\gamma$. We find that $S\simeq M_{\gamma}$ is geometric. This gives the desired splitting $X\simeq M_{\gamma}\oplus T$. 

Suppose now that we did not stop in case 1). We may then assume that $\on{ev}_e(X)\simeq 0$ for all external edges $e$.

{\bf Case 2)} There is a non-empty set $J\subset \{0,\dots,N-1\}$, such that $e_N=e_{j}$ for all $j\in J$ and $A_N\subset \langle A_{j}\rangle_{j\in J}$ lies in the submodule of $\on{ev}_{e_N}(X)$ generated by the $A_{j}$'s. Note that in this case $J$ contains the element $0$, as otherwise we again get a contradiction to $X$ being coCartesian. We obtain a direct summand $S\subset X$ with complement $T$, satisfying $\on{res}_{v}(S)=\bigoplus_{v_j=v,\,j<N}S_j$ and $\on{res}_v(T)=\bigcap_{v_j=v,\,j<N}T_j\oplus T_j'$ or $\on{res}_v(T)=\on{res}_v(X)$ if there are no $j<N$, such that $v_j=v$. 

Again, at each vertex $v_i$ we have by construction $S_i\simeq M_{\delta_i}$ for a segment $\delta_i$ at $v_i$. We compose these segments to a closed curve $\gamma$. If $\gamma$ is given by the composite of $a$ identical closed curves $\gamma'$, we have that the number $N$ of segments of $\gamma$ is a multiple of $a$ and $e_i=e_{i+jN/a}$ for $1\leq j\leq a-1$ and $0\leq i\leq N/a$. We relabel $A_i=\bigoplus_{0\leq j\leq a-1} A_{i+jN/a}$ and $S_i=\bigoplus_{0\leq j\leq a-1} S_{i+jN/a}$ for $0\leq i <N/a$. We have chosen $\gamma'$ so that it is not again given by the composite of multiple identical closed curves. To extend $\gamma'$ to a matching datum, we need to specify a rank and a monodromy datum. 

We choose an ordered basis $U_0$ of $A_0=\on{ev}_{e_0}(S_1)\in \mathcal{D}(k[t_1^\pm])$ with cardinality $a$. Consider the basis $U_1'$ of $\on{res}_{v_1}(S_1)$ which is mapped under $S_1(v_1\rightarrow e_0)$ to the chosen basis $U_0$. The image of $U_1'$ under $S_1(v_1\rightarrow e_1)$ defines an ordered basis $U_1$ of $A_1$. From this, we again find a basis $U_2'$ of $\on{res}_{v_2}(S_2)$. Proceeding in this way, performing these steps $N/a$ times, we obtain a second ordered basis $U_{N/a}$ of $A_{N/a}=A_0$. Consider the Jordan normal form of the linear map, which maps $U_{N/a}$ to $U_0$. For each Jordan $b\times b$-block with eigenvalue $\lambda\neq 0$, we can equip $\gamma$ with rank $b$ and monodromy datum $\lambda$. With these choices of monodromy data, we finally have $S\simeq \bigoplus_{\text{blocks}}M_{\gamma}$, where the sum runs over the Jordan blocks. We thus have the desired splitting $X\simeq \bigoplus_{\text{blocks}}M_{\gamma}\oplus T$, concluding the proof.
\end{proof}

\begin{proof}[Proof of \Cref{thm:geom}, part 2: uniqueness of decomposition.]
Part 1 of the proof of \Cref{thm:geom} and the proof of \Cref{thm:KS} show that $\mathcal{C}_{\bf S}$ is Krull-Schmidt, or equivalently the homotopy $1$-category $\on{ho}\mathcal{C}_{\bf S}$ is Krull-Schmidt. The essential uniqueness of a decomposition into indecomposables in a Krull-Schmidt category is shown in \cite[Theorem 4.2]{Kra15}. It thus suffices to show that for two matching curves $\gamma,\gamma'$, we have $M_{\gamma}\simeq M_{\gamma'}$ if and only if $\gamma\sim\gamma'$. It is now easy to see that $M_{\gamma}\simeq M_{\gamma'}$ implies that $\gamma$ and $\gamma'$ are either both open or both closed. In the former case, we can argue by using \Cref{thm:hom1}, by testing $M_{\gamma},M_{\gamma'}$ against objects arising from different matching curves. One finds that $\gamma\sim \gamma'$.

If $\gamma$ and $\gamma'$ are closed, we distinguish the cases that the matching curves underlying $\gamma$ and $\gamma'$ are identical (up to reversal of orientation) or not. In the first case, one can show that $\on{Mor}_{\mathcal{C}_{\bf S}}(M_{\gamma},M_{\gamma'})\simeq 0$ if the monodromy data of $\gamma,\gamma'$ are distinct; it follows that these objects are also distinct. In the second case, a similar argument as in the case of open matching curves applies to show that the open curves composed of segments $\eta$ and $\eta'$, obtained from cutting $\gamma$ and $\gamma'$ open at an edge of $\mathcal{T}$, are identical. This shows that $\gamma$ and $\gamma'$ are also identical, concluding the case distinction and the proof.
\end{proof}

\section{The categorification of cluster algebras with coefficients}\label{sec:categorification} 

For the entire section, we fix a marked surface ${\bf S}$. We will further assume that the base field $k$ is algebraically closed and satisfies $\on{char}(k)\neq 2$.

We begin in the \Cref{subsec:clusteralgebra,subsec:Skeinalgebra} by recalling the definitions of the cluster algebra and of the commutative Skein algebra associated with ${\bf S}$. In \Cref{subsec:clustertilting}, we show that the cluster tilting objects in the exact $\infty$-/extriangulated generalized cluster category $\mathcal{C}_{\bf S}$ are in bijection with the clusters of this cluster algebra. We also describe the endomorphism algebras of the relative cluster tilting objects in $\mathcal{C}_{\bf S}$ in terms of gentle algebras and discuss two simple examples. In \Cref{subsec:clusterchar}, we describe a cluster character on $\mathcal{C}_{\bf S}$ with values in the commutative Skein algebra.

\subsection{Cluster algebras of marked surfaces}\label{subsec:clusteralgebra}

In the following we recall the definition of the cluster algebra associated to ${\bf S}$ with coefficients in the boundary arcs. For the conventions on cluster algebras, we follow \cite[Chapter 3]{FWZ16}. 

\begin{definition}
Let $\mathscr{F}=\mathbb{Q}(y_1,\dots,y_{m_1+m_2})$ be the field of rational functions in $m_1+m_2$ variables. A (labeled) seed $\big({\bf x},\tilde{M}\big)$ in $\mathscr{F}$ consists of 
\begin{itemize}
\item an $m_1+m_2$-tuple ${\bf x}=(x_1,\dots,x_{m_1+m_2})$ in $\mathscr{F}$ forming a free generating set of $\mathscr{F}$ and 
\item an $(m_1+m_2)\times m_1$-matrix $\tilde{M}$, such that the upper $m_1\times m_1$-matrix is skew-symmetric\footnote{More generally, one can also consider skew-symmetrizable matrices, see \cite[Definition 3.1.1]{FWZ16}.}. 
\end{itemize}
The tuple ${\bf x}$ is called a cluster and the elements $x_1,\dots,x_{m_1+m_2}$ are called cluster variables. The elements $x_{m_1+1},\dots,x_{m_1+m_2}$ are also called the frozen cluster variables. The matrix $\tilde{M}$ is called the extended mutation matrix.
\end{definition}

\begin{definition}
Let $\big({\bf x},\tilde{M}\big)$ be a seed and $l\in \{1,\dots,m_1\}$. The seed mutation at $l$ is given by the seed $\big({\bf x'},\mu_l(\tilde{M})\big)$ with 
\begin{itemize}
\item $\displaystyle \mu_l(\tilde{M})_{i,j}=\begin{cases} -\tilde{M}_{i,j} & \text{if }i=l\text{ or }j=l\\ 
\tilde{M}_{i,j}+\tilde{M}_{i,l}\tilde{M}_{l,j} & \text{if }\tilde{M}_{i,l}>0\text{ and }\tilde{M}_{l,j}>0\\
\tilde{M}_{i,j}-\tilde{M}_{i,l}\tilde{M}_{l,j} & \text{if }\tilde{M}_{i,l}<0\text{ and }\tilde{M}_{l,j}<0\\
\tilde{M}_{i,j}& \text{else.}
\end{cases}$
\item ${\bf x'}=(x_1,\dots,x_{l-1},x_l',x_{l+1},\dots,x_{m_1+m_2})$, where $x_l'$ is determined by the cluster exchange relation
\[ x_l'x_l=\prod_{j\text{~with~}\tilde{M}_{j,l}> 0}x_{j}^{\tilde{M}_{j,l}}+\prod_{j\text{~with~}\tilde{M}_{j,l}< 0}x_{j}^{-\tilde{M}_{j,l}}\,.\]
\end{itemize} 
\end{definition}

\begin{definition}
Let $({\bf x},\tilde{M})$ be a seed.
\begin{itemize}
\item  The associated cluster algebra $\on{CA}\subset \mathscr{F}$ is the $\mathbb{Q}$-subalgebra of $\mathscr{F}$ generated by all cluster variables in all seeds obtained from $({\bf x},\tilde{M})$ via iterated seed mutation. 
\item The associated upper cluster algebra $\on{UCA}\subset \mathscr{F}$ is the $\mathbb{Q}$-subalgebra consisting of those elements which, for every cluster of $\on{CA}$, are Laurent polynomials in the cluster variables of that cluster.
\end{itemize}
\end{definition}

\begin{remark}
The  cluster algebra or upper cluster algebra associated to a seed only depends on the extended mutation matrix, up to isomorphism of $\mathbb{Q}$-algebras. We can thus speak of the cluster algebra or upper cluster algebra associated to an extended mutation matrix.
\end{remark}

We proceed in \Cref{def:matrix} with associating an extended mutation matrix to a choice of ideal triangulation $\I$ of ${\bf S}$ in the sense of \Cref{def:idealtr}. Choosing a different ideal triangulation changes the extended mutation matrix by matrix mutations.

\begin{definition}[{$\!\!$\cite[Definition 4.1]{FST08},\cite{FT12}}]\label{def:matrix}
Let $I$ be an ideal triangulation of ${\bf S}$ with $m_1$ interior arcs, labeled arbitrarily as $1,\dots,m_1$, and $m_2$ boundary arcs, labeled arbitrarily as $m_1+1,\dots,m_1+m_2$. 

The extended signed adjacency matrix $\tilde{M}_{\I}=\sum_{\Delta}M_\Delta$ of $\I$ is the $(m_1+m_2)\times m_1$-matrix given by the sum over all ideal triangles $\Delta$ of $\I$ of the $(m_1+m_2)\times m_1$-matrices defined by 
\[ (M_{\Delta})_{i,j}=\begin{cases} 1 & \text{if }\Delta\text{ has sides }i\text{ and }j\text{ with }i\text{ following }j\text{ in the counterclockwise direction},\\
-1 &  \text{if }\Delta\text{ has sides }i\text{ and }j\text{ with }i\text{ following }j\text{ in the clockwise direction},\\
0 & \text{else.}\end{cases}
\]
The upper $m_1\times m_1$-matrix of $\tilde{M}_{\I}$ is skew-symmetric and called the signed adjacency matrix. 
\end{definition}

\begin{definition}\label{cadef}
\begin{itemize}
\item Let ${\bf S}$ be an oriented marked surface with an ideal triangulation $\I$. We define $\on{CA}_{\bf S}$ to be the cluster algebra associated to the extended mutation matrix given by the extended signed adjacency matrix $\tilde{M}_{\I}$ of $\I$. Similarly, $\on{UCA}_{\bf S}$ is defined as the associated upper cluster algebra.
\item We define the algebra $\on{CA}_{\bf S}^{\on{loc}}$ as the localization of $\on{CA}_{\bf S}$ at the frozen cluster variables $x_{m_1+1},\dots, x_{m_1+m_2}$, where $m_1$ is the number of interior arc in $\I$ and $m_2$ is the number of boundary arcs in $\I$.
\end{itemize}
\end{definition}

The cluster algebra $\on{CA}_{\bf S}$ admits a beautiful description in terms of coordinates on the decorated Teichm\"uller space of ${\bf S}$, called lambda lengths, see \cite{FT12}.

\begin{definition}\label{clrel}
Let $\gamma,\gamma'$ be two arcs in ${\bf S}$. Suppose that $\gamma,\gamma'$ have a crossing as in \Cref{crosssmooth}. Then the Ptolemy relation is defined as $\gamma\cdot \gamma'=\gamma_1\cdot \gamma_2 +\gamma_3\cdot \gamma_4$.
\end{definition}

\begin{theorem}\label{thm:clusters}
The cluster variables of $\on{CA}_{\bf S}$ are canonically in bijection with the arcs in ${\bf S}$. A set of cluster variables of $\on{CA}_{\bf S}$ forms a cluster if and only if the corresponding arcs form an ideal triangulation of ${\bf S}$. The cluster exchange relations are the Ptolemy relations.
\end{theorem}

\begin{proof}
This is \cite[Theorem 8.6]{FT12} specialized to the case that ${\bf S}$ has no punctures.
\end{proof}

\subsection{Commutative Skein algebras}\label{subsec:Skeinalgebra}

We proceed with defining the commutative $q=1$ Skein algebra $\on{Sk}^1({\bf S})$ of links in ${\bf S}$. As shown in \cite{Mul16}, this algebra embeds into the upper cluster algebra $\on{UCA}_{\bf S}$. 

\begin{definition}
A link is a homotopy class relative $\partial {\bf S}\backslash M$ of continuous maps $\gamma\colon U\rightarrow {\bf S}\backslash M$, with $U$ a finite disjoint union of $[0,1]$'s and $S^1$'s, satisfying properties 1.~and 2.~from \Cref{def:curve}. We consider links up to reversal of orientation. We refer to the curves with domain $U=[0,1],S^1$ constituting a link as its components. There is an empty link with $U=\emptyset$. We denote by $\mathscr{L}({\bf S})$ the set of all links in ${\bf S}$.
\end{definition}

The $\mathbb{Q}$-vector space $\mathbb{Q}^{\mathcal{L}({\bf S})}$ inherits the structure of a commutative $\mathbb{Q}$-algebra, by defining the product of two links to be the union of the two links and extending this product $\mathbb{Q}$-bilinearly. 

\begin{definition}\label{def:skein}
We define the $q=1$ Skein algebra $\on{Sk}^1({\bf S})$ as the quotient of the $\mathbb{Q}$-algebra $\mathbb{Q}^{\mathcal{L}({\bf S})}$ by the ideal generated by the following elements. 
\begin{enumerate}[1)]
\item The $q=1$ Kauffman Skein relation.
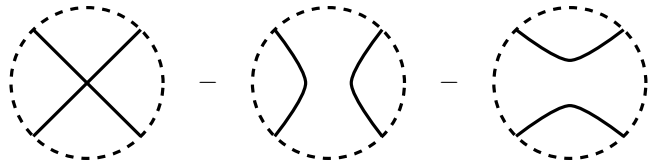
\begin{equation*}
\begin{tikzpicture}[baseline={([yshift=-.5ex]current bounding box.center)}]
\draw[very thick, dashed] (0,0) circle (1);
\draw[very thick] (0.7071,0.7071)--(-0.7071,-0.7071)
(0.7071,-0.7071)--(-0.7071,0.7071);
\end{tikzpicture}
 \quad - \quad
\begin{tikzpicture}[baseline={([yshift=-.5ex]current bounding box.center)}]
\draw[very thick,dashed] (0,0) circle (1);
\draw[very thick, very thick] plot [smooth] coordinates {(0.7071,0.7071) (0.3,0) (0.7071,-0.7071)};
\draw[very thick] plot [smooth] coordinates {(-0.7071,-0.7071) (-0.3,0) (-0.7071,0.7071)};
\end{tikzpicture}\quad -\quad
\begin{tikzpicture}[baseline={([yshift=-.5ex]current bounding box.center)}]
\draw[very thick, dashed] (0,0) circle (1);
\draw[very thick] plot [smooth] coordinates {(-0.7071,0.7071) (0,0.3) (0.7071,0.7071)};
\draw[very thick] plot [smooth] coordinates {(0.7071,-0.7071) (0,-0.3) (-0.7071,-0.7071)};
\end{tikzpicture}
\end{equation*}
\item The value of the unknot.
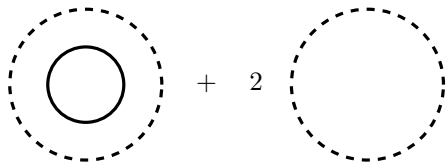
\[
\begin{tikzpicture}[baseline={([yshift=-.5ex]current bounding box.center)}]
\draw[very thick, dashed] (0,0) circle (1);
\draw[very thick] (0,0) circle (0.5);
\end{tikzpicture} \quad + \quad  2 \quad \begin{tikzpicture}[baseline={([yshift=-.5ex]current bounding box.center)}]
\draw[very thick, dashed] (0,0) circle (1);
\end{tikzpicture}\]
\item Any component with domain $[0,1]$ which is homotopic relative endpoints to a subset of $\partial {\bf S}\backslash M$.
\end{enumerate}
The relations in 1) and 2) are understood to describe local relations inside a small disc (indicated by the dotted circle) in ${\bf S}$, applicable to any subset of components of a link. The depicted curves are identical outside the small disc. 
\end{definition}

\begin{remark}
The $q=1$ Kauffman Skein relation recover the cluster exchange relations given by the Ptolemy relation, see \Cref{clrel}.
\end{remark}

We remark that \Cref{def:skein} is equivalent to the definition of the $q$-Skein algebra given by Muller in \cite{Mul16}, with $q$ set to $1$, as can be seen by using Remarks 3.4 and 3.6 in \cite{Mul16}.

\begin{definition}
We define the localized $q=1$ Skein algebra $\on{Sk}^{1,\on{loc}}({\bf S})$ as the localization of $\on{Sk}^{1}({\bf S})$ at the set of boundary arcs. 
\end{definition}

\begin{theorem}[$\!\!$\cite{Mul16}]\label{skclthm}
There exist injective morphisms of $\mathbb{Q}$-algebras 
\begin{equation}\label{eq:clskinc} \on{CA}_{\bf S}^{\on{loc}}\hookrightarrow \on{Sk}^{1,\on{loc}}({\bf S})\hookrightarrow \on{UCA}_{\bf S}\,.\end{equation}
If ${\bf S}$ additionally has at least two marked points, then the  maps in \eqref{eq:clskinc} are equivalences of $\mathbb{Q}$-algebras. 
\end{theorem}

\subsection{Classification of cluster tilting objects}\label{subsec:clustertilting}

We choose a \idtr{} $\mathcal{T}$ of ${\bf S}$. Consider the perverse schober $\mathcal{F}_{{\mathcal{T}}}^{\on{clst}}$ from \Cref{def:gclst} with its $k[t_2^\pm]$-linear smooth and proper $\infty$-category of global sections given by the generalized cluster category $\mathcal{C}_{\bf S}$. We let $G$ be the right adjoint of the adjunction 
\[ F\coloneqq \partial \mathcal{F}^{\on{clst}}_{{\mathcal{T}}}\colon \prod_{e\in {\mathcal{T}}_1^\partial}{\mathcal{F}}^{\on{clst}}_{{\mathcal{T}}}(e)\longleftrightarrow \mathcal{C}_{\bf S}\noloc G\coloneqq \prod_{e\in {\mathcal{T}}_1^\partial}\on{ev}_e \]
from \Cref{def:cupfunctor}. Since $\on{char}(k)\neq 2$, the functor $G$ admits a $2$-Calabi--Yau structure, see \Cref{thm:cyclcat} and the adjunction $F\dashv G$ is spherical by \Cref{cor:sphbdry}. By \Cref{subsec:exactfromCY}, we obtain a Frobenius, $2$-Calabi--Yau extrianguled category $(\on{ho}\mathcal{C}_{\bf S}^{\on{c}},\on{Ext}^{1,\on{CY}}_{\mathcal{C}_{\bf S}},\mathfrak{s})$ arising from of a Frobenius exact $\infty$-structure on $\mathcal{C}_{\bf S}^{\on{c}}$.

\begin{theorem}\label{rigidthm}\label{thm:rigid}
Let $k$ be an algebraically closed field with $\on{char}(k)\neq 2$.
\begin{enumerate}[i)]
\item Let $\gamma$ be a matching curve in ${\bf S}$. Then $M_{\gamma}$ is rigid in $(\on{ho}\mathcal{C}_{\bf S}^{\on{c}},\on{Ext}^{1,\on{CY}}_{\mathcal{C}_{\bf S}},\mathfrak{s})$ if and only if  $\gamma$ is an arc.  
\item Consider a collection $\I$ of distinct arcs in ${\bf S}$. Then $\bigoplus_{\gamma\in \I} M_\gamma$ is a cluster tilting object in $(\on{ho}\mathcal{C}_{\bf S}^{\on{c}},\on{Ext}^{1,\on{CY}}_{\mathcal{C}_{\bf S}},\mathfrak{s})$ if and only if $\I$ is an ideal triangulation of ${\bf S}$.
\end{enumerate}
\end{theorem}

We prove \Cref{thm:rigid} further below.

\begin{corollary}\label{cor:bmrcluster}
There are canonical bijections between the sets of the following objects.
\begin{itemize}
\item Clusters of the cluster algebra with coefficients $\on{CA}_{\bf S}$ of ${\bf S}$.
\item Ideal triangulations of ${\bf S}$.
\item Cluster tilting objects in $(\on{ho}\mathcal{C}_{\bf S}^{\on{c}},\on{Ext}^{1,\on{CY}}_{\mathcal{C}_{\bf S}},\mathfrak{s})$ up to equivalence.
\end{itemize}
\end{corollary}

\begin{proof}
Combining the geometrization \Cref{thm:geom} and \Cref{thm:rigid}, we obtain that there is a bijection between the sets of equivalence classes of cluster tilting objects and ideal triangulations of ${\bf S}$. The bijection between clusters and ideal triangulation is \Cref{thm:clusters}.
\end{proof}

\begin{lemma}\label{geomcyprop}
\begin{enumerate}[(1)]
\item Let $\gamma,\gamma'$ be two distinct matching curves in ${\bf S}$. Then 
\[ \on{Mor}^{\on{CY}}_{\mathcal{C}_{\bf S}}(M_{\gamma},M_{\gamma'})\subset \on{Mor}_{\mathcal{C}_{\bf S}}(M_{\gamma},M_{\gamma'})\] 
consists of those direct summands identified in \Cref{thm:hom1} which correspond to crossings of $\gamma$ and $\gamma'$. 
\item Let $\gamma$ be an open matching curve in ${\bf S}$. Then 
\[ \on{Mor}^{\on{CY}}_{\mathcal{C}_{\bf S}}(M_{\gamma},M_{\gamma})\subset \on{Mor}_{\mathcal{C}_{\bf S}}(M_{\gamma},M_{\gamma})\] 
consists of those direct summands identified in \Cref{thm:hom2} which corresponding to self-crossings.
\item Let $\gamma$ be a closed matching curve in ${\bf S}$. Then 
\[  \on{Mor}^{\on{CY}}_{\mathcal{C}_{\bf S}}(M_{\gamma},M_{\gamma})=\on{Mor}_{\mathcal{C}_{\bf S}}(M_{\gamma},M_{\gamma})\,.\] 
\end{enumerate}
\end{lemma}

\begin{proof}
Inspecting the construction of the direct summands of $\on{Mor}_{{\mathcal{C}}_{\bf S}}(M_{\gamma},M_{\gamma'})$ corresponding to the different types of intersections in \cite{Chr21b}, one finds the following.
\begin{itemize}
\item Boundary intersections give rise to morphisms which evaluate non-trivially at an external edge of $\mathcal{T}$, i.e.~do not get mapped by $G$ to zero. 
\item Crossings give rise to morphisms which restrict to zero on all external edges of $\mathcal{T}$ and thus define direct summands lying in $\on{Mor}^{\on{CY}}_{{\mathcal{C}}_{\bf S}}(M_{\gamma},M_{\gamma'})$. 
\item If $\gamma$ is open, then the direct summand $k[t_1^\pm]\subset \on{Mor}_{{\mathcal{C}}_{\bf S}}(M_{\gamma},M_{\gamma})$ does not evaluate to zero under $G$.
\item If $\gamma$ is closed, then $G(M_{\gamma})\simeq 0$ and hence $\on{Mor}^{\on{CY}}_{\mathcal{C}_{\bf S}}(M_{\gamma},M_{\gamma})=\on{Mor}_{\mathcal{C}_{\bf S}}(M_{\gamma},M_{\gamma})$.
\end{itemize}
This shows the Lemma.
\end{proof}

\begin{proof}[Proof of \Cref{thm:rigid}.]
We begin with part i). If $\gamma$ is a closed matching curve, then $M_\gamma$ is not rigid because there is always a direct summand $k[t_1^\pm]\subset \on{Mor}^{\on{CY}}_{\mathcal{C}_{\bf S}}(M_\gamma,M_\gamma)= \on{Mor}_{\mathcal{C}_{\bf S}}(M_\gamma,M_\gamma)$. We thus assume that $\gamma$ is open. \Cref{thm:hom1} and \Cref{geomcyprop} imply that $M_{\gamma}$ is rigid if and only if $\gamma$ has no self crossings, meaning that $\gamma$ is an arc, showing part i).

Using \Cref{thm:hom1} and \Cref{geomcyprop}, we find that given two arcs $\gamma,\gamma'$, we have an isomorphism $\on{Ext}^{1,\on{CY}}_{\mathcal{C}_{\bf S}}(M_{\gamma},M_{\gamma'})\simeq 0$ if and only if $\gamma$ and $\gamma'$ are compatible in the sense of \Cref{def:idealtr}. By part i) and the maximality of an ideal triangulation, we find that any basic, maximal rigid object is of the form $\bigoplus_{\gamma \in \I} M_{\gamma}$ for an ideal triangulation $\I$ and conversely any ideal triangulation $I$ gives rise to such a basic, maximal rigid object. Since any cluster tilting object is basic and maximal rigid, it remains to verify that
\begin{equation}\label{eq:cyn0}
\on{Ext}^{1,\on{CY}}_{\mathcal{C}_{\bf S}}(\bigoplus_{\gamma \in \I}M_{\gamma},M_{\gamma'})\not\simeq 0
\end{equation} 
for $\I$ an ideal triangulation and $\gamma'\notin \I$ a matching curve. 

Let thus $\I$ be an ideal triangulation. Then $\I$ decomposes ${\bf S}$ into  triangles with edges the arcs in $\I$. If a matching curve $\gamma'$ crosses an edge of one of these triangles, we find a direct summand $k[t_{1}^\pm]\subset \on{Mor}_{\mathcal{C}_{\bf S}}^{\on{CY}}(M_\gamma,M_{\gamma'})$, showing \eqref{eq:cyn0}. If $\gamma'$ does not cross any arcs in $\I$, then $\gamma'$ is contained in an ideal triangle and hence already in $\I$. We thus obtain that \eqref{eq:cyn0} holds for all $\gamma' \notin \I$, showing that $\bigoplus_{\gamma \in \I} M_{\gamma}$ is cluster tilting.
\end{proof}

Recall that given a \idtr{} $\mathcal{T}$ of ${\bf S}$, we denote the associated relative Ginzburg algebra by $\mathscr{G}_\mathcal{T}$.

\begin{proposition}\label{prop:gtl}
Consider an ideal triangulation $\I$ of ${\bf S}$ with dual \idtr{} $\mathcal{T}$ and let $X=\bigoplus_{\gamma\in \I}M_{\gamma}$ be the associated cluster tilting object.
\begin{enumerate}[(1)]
\item The Jacobian algebra $\mathscr{J}_{\mathcal{T}}\coloneqq \on{H}_0(\mathscr{G}_\mathcal{T})$ is a gentle algebra.
\item There exists an isomorphism of dg-algebras (with vanishing differentials)
\[ \on{H}_*\on{Mor}_{\mathcal{C}_{\bf S}}(X,X)\simeq \on{H}_0(\mathscr{G}_\mathcal{T})\otimes_k k[t_1^\pm]\,.\]
\end{enumerate}
In particular, the discrete endomorphism algebra $\on{Ext}^0_{\mathcal{C}_{\bf S}}(X,X)$ of $X$ is a gentle algebra.
\end{proposition}

\begin{proof}
Part (1) is observed in \cite[Section 7]{Chr21b}. Part (2) follows from \Cref{thm:hom1,thm:hom2} as in the proof of \cite[Prop.~7.3]{Chr21b}.
\end{proof}

\begin{remark}
The marked surface associated to the Jacobian gentle algebra $\on{H}_0(\mathscr{G}_\mathcal{T})$ (in the sense of \cite[Def 1.10]{OPS18}) is given by the complement in ${\bf S}$ of the set of vertices of the trivalent spanning graph $\mathcal{T}$. This gentle algebra is finite dimensional but not smooth.
\end{remark}

\begin{example}\label{ex:4gon}
Consider the $4$-gon ${\bf S}$, depicted as follows, with an ideal triangulation (in blue).
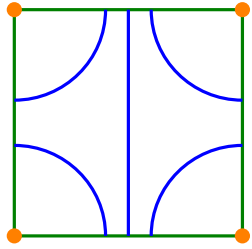
\begin{center}
\begin{tikzpicture}
\draw[color=ao, very thick] (3,0)--(3,3)--(0,3)--(0,0)--(3,0);
\draw[color=blue][very thick] (0,1.2)arc[start angle=90, end angle=0,radius=1.2];
\draw[color=blue][very thick] (0,1.8)arc[start angle=-90, end angle=0,radius=1.2];
\draw[color=blue][very thick] (1.8,0)arc[start angle=180, end angle=90,radius=1.2];
\draw[color=blue][very thick] (1.8,3)arc[start angle=180, end angle=270,radius=1.2];
\draw[color=blue][very thick] (1.5,0)--(1.5,3);
\fill[color=orange] (0,0) circle (0.1);
\fill[color=orange] (0,3) circle (0.1);
\fill[color=orange] (3,0) circle (0.1);
\fill[color=orange] (3,3) circle (0.1);
\end{tikzpicture}
\end{center}
The generalized cluster category $\mathcal{C}_{\bf S}$ is equivalent to the $1$-periodic derived $\infty$-category of the $A_3$-quiver, which can be defined as $\mathcal{D}(kA_{3}\otimes_k k[t_1^\pm])$. The arising Jacobian gentle algebra is given by quotient of the path algebra of the quiver 
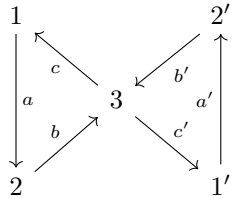
\[
\begin{tikzcd}
1 \arrow[dd, "a"] &                                    & 2' \arrow[ld, "b'"] \\
                  & 3 \arrow[lu, "c"] \arrow[rd, "c'"] &                     \\
2 \arrow[ru, "b"] &                                    & 1' \arrow[uu, "a'"]
\end{tikzcd}
\]
by the ideal $(ba,cb,ac,b'a',c'b',a'c')$.
\end{example}

\begin{example}
Consider the annulus ${\bf S}$ with a marked point on each boundary component and an ideal triangulation depicted as follows.
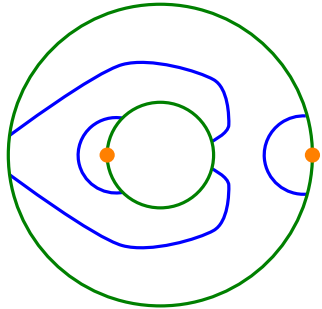
\begin{center}
\begin{tikzpicture}
 \draw[color=blue][very thick] plot [smooth] coordinates {(-1.982,0.261) (-0.5,1.2) (0.7,1) (0.9,0.4) (0.676,0.181)};
 \draw[color=blue][very thick] plot [smooth] coordinates {(-1.982,-0.261) (-0.5,-1.2) (0.7,-1) (0.9,-0.4) (0.676,-0.181)};
\draw[color=blue][very thick] (1.931,-0.518) arc	[start angle=275,	end angle=87, x radius=0.52cm, y radius =0.52cm];
\draw[color=blue][very thick] (-0.495,-0.495) arc	[start angle=280,	end angle=80, x radius=0.5cm, y radius =0.5cm];

\draw[color=ao, very thick] (0,0) circle (2);
\draw[color=ao, very thick] (0,0) circle (0.7);
\fill[color=orange] (2,0) circle (0.1);
\fill[color=orange] (-0.7,0) circle (0.1);
\end{tikzpicture}
\end{center}
The Jacobian gentle algebra is given by the quotient of the path algebra of the quiver 
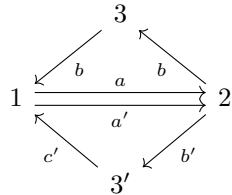
\[
\begin{tikzcd}
                                                            & 3 \arrow[ld, "b"]   &                                    \\
1 \arrow[rr, "a", shift left] \arrow[rr, "a'"', shift right] &                     & 2 \arrow[lu, "b"] \arrow[ld, "b'"] \\
                                                            & 3' \arrow[lu, "c'"] &                                   
\end{tikzcd}
\]
by the ideal $(ba,cb,ac,b'a',c'b',a'c')$. This is a relative version of the Kronecker quiver. While the mapping class group of the $4$-gon in \Cref{ex:4gon} is trivial, the mapping class group of the annulus ${\bf S}$ is given by $\mathbb{Z}$. The generator $1\in \mathbb{Z}$ corresponds to a diffeomorphism rotating one boundary circle by one full rotation and fixing the other boundary circle. Under the action of an element $\alpha$ of the mapping class group, an object $M_\gamma\in \mathcal{C}_{\bf S}$ is mapped to $M_{\alpha\circ \gamma}$.
\end{example}

\subsection{A cluster character on \texorpdfstring{$\mathcal{C}_{\bf S}$}{the generalized cluster category}}\label{subsec:clusterchar}

We begin with the notion of a cluster character on an extriangulated category, generalizing the notion of a cluster character of \cite{Pal08}.

\begin{definition}\label{clchardef}
Let $(C,\mathbb{E},\mathfrak{s})$ be a $k$-linear extriangulated category. We denote by $\on{obj}(C)$ the set of equivalence classes of objects in $C$. A cluster character $\chi$ on $C$ with values in a commutative ring $R$ is a map 
\[ \chi\colon \on{obj}(C)\rightarrow R\] 
such that for all $X,Y\in \mathcal{C}$ the following holds.
\begin{itemize}
\item $\chi(X)=\chi(Y)$ if $X\simeq Y$.
\item $\chi(X\oplus Y)=\chi(X)\chi(Y)\,$.
\item $\chi(X)\chi(Y)=\chi(B)+\chi(B')$ if $\on{dim}_k\mathbb{E}(X,Y)=\on{dim}_k\mathbb{E}(Y,X)=1$ and $B,B'\in C$ are the middle terms of the corresponding non-split extensions of $X$ by $Y$ and $Y$ by $X$. 
\end{itemize}
The last property is called the cluster multiplication formula.
\end{definition}

Given a matching curve $\gamma$ in ${\bf S}$, we consider the underlying curve as a link in ${\bf S}$ with a single component, denoted $l(\gamma)$. 

\begin{theorem}\label{thm:char}
Let $k$ be an algebraically closed field with $\on{char}(k)\neq 2$. Consider the map 
\[\chi\colon \on{obj}(\on{ho}\mathcal{C}_{\bf S}^{\on{c}})\rightarrow \on{Sk}^{1}_{\bf S}\] determined by 
\begin{itemize}
\item $\chi(A)=\chi(B)$ if $A\simeq B$
\item $\chi(A\oplus B)=\chi(A)\chi(B)$
\item $\chi(M_{\gamma})=l(\gamma)$ for any matching curve $\gamma$ in ${\bf S}$.
\end{itemize}
Then $\chi$ is a cluster character on $(\on{ho}\mathcal{C}_{\bf S}^{\on{c}},\on{Ext}^{1,\on{CY}}_{\mathcal{C}_{\bf S}},\mathfrak{s})$.
\end{theorem}

\begin{remark}
Composing $\chi$ with $\on{Sk}^{1}_{\bf S}\hookrightarrow \on{Sk}^{1,\on{loc}}_{\bf S}\hookrightarrow \on{UCA}_{\bf S}$ defines a cluster character to the upper cluster algebra of ${\bf S}$. If $\mathbb{S}$ has at least two marked points, the upper cluster algebra is equivalent to the cluster algebra with coefficients $\on{CA}_{\bf S}^{\on{loc}}$ localized at the boundary arcs, see \Cref{skclthm}.
\end{remark}

\begin{proof}[Proof of \Cref{thm:char}]
The geometrization \Cref{thm:geom} shows that $\chi$ is well-defined. It is also clear that $\chi$ satisfies all parts of \Cref{clchardef}, except for the cluster multiplication formula.
We thus consider two matching curves $\gamma,\gamma'$ in ${\bf S}$, which satisfy $\on{dim}_k\on{Ext}^{1,\on{CY}}_{\mathcal{C}_{\bf S}}(M_{\gamma},M_{\gamma'})=1$ and let $B,B'\in \mathcal{C}_{\bf S}^{\on{c}}$ be the corresponding extensions. Suppose that $\gamma$ and $\gamma'$ are distinct. By \Cref{thm:hom1}, $\on{dim}_k\on{Ext}^{1,\on{CY}}_{\mathcal{C}_{\bf S}}(M_{\gamma},M_{\gamma'})=1$ implies that $\gamma,\gamma'$ have a single crossing. The cluster multiplication formula $\chi(M_{\gamma})\chi(M_{\gamma'})=\chi(B)+\chi(B')$ thus follows from \Cref{prop:cone}, see also \Cref{rem:ptlmycone}.

If $\gamma=\gamma'$, we find that $\gamma$ is closed, and that $\gamma$ and $\gamma'$ have the same monodromy datum. Denote by $a$ and $a'$ the ranks of $\gamma$ and $\gamma'$. We show that $\on{dim}_k\on{Ext}^{1,\on{CY}}_{\mathcal{C}_{\bf S}}(M_{\gamma},M_{\gamma'})$ is always an even number and hence not equal to $1$. Self-crossings of $\gamma$ give rise to two crossings of $\gamma,\gamma'$ and thus contribute the number $2aa'$ to $\on{dim}_k\on{Ext}^{1,\on{CY}}_{\mathcal{C}_{\bf S}}(M_{\gamma},M_{\gamma'})$. The global section $M_{\gamma}$ arises as a coequalizer, see \eqref{eq:ii}. It follows that $\on{Mor}_{\mathcal{C}_{\bf S}}(M_{\gamma},M_{\gamma'})$ is equivalent to the fiber of some morphism
\begin{equation}\label{eq:2kt1} k[t_1^\pm]^{\oplus aa'}\longrightarrow k[t_1^\pm]^{\oplus aa'+2aa'i^{\on{cr}}(\gamma,\gamma')}\,.\end{equation}
Inspecting the construction of the morphism \eqref{eq:2kt1}, one finds that it lies in the image of the functor $\zeta^*\colon \mathcal{D}(k[t_1^\pm])\rightarrow \mathcal{D}(k[t_2^\pm])$. Using that every morphism in $\mathcal{D}(k[t_1^\pm])$ splits into a direct sum of an equivalence and a zero morphism, it follows that $\on{Mor}_{\mathcal{C}_{\bf S}}^{\on{CY}}(M_{\gamma},M_{\gamma'})=\on{Mor}_{\mathcal{C}_{\bf S}}(M_{\gamma},M_{\gamma'})$ is free of even rank in $\mathcal{D}(k[t_1^\pm])$, which implies that $\on{dim}_k\on{Ext}^{1,\on{CY}}_{\mathcal{C}_{\bf S}}(M_{\gamma},M_{\gamma})$ is indeed an even number. This concludes the case distinction and the proof. 
\end{proof}

\begin{remark}
The proof of \Cref{thm:char} shows that there is some flexibility in the formula for the cluster character. In fact, it appears that one can assign arbitrary values to the closed matching curves of rank $a\geq 2$. 
\end{remark}

\section{Further directions} 

In \Cref{subsec:Higgs}, we recall the construction of the Higgs category \cite{Wu21} and show that the $1$-periodic topological Fukaya category $\C_{\bf S}^{\on{c}}$ is equivalent as an exact $\infty$-category to the Higgs category. From this, it follows that the (homotopy category of the) stable category of $\C_{\bf S}^{\on{c}}$ recovers the standard triangulated $2$-Calabi--Yau cluster category of the marked surface ${\bf S}$. In \Cref{subsec:stabcat}, we observe that the mapping class group action on $\C_{\bf S}$ induces a mapping class group action on this triangulated $2$-Calabi--Yau cluster category. Finally, in \Cref{subsec:nclustercats}, we describe the generalized $n$-cluster categories arising from relative higher Ginzburg algebras associated with $n$-angulated marked surfaces.

\subsection{Relation with the Higgs category}\label{subsec:Higgs}

In \cite{Wu21}, Yilin Wu describes a generalization of Amiot's construction of the generalized cluster category, that is different to the approach of this paper and considers a more general setup. Given the input of suitable relative left $3$-Calabi--Yau dg-algebra, the result is a $2$-Calabi--Yau Frobenius extriangulated category equipped with cluster tilting objects. In the following, we recall this construction in the special case where the relative $3$-Calabi--Yau dg-algebra is the relative Ginzburg algebras associated with an ice quiver with potential and show the compatibility with the approach of this paper if the ice quiver with potential comes from a triangulated marked surface.

Let $(Q,F)$ be a finite ice quiver, meaning that $Q$ is a finite quiver and $F\subset Q$ a (not necessarily full) sub-quiver. Let $W$ be a potential for $Q$. Associated with this is the relative Ginzburg algebra $\mathscr{G}(Q,F,W)$ and the Ginzburg dg-functor \cite{Wu21}
\[ {\bf Gi}\colon \Pi_2(kF)\to \mathscr{G}(Q,F,W)\,,\]
where $\Pi_2(kF)$ denotes the $2$-Calabi--Yau completion \cite{Kel11} of the path algebra of $F$. The Ginzburg functor ${\bf Gi}$ is equivalent to the deformed relative Calabi--Yau completion of $kF\to kQ$ and thus admits a left $3$-Calabi--Yau structure \cite{Yeu16,BCS20}. We will assume that the relative Jacobian algebra $H_0(\mathscr{G}(Q,F,W))$ is finite dimensional and refer to \cite{KW23} for a discussion of Higgs categories in the Jacobi-infinite case.  

Pulling back along ${\bf Gi}$ yields the functor 
\[ {\bf Gi}^*\colon \D(\mathscr{G}(Q,F,W))\to \D(\Pi_2(kF))\,.\]
Its left adjoint is denoted ${\bf Gi}_!$ and given by tensoring with the bimodule $\bigoplus_{x\in F_0}P_x$ given by the sum of the indecomposable direct summands of $\mathscr{G}(Q,F,W)$ associated with the frozen vertices. 

We let $\D^{\on{fin}}_F(\mathscr{G}(Q,F,W))\subset \D^{\on{fin}}(\mathscr{G}(Q,F,W))$ be the kernel, i.e.~the fiber, of the functor ${\bf Gi}^*$. The relative cluster category is defined as the Verdier quotient
\[ \C^{\on{rel}}(Q,F,W)\coloneqq \D^{\on{perf}}(\mathscr{G}(Q,F,W))/\D^{\on{fin}}_F(\mathscr{G}(Q,F,W))\,.\]
To define the Higgs category, we must first recall the definition of the relative fundamental domain $\mathcal{F}^{\on{rel}}\subset \D^{\on{perf}}(\mathscr{G}(Q,F,W))$.  Consider the set $\mathcal{P}$ of objects in $\D^{\on{perf}}(\mathscr{G}(Q,F,W))$ given by the indecomposable direct summands of $\mathscr{G}(Q,F,W)$ associated with the frozen vertices of $Q$. Note that $\on{Add}(\mathcal{P})$ is the essential image ${\bf Gi}_!(\on{Add}(kF))$ of $\on{Add}(kF)$ under the functor ${\bf Gi}_!$.

\begin{definition}\label{def:fundamentaldom}
The relative fundamental domain $\mathcal{F}^{\on{rel}}\subset \D^{\on{perf}}(\mathscr{G}(Q,F,W))$ is the full subcategory spanned by objects $X$ satisfying that 
\begin{enumerate}[1)]
\item $X$ fits into a fiber and cofiber sequence $M_1\to M_0\to X$ with $M_0,M_1$ lying in the additive closure $\on{Add}(\mathscr{G}(Q,F,W))$ and
\item $\on{Ext}^i(P,X)\simeq \on{Ext}^i(X,P)\simeq 0$ for all $P\in \mathcal{P}$ and $i>0$.
\end{enumerate}
\end{definition}

\begin{lemma}\label{lem:fundamentaldomain}
Suppose that the functor ${\bf Gi}^*$ admits a colimit preserving right adjoint ${\bf Gi}_*$ satisfying that $\on{Add}(\mathcal{P})={\bf Gi}_*(\on{Add}(\Pi_2(kF)))$. Then condition 2) in \Cref{def:fundamentaldom} is equivalent to $X$ satisfying ${\bf Gi}^*(X)\in \on{Add}(\Pi_2(kF))$. 
\end{lemma}

Note that the condition that ${\bf Gi}_*$ preserves colimits is equivalent to ${\bf Gi}^*$ preserving compact objects, i.e.~perfect modules.

\begin{proof}[Proof of \Cref{lem:fundamentaldomain}]
 Let $X\in \D^{\on{perf}}(\mathscr{G}(Q,F,W))$. Let $P\simeq {\bf Gi}_!(Q)\simeq {\bf Gi}_*(Q')\in \on{Add}(\mathcal{P})$ with $Q,Q'\in \on{Add}(\Pi_2(kF))$. We have 
\[ 
\on{Ext}^i_{\D^{\on{perf}}(\mathscr{G}(Q,F,W))}(P,X)\simeq \on{Ext}^i_{\D^{\on{perf}}(\Pi_2(kF))}(Q,{\bf Gi}^*(X))
\]
and 
\[ 
\on{Ext}^i_{\D^{\on{perf}}(\mathscr{G}(Q,F,W))}(X,P)\simeq \on{Ext}^i_{\D^{\on{perf}}(\Pi_2(kF))}({\bf Gi}^*(X),Q')\,.
\]
Since the dg-algebra $\Pi_2(kF)$ is concentrated in positive degrees (in the homological grading convention), it defines a silting object in $\D^{\on{perf}}(\Pi_2(kF))$. It follows that $\on{Gi}^*(X)\in \on{Add}(\Pi_2(kF))$ holds if and only if 
\[ \on{Ext}^i_{\D^{\on{perf}}(\Pi_2(kF))}(\on{Gi}^*(X),\Pi_2(kF))\simeq \on{Ext}^i_{\D^{\on{perf}}(\Pi_2(kF))}(\Pi_2(kF),\on{Gi}^*(X))\simeq 0\] 
for all $i>0$, see \cite[Thm.~5.5]{MSSS13} or also \cite[Lem.~2.5, Prop.~2.8]{IY18}, \cite[Thm.~10.2.2)]{KN13}. By the above, this is satisfied if and only if $\on{Ext}^i(P,X)\simeq \on{Ext}^i(X,P)\simeq 0$ for all $P\in \mathcal{P}$ and $i>0$.
\end{proof}

The Higgs category $\mathcal{H}(Q,F,W)\subset \C^{\on{rel}}(Q,F,W)$ is defined as the image of $\mathcal{F}^{\on{rel}}$ under the quotient functor $\D^{\on{perf}}(\mathscr{G}(Q,F,W))\to \C^{\on{rel}}(Q,F,W)$. The arising functor $\phi\colon \mathcal{F}^{\on{rel}}\to \mathcal{H}(Q,F,W)$ is furthermore fully faithful on the level of homotopy categories \cite[Prop.~5.20]{Wu21}. The Higgs category is an extension closed subcategory of $\C^{\on{rel}}(Q,F,W)$ and thus inherits the structure of an exact $\infty$-category. The homotopy category $\on{ho}\mathcal{H}(Q,F,W)$ is therefore extriangulated. Furthermore, the image of $\mathscr{G}(Q,F,W)$ in $\mathcal{H}(Q,F,W)$ is a cluster tilting object.

We next discuss under the assumptions of \Cref{lem:fundamentaldomain} the relation of the Higgs category with the generalized cluster category 
\[  \C(Q,F,W)\coloneqq \D^{\on{perf}}(\mathscr{G}(Q,F,W))/\D^{\on{fin}}(\mathscr{G}(Q,F,W))\,.
\]
We equip $\C(Q,F,W)$ with the exact $\infty$-structure obtained from pulling back as in \Cref{ex:splitpb} the split exact structure on $\D^{\on{perf}}(\Pi_2(kF))/\D^{\on{fin}}(\Pi_2(kF))$ along the functor induced by ${\bf Gi}^*$ (note that ${\bf Gi}^*$ automatically preserves finite modules and we have assumed that it preserves perfect modules):
\[
\widetilde{{\bf Gi}}^*\colon \C(Q,F,W)\longrightarrow \D^{\on{perf}}(\Pi_2(kF))/\D^{\on{fin}}(\Pi_2(kF))\,.
\]

Composing the inclusion $\mathcal{H}(Q,F,W)\subset \C^{\on{rel}}(Q,F,W)$ with the quotient functor $\C^{\on{rel}}(Q,F,W)\rightarrow \C(Q,F,W)$ to the generalized cluster category yields the functor 
\[ \tau \colon \mathcal{H}(Q,F,W) \rightarrow \C(Q,F,W)\,.\]

\begin{proposition}\label{prop:tauextr}
Suppose that the functor ${\bf Gi}^*$ admits a colimit preserving right adjoint ${\bf Gi}_*$ satisfying that $\on{Add}(\mathcal{P})={\bf Gi}_*(\on{Add}(\Pi_2(kF)))$. Then the functor $\tau$ is an exact functor between exact $\infty$-categories and thus also gives rise to an extriangulated functor between the extriangulated homotopy categories.
\end{proposition}

\begin{proof}[Proof of \Cref{prop:tauextr}]
Since zero objects, inflations, deflations and exact sequences are detected on the level of the extriangulated homotopy categories, the functor $\tau$ can be lifted to an exact functor between exact $\infty$-categories if and only if $\on{ho}\tau$ can be lifted to an extriangulated functor between the extriangulated homotopy categories. Lifting $\on{ho}\tau\colon \on{ho}\mathcal{H}(Q,F,W)\to \on{ho}\C(Q,F,W)$ to an extriangulated functor amounts to specifying a natural transformation between functors $\on{ho}\mathcal{H}(Q,F,W)^{\on{op}}\times \on{ho}\mathcal{H}(Q,F,W)\to \on{Vect}_k$
\[ \mathbb{E}_{\on{ho}\!\mathcal{H}(Q,F,W)}(\mhyphen,\mhyphen)\longrightarrow  \mathbb{E}_{\on{ho}\!\C(Q,F,W)}(\tau(\mhyphen),\tau(\mhyphen))\,,\]
such that $\tau$ applied to the realization $\mathfrak{s}(\alpha)$ of an extension $\alpha$ yields the realization $\mathfrak{s}(\tau(\alpha))$.

The fiber and cofiber sequence preserving functor $\tau$ induces a natural transformation $\on{Ext}^1_{\mathcal{H}(Q,F,W)}(\mhyphen,\mhyphen)\to \on{Ext}^1_{\C(Q,F,W)}(\tau(\mhyphen),\tau(\mhyphen))$. Since $\mathbb{E}_{\on{ho}\!\mathcal{H}(Q,F,W)}=\on{Ext}^1_{\on{ho}\!\mathcal{H}(Q,F,W)}$ and 
\[ \mathbb{E}_{\on{ho}\!\C(Q,F,W)}=\on{ker}(\widetilde{\bf Gi}^*|_{\on{Ext}^1})\subset \on{Ext}^1_{\on{ho}\!\C(Q,F,W)}\]
it suffices to show that $\widetilde{{\bf Gi}}^* \circ \tau$ vanishes on extensions to obtain the desired natural transformation.

Denote by $(\bar{Q},\bar{W})$ the (non-ice) quiver with potential obtained from $(Q,F,W)$ by removing $F$ from $Q$ and all cycles from $W$ which contain frozen arrows. Let $\mathscr{G}(\bar{Q},\bar{W})$ be the corresponding non-relative Ginzburg algebra and $\mathcal{C}(\bar{Q},\bar{W})=\mathcal{D}^{\on{perf}}(\mathscr{G}(\bar{Q},\bar{W}))/\D^{\on{fin}}(\mathscr{G}(\bar{Q},\bar{W}))$ the cluster category. There is a cofiber sequence of compactly generated stable $\infty$-categories, see \cite[Prop.~7.8]{Wu21},
\[ \D^{\on{perf}}(\Pi_2(F))\longrightarrow \D^{\on{perf}}(\mathscr{G}(Q,F,W))\xlongrightarrow{\pi} \D^{\on{perf}}(\mathscr{G}(\bar{Q},\bar{W}))\,.\]
The functor $\pi$ induces the functor $\bar{\pi}\colon \C^{\on{rel}}(Q,F,W)\to \C(\bar{Q},\bar{W})$.

Let $X_1,X_2\in \mathcal{H}(Q,F,W)$ and $Y_1,Y_2\in \mathcal{F}^{\on{rel}}$ denote the images of $X_1,X_2$ under the equivalence $\on{ho}\mathcal{F}^{\on{rel}}\simeq \on{ho}\mathcal{H}(Q,F,W)$. Let further $Y_1'=\tau_{\leq -1}^{\on{rel}}(Y_1)$ be the relative truncation of $Y_1$, see \cite[Def.~4.9]{Wu21}. There is a morphism $Y_1'\to Y_1$, whose cofiber lies in $\D^{\on{fin}}_F(\mathscr{G}(Q,F,W))$, see loc.~cit.. The image of the morphism in $\C^{\on{rel}}(Q,F,W)$ is thus equivalence. We have the following equivalences,
\begin{align*} 
\on{Ext}^1_{\mathcal{H}(Q,F,W)}(X_1,X_2)&\simeq \on{Ext}^1_{\C(\bar{Q},\bar{W})}(\bar{\pi}(X_1),\bar{\pi}(X_2))\\
& \simeq \on{Ext}^1_{\D(\mathscr{G}(\bar{Q},\bar{W}))}(\pi(Y_1'),\pi(Y_1))\\
& \simeq \on{Ext}^1_{\D(\mathscr{G}(Q,F,W)}(Y_1',Y_2)
\end{align*}
where the first equivalence is \cite[Prop.~5.36]{Wu21}, the second equivalence is shown in step 2 in the proof of Proposition 2.12 in \cite{Ami09}, using that $\pi(\tau^{\on{rel}}_{\leq -1}(Y_1))\simeq \tau_{\leq -1}\pi(Y_1)$, and the last equivalence is shown in \cite[Lem.~5.16, Lem.~5.29]{Wu21}.

Any extension $\alpha\colon X_1\to X_2[1]$ in $\mathcal{H}(Q,F,W)$ thus lifts to an extension $\alpha'\colon Y_1'\to Y_2[1]$ in $\mathcal{D}(\mathscr{G}(Q,F,W))$. We note that $\widetilde{\bf Gi}^*\circ \tau(\alpha)$ is equivalent to the image under the quotient functor $\D^{\on{perf}}(\Pi_2(kF))\to \D^{\on{perf}}(\Pi_2(kF))/\D^{\on{fin}}(\Pi_2(kF))$ of ${\bf Gi}^*(\alpha')$. Since ${\bf Gi}^*(Y_1'),{\bf Gi}^*(Y_2)\in \on{Add}(\Pi_2(kF))$ and $\Pi_2(kF)$ is rigid, we find that ${\bf Gi}^*(\alpha')$ must vanish. It follows that $\widetilde{\bf Gi}^*\circ \tau(\alpha)\simeq 0$, as desired. 
\end{proof}

The functor $\tau$ considered in \Cref{prop:tauextr} is not for every choice of ice quiver with potential an equivalence of extriangulated categories and the generalized cluster category $\C(Q,F,W)$ does not necessarily categorify the corresponding cluster algebra with coefficients. This is because $\C(Q,F,W)$ may vanish for an arbitrary choice of ice quiver with potential. This is the case for instance if $\mathscr{G}(Q,F,W)$ is Jacobi-finite and concentrated in degree $0$, see \cite{Wu21} for examples. For a basic example where $\C(Q,F,W)$ vanishes, see \Cref{ex:rela1}.

Finally, we consider the specific choice of ice quiver with potential $(Q,F,W)=(Q_\mathcal{T},F_\mathcal{T},W_\mathcal{T})$ associated with a marked surface ${\bf S}$ with an ideal triangulation $I$ dual to a trivalent spanning graph $\mathcal{T}$.  The vertices of $Q_\mathcal{T}$ are given by the arcs in $I$ and the arrows are obtained by inscribing clockwise $3$-cycles into the faces of the ideal triangulation. The frozen part $F_\mathcal{T}$ consists of the vertices given by boundary arcs. The potential $W_\mathcal{T}$ is the sum of the $3$-cycles inscribed into the faces. The corresponding non-frozen quiver with potential $(\bar{Q},\bar{W})$ was defined in \cite{Lab09}. We note that there is an isomorphisms of dg-algebras $\mathscr{G}(Q_\mathcal{T},F_\mathcal{T},W_\mathcal{T})\simeq \mathscr{G}_\mathcal{T}$ and thus an equivalence $\C(Q,F,W)\simeq \C_{\bf S}^{\on{c}}$ with the subcategory of compact objects of the $(\on{Ind}$-complete) $1$-periodic topological Fukaya $\C_{\bf S}$ of ${\bf S}$.

\begin{theorem}\label{thm:equivHiggs}
The functor 
\[ \tau \colon \mathcal{H}(Q_\mathcal{T},F_\mathcal{T},W_\mathcal{T}) \rightarrow \C_{\bf S}^{\on{c}}\]
is an equivalence of exact $\infty$-categories. In particular, it induces an equivalence between the extriangulated homotopy categories.
\end{theorem}

\begin{proof}
The conditions of \Cref{prop:tauextr} are satisfied by \Cref{rem:rightadjointGi} below. We first show that $\on{ho}\tau$ is an equivalence of extriangulated homotopy categories using \Cref{lem:extriangulatedequivalence} below.

Under the quotient functors $\D(\mathscr{G}_\mathcal{T})\to \C^{\on{rel}}(Q_\mathcal{T},F_\mathcal{T},W_\mathcal{T}),\C_{\bf S}$, the relative Ginzburg algebra is mapped to cluster tilting objects $T\in \mathcal{H}(Q_\mathcal{T},F_\mathcal{T},W_\mathcal{T})$ and $T'\simeq \bigoplus_{\gamma \in I}M_\gamma\in \C_{\bf S}$, respectively. The Homs assemble into an apparent commutative diagram, where $\gamma$ is induced by the functor $\tau$.
\[
\begin{tikzcd}
{\on{Ext}^0_{\mathcal{D}(\mathscr{G}_\mathcal{T})}(\mathscr{G}_\mathcal{T},\mathscr{G}_\mathcal{T})} \arrow[rr, "\beta"] \arrow[rd, "\alpha"] &                                     & {\on{Ext}^0_{\mathcal{H}(Q_\mathcal{T},F_\mathcal{T},W_\mathcal{T})}(T,T)} \arrow[ld, "\gamma"] \\
                                                                                                                                              & {\on{Ext}^0_{\C_{\bf S}}(T',T')} &                                                       
\end{tikzcd}
\]
The morphism $\beta$ is an equivalence by the fully faithfulness of the functor $\phi\colon \mathcal{F}^{\on{rel}}\to \mathcal{H}(Q,F,W)$. Using the geometric description of the endomorphisms of $\mathscr{G}_\mathcal{T}$ in \cite[Thm.~6.4,6.5]{Chr21b} and of $T'$ in \Cref{thm:hom1,thm:hom2}, it is straightforward to check that $\alpha$ is also an equivalence. This implies that the map $\gamma$ is also an equivalence. 

Next, we show that $\tau$ is essentially surjective. Consider an indecomposable object $X\in \C_{\bf S}^{\on{c}}$. By \Cref{thm:geom}, we have an equivalence $X\simeq M_{\gamma}$ for a matching curve $\gamma$ in ${\bf S}$. The matching curve $\gamma$ lifts uniquely to a pure matching datum $(\gamma, k[t_1])$ in ${\bf S}\backslash M$ in the sense of \cite[Def.~5.5]{Chr21b}, see \cite[Lem.~5.8]{Chr21b}. The corresponding object $M_{\gamma}^{k[t_1]}\in \mathcal{D}(\mathscr{G}_\mathcal{T})$ lies in $\mathcal{F}^{\on{rel}}$ and its image in $\mathcal{H}(Q_{\mathcal{T}},F_\mathcal{T},W_{\mathcal{T}})$ is the desired lift of $M_{\gamma}$.

We have shown that $\on{ho}\tau$ is an equivalence on extriangulated homotopy categories. The functor $\tau$ is thus also an exact functor between exact $\infty$-categories and further gives bijections between the subsets of morphisms given by inflations and deflations. To prove that $\tau$ is an equivalence of exact $\infty$-categories, it remains to show that $\tau$ is fully faithful, i.e.~gives rise to equivalences between the mapping spaces, or equivalently, on all negative extensions groups. Consider two objects $X,X'\in \mathcal{F}^{\on{rel}}$ with images $Y,Y'\in \mathcal{H}(Q_\mathcal{T},F_\mathcal{T},W_\mathcal{T})$. By the equivalence of homotopy categories $\on{ho}\mathcal{F}^{\on{rel}}\simeq \on{ho}\C_{\bf S}^{\on{c}}$, there exist equivalences $X\simeq M_{\gamma}^{k[t_1]}$ and $X'\simeq M_{\gamma'}^{k[t_1]}$ for two pure matching data $(\gamma,k[t_1])$, $(\gamma',k[t_1])$ in ${\bf S}\backslash M$. Note that $\tau(Y)\simeq M_{\gamma}$ and $\tau(Y')\simeq M_{\gamma'}$. The negative extension groups organize into the following commutative diagram,
\[
\begin{tikzcd}
{\on{Ext}^{-i}_{\mathcal{D}(\mathscr{G}_\mathcal{T})}(X,X')} \arrow[rd, "\simeq"] \arrow[rr, "\simeq"] &                                                              & {\on{Ext}^{-i}_{\mathcal{C}^{\on{rel}}(Q_\mathcal{T},F_\mathcal{T},W_\mathcal{T})}(Y,Y')} \arrow[ld] \\
                                                                                                       & {\on{Ext}^{-i}_{\mathcal{C}_{\bf S}}(M_{\gamma},M_{\gamma'})} &                                                                                          
\end{tikzcd}
\]
where the left morphism is an isomorphism by the geometric description of the extension groups in Theorems 6.4 and 6.5 in \cite{Chr21b}, using that $\on{H}_i(k[t_1])\simeq H_i(k[t_1^\pm])\simeq k$ for all $i>0$. The upper morphism is also an isomorphism, as shown in \cite[Lem.~6.70]{ChenThesis}. This shows the desired fully faithfulness.
\end{proof}

\begin{remark}\label{rem:rightadjointGi}
The conditions of \Cref{prop:tauextr} are satisfied for the relative Ginzburg algebra $\mathscr{G}_\mathcal{T}$, which can be seen as follows. The functor ${\bf Gi}_!$ is described geometrically in \cite[Prop~5.19]{Chr21b}. A minor variation of that argument (where one glues the left adjoints instead of the right adjoints) yields a similar description of the functor ${\bf Gi}_*$. In the notation of loc.~cit., we find for every edge $e$ of the trivalent spanning graph $\mathcal{T}$ that $M^{k[t_1]}_{c_e'} \simeq \on{ev}^!(k[t_1])$, with $\on{ev}^!$ the right adjoint of $\on{ev}_e$ (instead of the left adjoint), and $c_e'$ the pure matching curve consisting of segments which each wrap exactly one step clockwise (instead of counterclockwise) around a vertex of $\mathcal{T}$. If $e$ is an external edge, then $c_e'$ is a boundary arc, and all boundary arcs arise this way. This relation between the left and right adjoints of ${\bf Gi}^*$ can also be explained by noting that ${\bf Gi}^*$ is a spherical functor whose cotwist functor permutes the components of $\D^{\on{perf}}(\Pi_2(kF))$.
\end{remark}

\begin{lemma}\label{lem:extriangulatedequivalence}
Let $\tau\colon (C,\mathbb{E},\mathfrak{s})\to (C',\mathbb{E}',\mathfrak{s}')$ be an extriangulated functor between extriangulated categories, such that $C$ is Frobenius with finite dimensional Homs and extriangulated $2$-Calabi--Yau. Assume further that 
\begin{enumerate}[1)]
\item there are cluster tilting objects $T\in C$ and $T'\in C'$,
\item $\tau(T)=T'$,
\item $\tau\colon \on{Hom}_C(T,T)\to \on{Hom}_{C'}(T',T')$ is an isomorphism and
\item the functor $\tau$ is essentially surjective.
\end{enumerate}
Then $\tau$ is an extriangulated equivalence.
\end{lemma}

\begin{remark}
A version of \Cref{lem:extriangulatedequivalence} for triangulated categories (without the essential surjectivity assumption) was proven in section 4.5 of \cite{KR08} 
\end{remark}

\begin{proof}[Proof of \Cref{lem:extriangulatedequivalence}.]
Let $X\in C$ and $Y\in \on{Add}(T)$. By \Cref{lem:2termresolution}, there exist exact sequences $T_1\to T_0 \to X$ and $X\to S_0 \to S_1$ in $C$ with $T_0,T_1,S_0,S_1\in \on{Add}(T)$. We get two morphisms between sequences of vector spaces which are exact by \cite[Thm.~3.5]{GNP21}:
\[
\begin{tikzcd}[column sep=tiny]
{\on{Hom}(Y,T_1)} \arrow[d, "\simeq"] \arrow[r] & {\on{Hom}(Y,T_0)} \arrow[d, "\simeq"] \arrow[r] & {\on{Hom}(Y,X)} \arrow[d] \arrow[r]   & {\mathbb{E}(Y,T_1)} \arrow[d, "\simeq"] \arrow[r]    & {\mathbb{E}(Y,T_0)} \arrow[d, "\simeq"]    \\
{\on{Hom}(\tau(Y),\tau(T_1))} \arrow[r]         & {\on{Hom}(\tau(Y),\tau(T_0))} \arrow[r]         & {\on{Hom}(\tau(Y),\tau(X))} \arrow[r] & {\mathbb{E}(\tau(Y),\tau(T_1))} \arrow[r] & {\mathbb{E}(\tau(Y),\tau(T_0))}
\end{tikzcd}
\]
\[
\begin{tikzcd}[column sep=tiny]
{\on{Hom}(S_1,Y)} \arrow[d, "\simeq"] \arrow[r] & {\on{Hom}(S_0,Y)} \arrow[d, "\simeq"] \arrow[r] & {\on{Hom}(X,Y)} \arrow[d] \arrow[r]   & {\mathbb{E}(S_1,Y)} \arrow[d, "\simeq"] \arrow[r]    & {\mathbb{E}(S_0,Y)} \arrow[d, "\simeq"]    \\
{\on{Hom}(\tau(S_1),\tau(Y))} \arrow[r]         & {\on{Hom}(\tau(S_0),\tau(Y))} \arrow[r]         & {\on{Hom}(\tau(X),\tau(Y))} \arrow[r] & {\mathbb{E}(\tau(S_1),\tau(Y))} \arrow[r] & {\mathbb{E}(\tau(S_0),\tau(Y))}
\end{tikzcd}
\]
All vertical arrows, except the middle arrows, are equivalences by assumption 3) and the fact that $T$ is rigid. By the five lemma, the middle vertical arrows above are also equivalences.

Similar arguments apply to the following two diagrams:
\[
\begin{tikzcd}[column sep=tiny]
{\on{Hom}(T_0,Y)} \arrow[d, "\simeq"] \arrow[r] & { \on{Hom}(T_1,Y)} \arrow[d, "\simeq"] \arrow[r] & {\mathbb{E}(X,Y)} \arrow[d] \arrow[r]   & {\mathbb{E}(T_0,Y)} \arrow[d, "\simeq"] \arrow[r] & {\mathbb{E}(T_1,Y)} \arrow[d, "\simeq"] \\
{\on{Hom}(\tau(T_0),\tau(Y))} \arrow[r]         & { \on{Hom}(\tau(T_1),\tau(Y))} \arrow[r]         & {\mathbb{E}(\tau(X),\tau(Y))} \arrow[r] & {\mathbb{E}(\tau(T_0),\tau(Y))} \arrow[r]         & {\mathbb{E}(\tau(T_1),\tau(Y))}        
\end{tikzcd}
\]
\[
\begin{tikzcd}[column sep=tiny]
{\on{Hom}(Y,S_0)} \arrow[d, "\simeq"] \arrow[r] & { \on{Hom}(Y,S_1)} \arrow[d, "\simeq"] \arrow[r] & {\mathbb{E}(Y,X)} \arrow[d] \arrow[r]   & {\mathbb{E}(Y,S_0)} \arrow[d, "\simeq"] \arrow[r] & {\mathbb{E}(Y,S_1)} \arrow[d, "\simeq"] \\
{\on{Hom}(\tau(Y),\tau(S_0))} \arrow[r]         & { \on{Hom}(\tau(Y),\tau(S_1))} \arrow[r]         & {\mathbb{E}(\tau(Y),\tau(X))} \arrow[r] & {\mathbb{E}(\tau(T),\tau(S_0))} \arrow[r]         & {\mathbb{E}(\tau(T),\tau(S_1))}        
\end{tikzcd}
\]

We use the above constructed equivalences to show that $\tau$ is fully faithful on Homs and exact extensions. Let $X,X'\in C$. By \Cref{lem:2termresolution}, there exist exact sequences $T_1\to T_0 \to X$ and $T_1'\to T_0\to X'$ in $C$ with $T_0,T_1,T_0',T_1'\in \on{Add}(T)$. We form the following diagrams with horizontal exact sequences:
\[
\begin{tikzcd}[column sep=tiny]
{\on{Hom}(X',T_1)} \arrow[d, "\simeq"] \arrow[r] & {\on{Hom}(X',T_0)} \arrow[d, "\simeq"] \arrow[r] & {\on{Hom}(X',X)} \arrow[d] \arrow[r]   & {\mathbb{E}(X',T_1)} \arrow[d, "\simeq"] \arrow[r]    & {\mathbb{E}(X',T_0)} \arrow[d, "\simeq"]    \\
{\on{Hom}(\tau(X'),\tau(T_1))} \arrow[r]         & {\on{Hom}(\tau(X'),\tau(T_0))} \arrow[r]         & {\on{Hom}(\tau(X'),\tau(X))} \arrow[r] & {\mathbb{E}(\tau(X'),\tau(T_1))} \arrow[r] & {\mathbb{E}(\tau(X'),\tau(T_0))}
\end{tikzcd}
\]

\[
\begin{tikzcd}[column sep=tiny]
{\on{Hom}(T_0',X)} \arrow[d, "\simeq"] \arrow[r] & { \on{Hom}(T_1',X)} \arrow[d, "\simeq"] \arrow[r] & {\mathbb{E}(X',X)} \arrow[d] \arrow[r]   & {\mathbb{E}(T_0',X)} \arrow[d, "\simeq"] \arrow[r] & {\mathbb{E}(T_1',X)} \arrow[d, "\simeq"] \\
{\on{Hom}(\tau(T_0'),\tau(X))} \arrow[r]         & { \on{Hom}(\tau(T_1'),\tau(X))} \arrow[r]         & {\mathbb{E}(\tau(X'),\tau(X))} \arrow[r] & {\mathbb{E}(\tau(T_0'),\tau(X))} \arrow[r]         & {\mathbb{E}(\tau(T_1'),\tau(X))}        
\end{tikzcd}
\]
The above shows that the non-central vertical morphisms are equivalences, and the five lemma yields that the central morphisms are also isomorphisms, showing the desired fully faithfulness of $\tau$.
\end{proof}

\begin{example}\label{ex:rela1}
Consider the $A_2$-quiver with a frozen vertex:
\[ 
\begin{tikzcd}
1 \arrow[r, "a"] & \begin{color}{cyan}2\end{color}
\end{tikzcd}
\]
The corresponding relative Ginzburg algebra $\mathscr{G}$ is the path algebra of the graded quiver
\[
 \tilde{Q}\quad =\quad \begin{tikzcd}
1 \arrow[r, "a", bend left] \arrow["l"', loop, distance=2em, in=215, out=145] & \begin{color}{cyan}2\end{color} \arrow[l, "a^*", bend left]
\end{tikzcd}
\]
with $|a|=0,\,|a^*|=1,\,|l|=2$ and differential determined on the generators via $d(l)=a^*a$ and $d(a)=d(a^*)=0$. We show in the following that the homology of $\mathscr{G}$ is generated by $e_1,e_2,a,b$, with $e_i$ the lazy path at $i$, and thus finite dimensional. It follows that $\mathcal{D}^{\on{perf}}(\mathscr{G})=\mathcal{D}^{\on{fin}}(\mathscr{G})$ and that the associated generalized cluster category $\mathcal{D}^{\on{perf}}(\mathscr{G})/\mathcal{D}^{\on{fin}}(\mathscr{G})$ vanishes.
 
Let $X\in \mathscr{G}$ be a cycle, meaning that it satisfies $d(X)=0$. We may split $X=\sum_{i,j=1}^2 X_{i,j}$, with $X_{i,j}$ a cycle consisting of paths beginning at the vertex $i$ and ending at the vertex $j$. We can write $X$ as a $k$-linear sum of the generators, which are the paths of $\tilde{Q}$. We may thus write $X_{1,1}$ as $X_{1,1}=lY^1+a^*aY^2+\lambda e_1$, with $Y^{1},Y^2\in \mathscr{G}$ and $\lambda \in k$. Evaluating the condition $d(X_{1,1})=0$, we find $a^*aY^1-a^*ad(Y^2)=0$, hence $Y_1=d(Y^2)$, and thus get that $X_{1,1}=d(lY^2)+\lambda e_1$ is homologous to $\lambda e_i$. Similarly, we may write $X_{1,2}=aY$ for $Y$ some cycle consisting of paths beginning and ending at $1$. By the previous argument $Y$ is homologous to $e_1$, and $X_{1,2}$ is hence homologous to $a$. Analogous arguments show that $X_{2,1}$ and $X_{2,2}$ are homologous to $a^*$ or $e_2$.
\end{example}

\subsection{The stable category of \texorpdfstring{$\mathcal{C}_{\bf S}$}{the generalized cluster category} and the mapping class group action}\label{subsec:stabcat}

Let ${\bf S}$ be a marked surface and $\mathcal{C}_{\bf S}^{\on{c}}$ the associated (not $\on{Ind}$-complete) generalized cluster category, considered as a Frobenius exact $\infty$-category. The associated stable category $\bar{\mathcal{C}}^{\on{c}}_{\bf S}$ is a stable $\infty$-category whose triangulated homotopy $1$-category $\on{ho}\bar{\mathcal{C}}^{\on{c}}_{\bf S}$ is equivalent to the stable category of the Frobenius extriangulated category $\on{ho}\mathcal{C}^{\on{c}}_{\bf S}$. By \Cref{thm:equivHiggs}, $\on{ho}\mathcal{C}_{\bf S}^{\on{c}}$ is equivalent as an extriangulated category to the  Higgs category. As shown in \cite[Cor.~5.19]{Wu21}, the stable category $\on{ho}\bar{\mathcal{C}}^{\on{c}}_{\bf S}$ is thus equivalent to the usual triangulated $2$-Calabi--Yau cluster category associated with ${\bf S}$, considered for instance in \cite{BZ11}. 

We further note that the mapping class group action on $\mathcal{C}_{\bf S}$ induces an action on the stable $\infty$-category $\bar{\mathcal{C}}^{\on{c}}_{\bf S}$. 

\begin{proposition}\label{prop:MCGonclustercat}
The mapping class group action on $\mathcal{C}^{\on{c}}_{\bf S}$ from \Cref{thm:mcgact} induces an action on $\bar{\mathcal{C}}^{\on{c}}_{\bf S}$ by automorphisms in $\on{ho}\on{St}$, and thus also an action on $\on{ho}\bar{\mathcal{C}}_{\bf S}$ by equivalences of triangulated categories.
\end{proposition}

\begin{proof}
The action of the mapping class group does not affect the values of the global sections at the external edges of the auxiliary trivalent spanning graph $\mathcal{T}$. It follows that the action preserves the set $W$ from \Cref{prop:jkpw} and hence induces an action on the localization. 
\end{proof}

\subsection{Generalized \texorpdfstring{$n$}{n}-cluster categories of marked surfaces}\label{subsec:nclustercats}

The study of $2$-Calabi--Yau cluster categories admits a generalization to triangulated $n$-Calabi--Yau categories with $n\geq 2$, called $n$-cluster categories. Their representation theoretic behavior is very similar to that of the $2$-cluster categories. For example, there is a notion of $n$-cluster tilting objects which admit mutations similar to the $2$-Calabi--Yau case. $n$-Cluster categories for $n\geq 2$ however do not categorify cluster algebras. They can be constructed in similar ways as the $2$-cluster categories, for example as orbit categories or explicit geometric constructions. Amiot's construction of the generalized cluster category can also be used to construct $n$-Calabi--Yau categories with $n$-cluster tilting objects, see \cite{Guo11}. A possible input for this is a higher Ginzburg algebra, meaning an $(n+1)$-Calabi--Yau version of the usual Ginzburg algebra.

We fix a marked surface ${\bf S}$ and choose an $n$-valent spanning graph $\mathcal{T}$ of ${\bf S}$. Its dual describes an $n$-angulation (decomposition into $n$-gons) of the surface. As shown in \cite{Chr21b}, there is an associated relative higher Ginzburg algebra $\mathscr{G}_\mathcal{T}$ and a $\mathcal{T}$-parametrized perverse schober $\mathcal{F}_\mathcal{T}$ with global sections its derived $\infty$-category $\mathcal{D}(\mathscr{G}_\mathcal{T})$. It's generalized cluster category is given by the Verdier quotient 
\[ \C_{\bf S}^n\coloneqq \mathcal{D}(\mathscr{G}_\mathcal{T})/\on{Ind}\mathcal{D}^{\on{fin}}(\mathscr{G}_\mathcal{T})\,,\]
we refer to $\C_{\bf S}^n$ as the generalized $n$-cluster category for clarity. In the following, we survey how the results of this article generalize to $\mathcal{C}_{\bf S}^n$.

Locally at each vertex of $\mathcal{T}$, the spherical adjunction underlying the perverse schober $\mathcal{F}_\mathcal{T}$ is given by
\[ f^*\colon \mathcal{D}(k)\longleftrightarrow \on{Fun}(S^{n},\mathcal{D}(k))\noloc f_*\,,\]
where $f$ denotes the map $S^n\rightarrow \ast$. We now proceed in analogy with \Cref{sec4}. Under the equivalence of $\infty$-categories $\on{Fun}(S^{n},\mathcal{D}(k))\simeq \mathcal{D}(k[t_{n-1}])$, the functor $f^*$ is identified with the pullback functor $\phi^*$ along $\phi\colon k[t_{n-1}]\xrightarrow{t_{n-1}\mapsto 0} k$. By \Cref{lem:laurent}, the quotient $\mathcal{D}(k[t_{n-1}])/\on{Ind}\mathcal{D}(k[t_{n-1}])^{\on{fin}}$ is equivalent to the derived $\infty$-category $\mathcal{D}(k[t_{n-1}^\pm])$ of $(n-1)$-periodic chain complexes.

\begin{theorem}
There exist a vanishing-monadic and nearby-monadic $\mathcal{T}$-parametrized perverse schober $\mathcal{F}_\mathcal{T}^{\on{mnd}}$ and a locally constant $\mathcal{T}$-parametrized perverse schober $\mathcal{F}_{\mathcal{T}}^{\on{clst}}$ with generic stalk $\mathcal{D}(k[t_{n-1}^\pm])$ satisfying the following.
\begin{enumerate}[i)]
\item There is a semiorthogononal decomposition $\{\mathcal{F}_{\mathcal{T}}^{\on{clst}},\mathcal{F}_{\mathcal{T}}^{\on{mnd}}\}$ of $\mathcal{F}_\mathcal{T}$.
\item There exists an equivalence \[ \mathcal{H}(\mathcal{T},\mathcal{F}_{\mathcal{T}}^{\on{mnd}})\simeq \on{Ind}\mathcal{D}^{\on{fin}}(\mathscr{G}_\mathcal{T})\,.\]
\item There exists an equivalence 
\[\mathcal{C}_{\bf S}^{n}\coloneqq \mathcal{H}(\mathcal{T},\mathcal{F}_{\mathcal{T}}^{\on{clst}})\simeq \mathcal{D}(\mathscr{G}_\mathcal{T})/\on{Ind}\mathcal{D}^{\on{fin}}(\mathscr{G}_\mathcal{T})\,. \]
\end{enumerate}
\end{theorem}

\begin{proof}
Analogous to the proofs of \Cref{prop:sodG} and \Cref{thm:clstcat}.
\end{proof}

We thus find the generalized $n$-cluster category $\mathcal{C}_{\bf S}^n$ to be equivalent the topological Fukaya category of ${\bf S}$ with coefficients in the derived $\infty$-category of $(n-1)$-periodic chain complexes. In particular, the well-known $2$-periodic topological Fukaya category of a surface may thus be seen as the generalized $3$-cluster category of the surface.

The $\infty$-category $\mathcal{C}_{\bf S}^{n}$ is smooth and proper as an $(n-1)$-periodic, or $2(n-1)$-periodic if $n$ is even, stable $\infty$-category. The functor 
\[\prod_{e\in \mathcal{T}_1^\partial}\on{ev}_e\colon \mathcal{C}_{\bf S}^n=\mathcal{H}(\mathcal{T},\mathcal{F}_{\mathcal{T}}^{\on{clst}})\longrightarrow \prod_{e\in \mathcal{T}_1^\partial}\mathcal{F}_{\mathcal{T}}^{\on{clst}}(e)\simeq \prod_{e\in \mathcal{T}_1^\partial}\mathcal{D}(k[t_{n-1}^\pm])\]
further admits an $n$-Calabi--Yau structure, see \cite[Thm.~6.1]{Chr23}.

A geometric model for $\mathcal{C}_{\bf S}^n$, describing (a priori a subset of) objects in $\mathcal{C}_{\bf S}^n$ in terms of matching data with local value $k[t_{n-1}^\pm][i]$, $0\leq i\leq n-2$, is given in \cite{Chr21b}. The main difference to the geometric model for $\mathcal{C}_{\bf S}^n$ is thus that we may associate objects to matching curves equipped with a grading datum in $\mathbb{Z}/(n-1)\mathbb{Z}$. Note that \cite{Chr21b} only describes some of the morphism objects between the associated global sections (because the matching curves are not necessarily pure). One can prove analogs of \Cref{thm:hom1,thm:hom2} using a variation of the approach in \cite{Chr21b} applied directly to $\C_{\bf S}^n$. Furthermore, the proof of the geometrization \Cref{thm:geom} applies with minor modifications also to $\mathcal{C}_{\bf S}^n$, showing that all compact objects in $\mathcal{C}_{\bf S}^{n}$ arise from collections of graded matching curves in ${\bf S}\backslash M$.
 
As in \Cref{sec6.2}, we can use the functor $\prod_{e\in \mathcal{T}_1^\partial}\on{ev}_e$ to define an exact $\infty$-structure on $\mathcal{C}_{\bf S}^{n,\on{c}}$ and hence also an extriangulated structure on the homotopy category $\on{ho}\mathcal{C}^{n,\on{c}}_{\bf S}$. This exact $\infty$-structure is again Frobenius.

\bibliography{biblio} 
\bibliographystyle{alpha}

\textsc{Université Paris Cité and Sorbonne Université, CNRS, IMJ-PRG, F-75013 Paris, France.}

\textit{Email address:} \texttt{merlin.christ@imj-prg.fr}

\end{document}